\documentclass[11pt,letterpaper]{amsart}

\usepackage{letltxmacro}
\usepackage{amsmath}
\usepackage{amssymb}
\usepackage{graphicx}
\usepackage{amsthm}
\usepackage{pifont}
\usepackage{eucal}
\usepackage{mathrsfs}
\usepackage{color}
\usepackage{enumitem}
\usepackage{hyperref}
\usepackage{siunitx}
\usepackage{cancel}
\usepackage[normalem]{ulem}
\usepackage[all]{xy}
\usepackage{tikz-cd}
\usepackage{nameref}

\usepackage{mathalfa}
\usepackage[cbgreek]{textgreek}
\DeclareFontFamily{OT1}{pzc}{}
\DeclareFontShape{OT1}{pzc}{m}{it}{<-> s * [1.20] pzcmi7t}{}
\DeclareMathAlphabet{\mathpzc}{OT1}{pzc}{m}{it}

\newtheorem{theorem}{Theorem}
\newtheorem*{theorem*}{Theorem}
\newtheorem{lemma}[theorem]{Lemma}
\newtheorem{proposition}[theorem]{Proposition}
\newtheorem{corollary}[theorem]{Corollary}
\newtheorem{definition}[theorem]{Definition}
\newcommand\at[2]{\left.#1\right|_{#2}}

\newenvironment{PROOF}[1]{{\flushleft {\bfseries Proof of #1}:}}{\hfill\ensuremath{\Box}}

\theoremstyle{definition}

\theoremstyle{remark}

\makeindex

\newcommand{\brk}[1]{\left(#1\right)}  
\newcommand{\BRK}[1]{\left\{#1\right\}}

\newcommand{\bra}{\langle}
\newcommand{\ket}{\rangle}

\newcommand{\appref}[1]{Appendix~\ref{#1}}
\newcommand{\secref}[1]{Section~\ref{#1}}

\newcommand{\thmref}[1]{Theorem~\ref{#1}}
\newcommand{\defref}[1]{Definition~\ref{#1}}

\newcommand{\propref}[1]{Proposition~\ref{#1}}
\newcommand{\lemref}[1]{Lemma~\ref{#1}}
\newcommand{\corrref}[1]{Corollary~\ref{#1}}
\LetLtxMacro{\originaleqref}{\eqref}
\renewcommand{\eqref}{\originaleqref}
\newcommand{\beq}{\begin{equation}}
\newcommand{\eeq}{\end{equation}}

\newcommand{\frakG}{\mathfrak{G}}
\newcommand{\frakS}{\mathfrak{S}}
\newcommand{\frakT}{\mathfrak{T}}
\newcommand{\frakX}{\mathfrak{X}}

\newcommand{\frakn}{\mathfrak{n}}
\newcommand{\frakt}{\mathfrak{t}}

\newcommand{\frakI}{\mathfrak{I}}

\newcommand{\bbE}{{\mathbb E}}
\newcommand{\bbF}{{\mathbb F}}
\newcommand{\bbG}{{\mathbb G}}
\newcommand{\bbJ}{{\mathbb J}}

\newcommand{\bbN}{{\mathbb N}}
\newcommand{\bbP}{{\mathbb P}}
\newcommand{\bbR}{{\mathbb R}}

\newcommand{\bbZ}{{\mathbb Z}}

\newcommand{\scrC}{\mathscr{C}}

\newcommand{\scrE}{\mathscr{E}}
\newcommand{\scrH}{\mathscr{H}}
\newcommand{\scrN}{\mathscr{N}}
\newcommand{\scrR}{\mathscr{R}}

\newcommand{\scrM}{\mathscr{M}}

\newcommand{\calG}{{\mathcal{G}}}
\newcommand{\calL}{{\mathcal{L}}}

\newcommand{\calR}{{\mathcal{R}}}

\newcommand{\module}{\scrH}

\newcommand{\bA}{\mathpzc{A}}

\newcommand{\bD}{A}

\newcommand{\fbD}{\mathpzc{A}}

\newcommand{\bB}{B}
\newcommand{\bC}{\mathpzc{C}}
\newcommand{\bP}{\mathpzc{G}}
\newcommand{\tbP}{\mathpzc{P}}
\newcommand{\tGamma}{\Gamma}
\newcommand{\ttGamma}{\tilde{\Gamma}}

\usepackage{ifthen}

\usepackage{ifthen}

\usepackage{ifthen}

\newcommand{\Dng}[1]{\varrho_{g}^{#1}}

\newcommand{\trace}{\mathrm{tr}}
\renewcommand{\ker}{\operatorname{ker}}
\newcommand{\image}{\operatorname{im}}

\newcommand{\Rm}{\mathrm{Rm}}

\newcommand{\ord}{\operatorname{ord}}

\newcommand{\End}{{\operatorname{End}}}
\newcommand{\id}{{\operatorname{Id}}}
\newcommand{\Ah}{{\operatorname{A}}}

\newcommand{\Ein}{\mathrm{Ein}}
\newcommand{\Ric}{\mathrm{Ric}}

\newcommand{\textand}{\quad\text{ and }\quad}
\newcommand{\Textand}{\qquad\text{ and }\qquad}

\newcommand{\overbar}[1]{\mkern 1.5mu\overline{\mkern-1.5mu#1\mkern-1.5mu}\mkern 1.5mu}

\newcommand{\dM}{{\partial M}}
\newcommand{\Nzero}{{\bbN_0}}

\newcommand{\E}{\bbE}

\newcommand{\g}{g}

\newcommand{\PnD}{\bbP^{\frakn}}
\newcommand{\PttD}{\bbP^{\frakt\frakt}}
\newcommand{\PtnD}{\bbP^{\frakt\frakn}}
\newcommand{\PntD}{\bbP^{\frakn\frakt}}
\newcommand{\PnnD}{\bbP^{\frakn\frakn}}
\newcommand{\nabg}{\nabla^{g}}

\newcommand{\gD}{\gamma}
\newcommand{\OP}{\mathrm{OP}}

\newcommand{\Hg}{H_{g}}

\newcommand{\dg}{d_{g}}

\newcommand{\deltag}{\delta_{g}}

\newcommand{\dgV}{d_{g}^V}

\newcommand{\deltagV}{\delta_{g}^V}
\newcommand{\dertZero}{\at{\tfrac{d}{dt}}{t=0}}

\newcommand{\Def}{\delta_{g}^*}
\newcommand{\Defd}{\delta^*_{\gD}}

\newcommand{\Sh}{\mathrm{S}}

\newcommand{\Volume}{\mathrm{Vol}}

\newcommand{\pzcdel}{\text{\textit{\textbf{\textdelta}}}}

\newcommand{\pzcDRic}{\mathpzc{Ric}_{g}}
\newcommand{\pzcD}{\mathpzc{D}}

\newcommand{\starG}{\star_g}
\newcommand{\starGV}{\starG^V}

\newcommand{\dr}{\partial_r}

\newcommand{\PttG}{\bbP^{\frakt\frakt}}
\newcommand{\PtnG}{\bbP^{\frakt\frakn}_g}
\newcommand{\PntG}{\bbP^{\frakn\frakt}_g}
\newcommand{\PnnG}{\bbP^{\frakn\frakn}_g}

\newcommand{\NN}{\mathrm{N}}

\newcommand{\DD}{\mathrm{D}}

\newcommand{\D}{\mathrm{D}}

\newcommand{\Lkm}[2]{\Lambda^{#1,#2}T^{*}{M}}
\newcommand{\Gkm}[2]{\calG^{#1,#2}T^{*}{M}}
\newcommand{\Wkm}[2]{\Omega^{#1,#2}(M)}
\newcommand{\Ckm}[2]{\scrC^{#1,#2}(M)}

\newcommand{\plWkm}[2]{\Omega^{#1,#2}(\dM)}
\newcommand{\plCkm}[2]{\scrC^{#1,#2}(\dM)}

\newcommand{\SM}{S^{2}(M)}
\newcommand{\SdM}{S^{2}(\dM)}

\newcommand{\XM}{\frakX(M)}

\newcommand{\StM}{S^{2}_{\frakt}(M)}
\newcommand{\SnM}{S^{2}_{\frakn}(M)}
\newcommand{\XzM}{\frakX_0(M)}

\newcommand{\dBianchi}{d_{g}}
\newcommand{\delBianchi}{\delta_g}
\newcommand{\dBianchiV}{d_g^V}

\makeatletter
\newcommand*\owedge{\mathpalette\@owedge\relax}
\newcommand*\@owedge[1]{%
  \mathbin{%
\ooalign{%
  $#1\m@th\bigcirc$\cr
  \hidewidth$#1\m@th\wedge$\hidewidth\cr
}%
  }%
}
\makeatother

\numberwithin{equation}{section}

\begin{document} 
\title[Cohomology for Linearized Ricci Curvature]{
Cohomology for Linearized Ricci Curvature 
}

\author{
Roee Leder}
\email{roee.leder@mail.huji.ac.il}
\address{
Einstein Institute of Mathematics,
The Hebrew University,
Jerusalem 9190401 Israel.
}

\maketitle

\begin{abstract}
The Ricci curvature equations are a central subject of study in geometry. However, in the smooth real case, their linear analysis is often confined to settings in which the background metric is Einstein. In this paper, we establish solvability and uniqueness conditions for the linearized problem on any compact Riemannian manifold with boundary. These conditions are formulated in terms of the cohomology of a canonical cochain complex, constructed by means of a generalized Hodge theory based on pseudodifferential methods. An important element of the theory is that it allows the incorporation of tensorial error terms, arising from linearized metric-dependent sources or from connections on the manifold of metrics. Using Bochner technique, we prove vanishing theorems for the cohomology under geometric assumptions on the boundary and error term, without imposing further interior restrictions.
\end{abstract}




{
\footnotesize
}

\section{Introduction}
\subsection{Background} 
\label{sec:cohomology}
On a compact manifold with boundary $M$, satisfying $\dim{M}\geq 3$, consider the map sending a metric to its Ricci curvature tensor,
\beq 
g \mapsto \Ric_{g} : \scrM(M) \rightarrow \SM.
\label{eq:Ricimap} 
\eeq 
Here, $\SM$ stands for symmetric tensor fields, and $\scrM(M) \subset \SM$ for the space of Riemannian metrics. Given a metric-dependent source $\mathrm{T}_g\in\SM$, the \emph{Ricci curvature equations} \cite{Ham84}, for an unknown $g \in \scrM(M)$, are then
\beq
\Ric_{g} = \mathrm{T}_{g}. 
\label{eq:Einz_Hinz}
\eeq
In K\"ahler-Einstein settings, where $\mathrm{T}_g=\lambda g$, global solvability results for these equations, obtained through continuity methods, are among the greatest achievements in geometry (to name but a few, \cite{Yau78,CDS15a,CDS15b,CDS15c,Tia15,JMR16}). In the smooth real case, however, the equations remain much less understood, and results on global solvability are scarce.

In this paper, we establish solvability and uniqueness conditions for linearizations of \eqref{eq:Einz_Hinz} at an arbitrary background metric $g\in\scrM(M)$, which is a necessary first step in studying \eqref{eq:Einz_Hinz} by means of implicit-function-theorem-based arguments. In order not to be confined to a specific linear theory, we allow the incorporation of a tensorial error term $\tGamma:\SM\rightarrow\SM$ into the standard linearization of $g\mapsto\Ric_{g}$, arising, for example, from linearizing a metric-dependent source $g\mapsto \mathrm{T}_{g}$ as in \eqref{eq:Einz_Hinz}, or from a connection on the manifold of metrics $\scrM(M)$ (cf.~\cite[pp.~93, 214]{Ham82}). The resulting linear equations then read, for prescribed $T\in\SM$ and unknown $\sigma \in \SM$,
\beq
\D_{\tGamma}\Ric_{g} \sigma = T, \qquad\text{where }\quad \D_{\tGamma}\Ric_g\sigma:= \dertZero \Ric_{g + t\sigma}+\tGamma\sigma.
\label{eq:Hinz_intro}
\eeq
Although equations of this form have been extensively studied in the real case, much of the existing analysis relies on additional restrictive structural assumptions on $M$, $g$, or $\tGamma$ (e.g., \cite{DeT81,DeT82,DK84,Ham84,Aub98,DG99,And08,AH22,AH24,Hin24,Del25,BEH+25}). Perhaps the most common such assumption, which this paper aims to relax, is that the background metric $g$ is Einstein and that $\tGamma$ is a multiple of the identity, namely,
\beq 
\Ric_{g} = \lambda g, \qquad \tGamma=-\lambda\id, \qquad \lambda\in\bbR. 
\label{eq:Einstein_assumption}
\eeq
One of the main analytic reasons for making these assumptions lies in a closer examination of the natural symmetries and constraints of \eqref{eq:Ricimap}, namely: 
\beq
\Ric_{\phi^* g} = \phi^* \Ric_{g},
\qquad\quad
\delBianchi \mathrm{B}_{g}\Ric_{g} = 0.
\label{eq:gauge_equivariance}
\eeq
Here, $\phi:M\rightarrow M$ is any diffeomorphism; $\delBianchi:\SM\rightarrow\XM$ is the tensor divergence; and $\mathrm{B}_{g}:\SM\rightarrow \SM$ the isomorphism $\mathrm{B}_{g}:=\id - \tfrac{1}{2} \trace_{g} (\cdot)\, g$. 

As has been widely recognized, under the Einstein assumption \eqref{eq:Einstein_assumption}, the linearization of \eqref{eq:gauge_equivariance} yields useful identities for the linearized operators:
\beq
\begin{aligned} 
(\D\Ric_{g}-\lambda\id)\circ\Def  = 0, \qquad \delBianchi\mathrm{B}_{g}\circ(\D\Ric_{g}-\lambda\id) = 0,
\label{eq:Einstein_relation_hinz}
\end{aligned} 
\eeq
where $\Def:\XM\rightarrow \SM$ is the Killing operator, given by $\Def{X}=\tfrac{1}{2}\calL_{X}g$. 

In other words, under \eqref{eq:Einstein_assumption} one can fit \eqref{eq:Hinz_intro} into a cochain complex \cite{Ham84}:
\beq 
\begin{tikzcd}
0 \arrow[r] &[-1.2em]
\XM \arrow[r, "\Def"] &
\SM \arrow[rr, "\D_{\tGamma}\Ric_{g}"] & &
\SM \arrow[r, "\delBianchi\mathrm{B}_{g}"] &
\XM \arrow[r] &[-1.2em]
0
\end{tikzcd}
\label{eq:cochain_complex} 
\eeq 
which, as in classical Hodge theory \cite{Sch95b}, gives rise to \emph{cohomology modules} $\scrE^{\bullet}(M,g,\lambda)$, identified by topologically direct, $L^{2}$-orthogonal decompositions
\beq
\begin{split} 
&\ker (\D_{\tGamma}\Ric_{g})=\image (\Def)\oplus \scrE^{1}(M,g,\lambda),
\\
&\ker (\delta_{g}\mathrm{B}_{g})=\image (\D_{\tGamma}\Ric_{g})\oplus\scrE^{2}(M,g,\lambda).
\end{split} 
\label{eq:cohomology_Hinz}
\eeq
These then translate directly into solvability and uniqueness for \eqref{eq:Hinz_intro}:
explicitly, solvability holds if and only if $T$ satisfies the compatibility conditions
\[
\delta_{g}\mathrm{B}_{g} T = 0, \qquad T \perp \scrE^{2}(M,g,\lambda),
\]
and uniqueness holds modulo $\scrE^{1}(M,g,\lambda)$ and the gauge freedom $X \mapsto \sigma + \Def{X}$, which can be fixed by imposing the divergence-free gauge, i.e.\ $\delta_{g}\sigma = 0$.

It is sometimes possible to identify the cohomology modules in \eqref{eq:cohomology_Hinz} explicitly; notably, it is shown in~\cite{Hin24} that if $(M,g)$ is a globally hyperbolic spacetime, then 
$\scrE^{2}(M,g,0)\simeq \ker(\Def)$. 
There is also a significant body of work on Einstein manifolds with nonempty boundary (pioneered in \cite{And08}), which are of particular importance, as the natural gauge group (diffeomorphisms that fix the boundary) acts freely in this case, thereby yielding smooth moduli spaces. In this setting, one has $\scrE^{1}(M,g,\lambda)=\{0\}$ under the topological condition $\pi_1(M;\dM)=\{0\}$, once the sequence \eqref{eq:cochain_complex} is adapted to incorporate the natural boundary conditions of the problem:
\[
\PttD\sigma=0, \qquad \D\Ah_{g}\sigma =0,
\]
where $\PttD:\SM\rightarrow\SdM$ is the pullback, and $\D\Ah_{g}:\SM\rightarrow\SdM$ is the linearized second fundamental form, $g\mapsto\Ah_{g}:\scrM(M)\rightarrow\SdM$. 

If the demand $\Ric_{g} = \lambda g$ is relaxed, however, the relations \eqref{eq:Einstein_relation_hinz} fail, the sequence \eqref{eq:cochain_complex} is not a cochain complex, and these analyses break down.

\subsection{Main results}
\label{sec:main_results_intro}
The circumstances laid out above are, in fact, quite ubiquitous in geometry: overdetermined (elliptic) problems in which linearized symmetries and constraints give rise to a sequence of operators that forms a cochain complex under geometric assumptions, but ceases to do so once these assumptions are relaxed. There is a long history of research on such problems, namely compatibility theory for overdetermined systems, BGG complexes, and related areas (see, for example, \cite{Cal61,Gol67,Spe69,DS96,CSS01,AH21,CELM21,Hu26}). The machinery behind these studies, however, is usually available under additional geometric assumptions, such as constant curvature, homogeneity or Kähler geometry. 

To address these circumstances, we developed in \cite{Led25B} a generalized Hodge theory, based on pseudodifferential methods, which canonically \emph{lifts} such sequences so as to recover a cochain complex. The main idea of the theory is to generalize the classical notion of an elliptic complex \cite{AS68} so that it applies more ubiquitously, leading to the notion of an \emph{elliptic pre-complex}.

In the present work, we provide a self-contained, adapted version of the theory so that it applies to the sequence \eqref{eq:cochain_complex}, and achieve the following: the operators are lifted so as to restore the cohomological picture for any $(M,g)$ and any tensorial error term $\tGamma$. Due to the structure of the gauge and boundary operators pertinent to Ricci curvature, the application of the theory turns out to be nontrivial, and involves several geometric arguments. As we shall see below, what we gain in return from this special structure is that the resulting cohomology reflects geometric properties of $(M,g)$ and $\tGamma$.

\begin{theorem}[The lifted complex]
\label{thm:lifting_intro}
Let $(M,g)$ be a compact Riemannian manifold with boundary, and $\Gamma:\SM\rightarrow\SM$ a smooth tensorial map. Then there exist continuous linear maps of Fréchet spaces,
\beq
\pzcD_{\tGamma}\pzcDRic : \SM \to \SM,
\qquad\quad
\pzcdel_{g} : \SM \to \XM,
\label{eq:lifted_operators}
\eeq
uniquely characterized by the following properties:
\beq 
\begin{aligned}
&\;\pzcD_{\tGamma}\pzcDRic \Def = 0
&& \quad \text{on } \quad
&& \XM \cap \ker(\cdot|_{\dM}), \\[1pt]
&\;\pzcD_{\tGamma}\pzcDRic := \D_{\tGamma}\Ric_{g}
&& \quad \text{on } \quad
&& \SM \cap \ker(\delBianchi), \\[3pt]
&\;\pzcdel_{g} \pzcD_{\tGamma}\pzcDRic = 0
&& \quad \text{on } \quad
&& \SM \cap \ker(\PttD, \D \Ah_{g}), \\[1pt]
&\;\pzcdel_{g}\mathrm{B}_{g} := \delBianchi\mathrm{B}_{g}
&& \quad \text{on } \quad
&& \SM \cap \ker(\pzcD_{\tGamma}\pzcDRic^{*}),
\end{aligned}
\label{eq:lifting_relations} 
\eeq   
where $\pzcD_{\tGamma}\pzcDRic^{*}$ denotes the formal $L^{2}$-adjoint of $\pzcD_{\tGamma}\pzcDRic$. 
\end{theorem}

Thus, for any $g,\tGamma$, there is a canonical cochain complex generalizing \eqref{eq:cochain_complex}:
\beq 
\begin{tikzcd}
0 \arrow[r] &[-1.2em]
\XzM \arrow[r, "\Def"] & \StM \arrow[rr, "\pzcD_{\Gamma}\pzcDRic"] & &[-1em] \SM \arrow[r, "\pzcdel_{g}\mathrm{B}_{g}"] & \XM  \arrow[r] &[-1.2em] 0,
\end{tikzcd}
\label{eq:lifted_complex} 
\eeq
where $\XzM:=\XM\cap\ker(\cdot|_{\dM})$ and $\StM:=\SM\cap\ker(\PttD, \D\Ah_{g})$. Indeed, in light of the identities \eqref{eq:Einstein_relation_hinz}, if one restricts to $\Ric_{g} = \lambda g$ and $\Gamma=-\lambda\id$, the characterizing relations \eqref{eq:lifting_relations} are trivially satisfied by
\beq
\pzcD_{\Gamma}\pzcDRic = \D\Ric_{g}-\lambda\id
\qquad \text{and} \qquad
\pzcdel_{g} = \delBianchi.
\label{eq:retrieve_when_Ein_0}
\eeq
Consequently, by the uniqueness clause in \thmref{thm:lifting_intro}, the lifted and original operators coincide in this case, and so do the complexes \eqref{eq:lifted_complex} and \eqref{eq:cochain_complex}. 

Most notably, \thmref{thm:lifting_intro} holds without structural assumptions. At this level of generality, although the lifted operators may be nonlocal, their specific characterization in \thmref{thm:lifting_intro} allows us to retain solvability and uniqueness for \eqref{eq:Hinz_intro}, once it is supplemented by its natural gauge and boundary conditions \cite{And08}; namely, for the overdetermined boundary-value problem:
\beq
\begin{aligned}
& \D_{\Gamma}\Ric_{g}\sigma = T, \qquad
&& \delBianchi \sigma = 0
\qquad && \text{in } M,
\\
& \PttD \sigma = 0, \qquad
&& \D\Ah_{g}\sigma = 0
\qquad && \text{on } \dM.
\end{aligned}
\label{eq:Einstein_boundary_opening}
\eeq
Before the statements, we make a technical remark which, although important, is not strictly required for following the results presented here: as a by-product of the functional-analytic characterization \eqref{eq:lifting_relations}, the operators \eqref{eq:lifted_operators} in fact belong to a pseudodifferential calculus of boundary-value problems, and differ from the original ones by zero-order terms in this calculus \cite{Bou71,RS82,Gru96}. This also explains the terminology ``lifting'', which is borrowed from other approaches to generalizing Hodge theory in a similar spirit, though in a manner different from ours \cite{KTT07,Wal15,SS19}.

\textbf{Solvability and uniqueness.}  We begin with the uniqueness result:
\begin{theorem}[First cohomology -- Uniqueness]
\label{thm:Saint_Venant_opening}
In the setting of \thmref{thm:lifting_intro}, denote the kernel of the boundary value problem \eqref{eq:Einstein_boundary_opening} by
\beq 
\scrE^{1}(M,g,\tGamma)
:=
\ker(\D_{\tGamma}\Ric_{g},\delBianchi,\PttD, \D\Ah_{g})
\subset \StM.
\label{eq:explicit_1_intro}
\eeq 
Then $\scrE^{1}(M,g,\tGamma)$ is finite-dimensional, and there is a topologically direct, $L^{2}$-orthogonal decomposition: 
\[
\ker(\pzcD_{\tGamma}\pzcDRic|_{\StM})=\image(\Def|_{\XzM})\oplus \scrE^{1}(M,g,\tGamma).
\]
\end{theorem}
We also take a moment to record a consequence of independent interest:
\begin{theorem}[Natural gauge compatibility]
\label{thm:Saint_Venant_opening2}
In the setting of \thmref{thm:Saint_Venant_opening}, given $\sigma\in \SM$, the system of equations 
\[
\begin{aligned}
&  \calL_{X}g= \sigma && \quad \text{in } \quad M.
\end{aligned}
\]
admits a solution $X\in\frakX(M)$ satisfying $X|_{\dM}=0$ if and only if 
\[
\begin{gathered}
\pzcD_{\tGamma}\pzcDRic \sigma = 0, \qquad
\PttD \sigma = 0, \qquad
\D\Ah_{g} \sigma = 0, \qquad
\sigma\,\bot\,\scrE^{1}(M,g,\tGamma).
\end{gathered}
\]
\end{theorem}
In view of \eqref{eq:retrieve_when_Ein_0}, this recovers the known unique continuation results obtained for Einstein manifolds \cite{And08,AH22}. From a broader perspective, however, what \thmref{thm:Saint_Venant_opening2} provides is an explicit compatibility condition for the range of the Killing operator when restricted to vector fields that vanish on the boundary. This range constitutes not only the natural gauge group for the linearized problem at hand, but also for many linear geometric inverse problems \cite{PSU23}. In this generality, our result is, to the best of our knowledge, new. Below, we further sharpen this compatibility condition by identifying conditions on the boundary under which $\scrE^{1}(M,g,\tGamma)=\{0\}$.

So far, both the construction of the lifted complex and uniqueness results have been obtained without assumptions on the error term $\tGamma$. To obtain useful solvability conditions, however, the presence of the boundary conditions in \eqref{eq:Einstein_boundary_opening} necessitates compatibility conditions on the boundary as well.

To recognize these, following the same paradigm of exploiting linearized geometric constraints we have undertaken thus far, we turn to the Einstein constraint equations on the boundary (e.g., \cite{And08,Hin24}),
\[
\PnD_{g}\mathrm{B}_{g}\Ric_{g}=(\mathrm{Sc}_{\gD} - |\Ah_{g}|_{\gD}^{2} + (\trace_{\gD} \Ah_{g})^2,\delta_{\gD} \Ah_{g} + d\trace_{\gD}\Ah_{g}),
\]
where $\gD:=\PttD g$ is the pullback metric, $\mathrm{Sc}_{\gD}$ is its intrinsic scalar curvature, and $\PnD_{g}: \SM|_{\dM} \to \XM|_{\dM}$ denotes contraction with the inward unit normal induced by $g$  on the boundary. By linearizing these, we find that, in order to refine the complex \eqref{eq:lifted_complex} further into
\beq
\begin{tikzcd}
0 \arrow[r] &[-1.2em]
\XzM \arrow[r, "\Def"] & \StM \arrow[rr, "\pzcD_{\Gamma}\pzcDRic"] & &[-1em] \SnM \arrow[r, "\pzcdel_{g}\mathrm{B}_{g}"] & \XM  \arrow[r] &[-1.2em] 0,
\end{tikzcd}
\label{eq:lifted_complex2} 
\eeq
where $\SnM = \SM \cap \ker(\PnD_{g}\mathrm{B}_{g})$, we must require an extra condition on $\tGamma$:
\beq 
\PnD_{g}\mathrm{B}_{g}\D_{\tGamma}\Ric_{g}\sigma=0, \qquad \text{for } \qquad \sigma\in\ker(\PttD,\D\Ah_{g}),
\label{eq:compatibilityRic}
\eeq 
We emphasize that the demand \eqref{eq:compatibilityRic} does not impose any restriction on the geometry of $g$: indeed, by expanding and comparing with the linearized constraint equations, one finds that the left-hand side reduces to a tensorial expression involving only $\sigma|_{\dM}$ and $\Ric_{g}|_{\dM}$, and so $\tGamma$ can be constructed to be arbitrary in the interior while satisfying \eqref{eq:compatibilityRic} by a smooth extension argument. We will elaborate on the significance of this both for nonlinear applications and for analyzing the cohomology below. 



Analogously to \thmref{thm:Saint_Venant_opening}, we then show that the second cohomology of the complex \eqref{eq:lifted_complex2} governs the solvability conditions of \eqref{eq:Einstein_boundary_opening}, and is isomorphic to the kernel of a boundary value problem—albeit a nonlocal one. For the claim, let $\tbP_{g} \colon \SM \rightarrow \SM$ denote the $L^{2}$-orthogonal projection onto $\ker\delBianchi$ (cf.~\cite{PSU23}, where it is referred to as the space of solenoidal fields).

\begin{theorem}[Second cohomology -- Solvability]
\label{thm:Einstein_boundary_opening}
In the setting of \thmref{thm:lifting_intro}, assume $\tGamma$ satisfies \eqref{eq:compatibilityRic}, and consider the kernel 
\beq 
\scrE^{2}(M,g,\tGamma)
:=
\ker(\delta_{g}, \tbP_{g}\mathrm{B}_{g}\D_{\ttGamma}\Ric_{g}, \PnD_{g}), \qquad \ttGamma:=\mathrm{B}_{g}^{-1}\tGamma^*\mathrm{B}_{g},
\label{eq:explicit_2_intro}
\eeq
Then there is a topologically direct, $L^{2}$-orthogonal decomposition: 
\beq 
\ker(\pzcdel_{g}|_{\SnM})=\mathrm{im}(\mathrm{B}_{g}\pzcD_{\tGamma}\pzcDRic|_{\StM})\oplus\scrE^{2}(M,g,\tGamma). 
\label{eq:E2iso} 
\eeq 
That is, given $T\in \SM$, \eqref{eq:Einstein_boundary_opening} has a solution $\sigma\in \SM$ if and only if
\[
\pzcdel_{g}\mathrm{B}_{g}T = 0, \qquad
\PnD_{g} \mathrm{B}_{g}T = 0, \qquad
\mathrm{B}_{g}T\,\bot\,\scrE^{2}(M,g,\Gamma).
\]
\end{theorem}

\textbf{Bochner technique.} The explicit expressions \eqref{eq:explicit_1_intro}--\eqref{eq:explicit_2_intro} for the cohomology modules are not incidental. Their derivation relies on structure specific to Ricci curvature: namely, the formal self-adjointness of $\mathrm{B}_{g}\D\Ric_{g}$, a feature also prominent in other studies \cite{Hin24}, and the fact that the restriction to the natural gauge, namely to $\ker\delta_{g}$, appears in both the domain and the codomain of the problem. As alluded to above, this structure, in turn, is what paves the way to apply a {Bochner technique} to these cohomology modules, completing the analogy with classical Hodge theory \cite{Pet16,PW21}.

In the body of the paper, we derive explicit Bochner formulae for elements of both cohomology modules. Here, for expository reasons, we record only the consequences: vanishing theorems that, under specific choices of error terms $\tGamma$, reduce to nothing but boundary curvature assumptions. In the statements, let $\calR_{g}:\SM\rightarrow\SM$ denote the Lichnerowicz curvature.
\begin{theorem}[Vanishing first cohomology]
\label{thm:vanishing_intro}
In the setting of \thmref{thm:Saint_Venant_opening}, assume that the boundary is convex ($\mathrm{A}_{g}\leq 0$), that $\tGamma$ is symmetric, and that
\[
\calR_{g}+\,\bar{\tGamma}\geq 0,
\qquad \text{where} \qquad
\bar{\tGamma}\sigma
:=
2\,\tGamma\sigma
+(\trace_{g}\tGamma\sigma)\,g.
\]
Then every $\sigma\in\scrE^{1}(M,g,\Gamma)$ is parallel, and $\sigma=0$ if either: 
\begin{enumerate}
\item $(\calR_{g}+\,\bar{\tGamma})|_{p}> 0$ for some $p\in M$,
\item $(\trace_{\gamma}\mathrm{A}_{g})|_{q}\neq 0$ for some $q\in \dM$. 
\end{enumerate}
\end{theorem}
In particular, one may choose $2\,\tGamma = -\calR_{g}$ to ensure that, for strictly convex nonempty boundaries, $\scrE^{1}(M,g,\tGamma) = \{0\}$, thereby sharpening \thmref{thm:Saint_Venant_opening2} accordingly. This demonstrates how a nonempty boundary and the freedom in choosing $\Gamma$ can be instrumental for identifying the cohomology.

The Bochner technique for the second cohomology is both more delicate and geometrically richer, not least because of the presence of non-local terms in \eqref{eq:explicit_2_intro}. To retain a manageable Bochner formula, we demonstrate that it is necessary to impose additional conditions on $\tGamma$ and elements of $\scrE^{2}(M,g,\tGamma)$, that do not restrict the interior geometry of $g$. We first present these conditions and then state the theorem.

The first condition we need is to restrict to elements $\eta\in\scrE^{2}(M,g,\tGamma)$ with constant trace, i.e., $d\trace_{g}\eta=0$. This condition extends the traceless gauge, which is central not only in the study of linearized Ricci curvature (see, e.g., \cite{And08}; \cite{KS26}, where a cohomological framework is built around this gauge under vanishing curvature; and \cite{BEH+25}, where it is used to prove vanishing theorems), but also in other geometric areas \cite{FT84,PRS08,PSU23}.

The second condition is to impose an additional boundary condition on the
linearized \emph{magnetic part} of the Weyl tensor \cite{CK93}:
\beq 
\PtnG\D\mathrm{Wey}_{g}\eta:=\dertZero\PtnG\mathrm{Wey}_{g+t\eta}=0,
\label{eq:Weyl}
\eeq 
where $\PtnG$ is the tangential-normal boundary projection of an algebraic curvature tensor. We remark that this
condition is satisfied trivially in three dimensions, since
$\mathrm{Wey}_{g}=0$ in that case, and is vacuous when the boundary is empty.
We discuss it and the constant-trace gauge further below.

The third condition is to restrict attention to a tensorial error term $\tGamma$ that satisfies, in addition to \eqref{eq:compatibilityRic}, the additional boundary identity:
\beq 
\begin{split} 
&\PttD\deltag\mathrm{B}_{g}\D_{\tGamma}\Ric_{g}\eta=0, \qquad \eta\in\SM\cap\ker(\deltag,d\trace_{g},\PnD_{g}).
\end{split} 
\label{eq:tgamma_condition} 
\eeq 
As in \eqref{eq:compatibilityRic}, this time by comparing with the linearized Bianchi identity, it is shown that the left-hand side of \eqref{eq:tgamma_condition} depends only on $\Ric_{g}|_{\dM}$ and tangnetial derivatives of $\tGamma$, and thus again imposes no restriction on the geometry of $g$ or on the interior behavior of $\tGamma$. 

The final condition is to assume that, when nonempty, $(\partial M,\gamma)$ is a round sphere, and that its second fundamental form satisfies $\Ah_{g} = -\sqrt{\kappa}\, \gD$ for $\kappa > 0$. Under these assumptions, the vanishing theorem becomes:
\begin{theorem}[Vanishing second cohomology]
\label{thm:vanishing_intro2}
In the setting of \thmref{thm:Einstein_boundary_opening}, assume that $\tGamma$ is symmetric, satisfies \eqref{eq:tgamma_condition}, and that
\beq 
\calR_{g} + 2\,\tGamma \geq 0, \qquad \trace_{\gD}\PttD(\tGamma+\tfrac{1}{4}\calR_{g})=0.
\label{eq:vanishingCon2} 
\eeq 
Let $\eta\in\scrE^{2}(M,g,\Gamma)$ satisfy $d\trace_{g}\eta=0$ and $\PtnG\D\mathrm{Wey}_{g}\eta=0$. Then $\eta=0$ in each of the following cases:
\begin{enumerate}
\item $\Ah_{g} = -\sqrt{\kappa}\, \gD$ for $\kappa > 0$ and $(\dM,\gamma)$ is a round sphere. 
\item $\dM=\emptyset$ and $(\calR_{g} + 2\,\tGamma)|_{p}> 0$ for some $p\in M$. 
\end{enumerate} 
\end{theorem}
As in \thmref{thm:vanishing_intro}, we remark that $\tGamma$ can always be chosen so that \eqref{eq:compatibilityRic}, \eqref{eq:tgamma_condition}, and \eqref{eq:vanishingCon2} are satisfied trivially. Indeed, these conditions do not fundamentally contradict one another, as the expressions in \eqref{eq:tgamma_condition} depend only on derivatives of $\tGamma$, and thus do not in themselves conflict with the requirements in \eqref{eq:vanishingCon2} and \eqref{eq:compatibilityRic}, which involve contractions of different components on the boundary. To illustrate this, if $\Ric_{g}|_{\dM}=0$, which is compatible with the fact that $(\dM,\gamma)$ is round, then both \eqref{eq:compatibilityRic} and the right-hand condition in \eqref{eq:vanishingCon2} are satisfied trivially by taking $\tGamma|_{\dM}=0$.

In a nutshell, the reason we require this collective set of assumptions is to reduce the nonlocal constraints in \eqref{eq:cohomology_2} to a condition of the form $d_{\gamma}\xi=0$, where $\xi\in \SdM$ is constructed from $\eta$, and $d_{\gamma}$ is the exterior covariant derivative acting on $\SdM$. This reduction is highly nontrivial and is based on comparing Weitzenb\"ock identities with linearized constraints and vector potentials provided by the nonlocal terms. The simple connectivity of the boundary is then needed to represent tensors satisfying $d_{\gamma}\xi=0$ by scalar potentials, as well as to enable a holonomy argument.
\subsection{Outlook and future directions}
\label{sec:outlook}
From the perspective of the utility of our results in the nonlinear setting, we revisit the natural interpretation of the error term $\tGamma$: namely, the case in which it arises from a connection on the manifold of metrics $\scrM(M)$ (\cite[pp.~93, 214]{Ham82}, \cite{Ebi70,GM91}), 
\[
\Gamma\sigma:=\Gamma_{g}(\Ric_{g},\sigma), \qquad \Gamma:\scrM(M)\times\SM\times\SM\to\SM.
\] 
In this case, the operator $\D_{\tGamma}\Ric_{g}$ is nothing but the covariant derivative of the vector field $g \mapsto \Ric_{g}$. We expect the incorporation of a connection to be useful for the nonlinear problem for several reasons. The first is that the covariant formulation extends the scope of the inverse function theorem (cf.~\cite[Thm.~3.34]{Ham82}). The second is that, as demonstrated in \thmref{thm:vanishing_intro}--\thmref{thm:vanishing_intro2}, when chosen appropriately, the error term can eliminate altogether the need for interior assumptions. This is expected to be useful in showing that the vanishing of an obstruction at the linear level leads to smooth moduli spaces, and also in setting up continuity paths for \eqref{eq:Einz_Hinz}.

We also record a geometric interpretation of the condition $d\trace_{g}\eta = 0$ in \thmref{thm:vanishing_intro2}: it corresponds, at the linearized level, to considering Ricci curvature as acting on metrics whose volume form satisfies $d\Volume_{g}=c\mu$ for some $c>0$, where $\mu$ is a fixed volume form. It appears that our framework can be adapted to this gauge, so that \thmref{thm:vanishing_intro2} takes the form of a vanishing theorem for the resulting cohomology. In this context, the extra boundary condition on the magnetic part of the Weyl tensor in \eqref{eq:Weyl} also suggests that, in higher dimensions and for nonempty boundary, the framework may require a corresponding adaptation in order to obtain vanishing cohomology. Both of these will be pursued in future work.
\subsection{Structure of the paper} 
In \secref{sec:Hodge}, we survey the required elements of the generalized Hodge theory. We include an appendix (\appref{app:Construction}) providing adapted, self-contained proofs, which are not strictly required to follow the geometric narrative of the paper. In \secref{sec:TecSetup}, we derive the geometric formulae and identities required for the analysis. In \secref{sec:Hodge_Einstein}, we cast linearized Ricci curvature within the scope of the Hodge theory, and prove the cohomological solvability and uniqueness results. In \secref{sec:uniqueness_result}, we derive the Bochner formulae and prove the vanishing theorems. In \secref{sec:duality_for_curvature}, we establish a duality relation between the Einstein and Riemann curvature tensors, which is used throughout the paper.
\subsection*{Acknowledgments}
I thank my advisor, Raz Kupferman, for his guidance and constructive comments; Or Hershkovits for his encouragement and remarks; Deane Yang for his feedback and for the suggestion to apply generalized Hodge theory to Ricci curvature; and Zhongshan An, Rotem Assouline, Louis-Brahim Beaufort, Tristan Humbert, Thibault Lefeuvre, Rafe Mazzeo, Pierre Pansu, Jake Solomon and Shira Tanny for helpful discussions.

This project was supported by the Israel Science Foundation, Grant No.~560/22, and by an Azrieli Graduate Fellowship from the Azrieli Foundation.
\section{Generalized Hodge theory}
\label{sec:Hodge} 
Here we survey the generalized Hodge theory developed in \cite{KL23,Led25B}. We do not require its full generality; we review a prototypical, self-contained framework applicable to the linearized Ricci curvature equations.

The theory is based on the pseudodifferential calculus of boundary value problems—namely, the Boutet de Monvel calculus, whose key properties we briefly review below \cite{Bou71,Gru96,RS82}. As remarked in the introduction, there are other approaches to generalizing Hodge theory within the Boutet de Monvel calculus~\cite{KTT07,Wal15,SS19}; see the comparison carried out in~\cite[Ch.~3.2.5]{Led25B} for details.
\subsection{The Boutet de Monvel calculus}
\label{sec:OD}
Here we review the pseudodifferential calculus of boundary value problems, namely the Boutet de Monvel calculus of Green operators, on which our theory relies. We do not require its full details, and therefore keep the presentation as concise as possible, referring to our survey~\cite[Ch.~2.2]{Led25B} and the texts \cite{Bou71,RS82,Gru96}.

The standard pseudodifferential calculus (e.g., \cite[Ch.~7]{Tay11b}), defined on every smooth manifold $M$ with empty boundary, consists of maps between sections of vector bundles $\bbE,\bbF\rightarrow M$:
\[
P:\Gamma(\bbE)\rightarrow\Gamma(\bbF),
\]
and is designed to preserve adjunction, composition, and inversion whenever these operations are well defined, starting from differential operators and building upward from there. Due to the nonlocal nature of the calculus, and the fact that it is based on the Fourier transform on $\bbR^{n}$, it does not directly adapt to manifolds with boundary. The Boutet de Monvel calculus of Green operators thus provides a framework to extend it, in the sense that, when the boundary is empty, the standard pseudodifferential calculus is recovered.

In essence, the construction of the Boutet de Monvel calculus proceeds as follows: one begins with a restricted class of pseudodifferential operators, defined on the double of a given manifold, that truncate appropriately to the manifold with boundary; these are the operators with the \emph{transmission property}. The shortcoming of this class is that compositions of truncations are not necessarily themselves truncations, and that boundary operators must also be incorporated to account for boundary conditions. One therefore introduces additional terms into the calculus, ultimately arriving at the notion of a \emph{Green operator}, which is a map between sections of vector bundles over the manifold with boundary, so as to retain the algebra structure. Although this is the general idea, the construction is technical and nontrivial.

Every Green operator is characterized by two parameters: an order $m\in\bbR$, quantifying the number of interior derivatives, and a class $r\in\mathbb{Z}$, quantifying the number of boundary normal derivatives. These parameters, in turn, determine continuous $W^{s,2}$ Sobolev extensions of section spaces over vector bundles. In particular, an operator of order and class zero extends to a continuous map between $L^2$-sections, whereas every operator $A$ of class zero admits an adjoint $A^*$.

Although the Boutet de Monvel calculus involves many technicalities, it suffices for the sake of this paper to review a small number of definitions and facts. We begin with the following notions \cite[Ch.~1.4--1.6]{Gru96}: 

\begin{definition}[Trace operator]
\label{def:trace}
Let $\bbE \to M$ be a smooth vector bundle, $g:\bbE\oplus\bbE\rightarrow\bbR$ a fiber metric, and $\nabg$ a connection. Let $\dr\in\frakX(U)$ denote the inward-pointing unit normal, given by the gradient of the distance function $r: U\rightarrow\bbR_{+}$ from the boundary, defined in a collar $U \supset \dM$. The $k$-trace operator is the map
\[
\Dng{k} : \Gamma(\bbE) \to \Gamma(\bbE|_{\dM})
\quad\text{given by}\quad
\Dng{k} \sigma = \big((\nabg_{\dr})^{k} \sigma \big)|_{\dM}.
\]
In particular, $\Dng{0} =(\cdot)|_{\dM}$; we use either notation as convenient.
\end{definition}
The following is then taken directly from \cite[Def.~1.4.3]{Gru96}:
\begin{definition}[Normal system of trace operators]
\label{def:normal_system}
Let $\bbE \rightarrow M$ and $\bbJ = \bigoplus_{j=1}^{r} \bbJ_{j} \rightarrow \dM$ be vector bundles. An operator 
$B : \Gamma(\bbE) \rightarrow \Gamma(\bbJ)$ is called a system of trace operators of class $r\in\Nzero$ if it is of the form
\[
B = \bigoplus_{j=1}^{r} B_{j}, \qquad 
B_{j} = \sum_{k=1}^{j} S_{kj}\Dng{k-1}+Q_{j},
\]
where $S_{kj} : \Gamma(\bbE|_{\dM}) \rightarrow \Gamma(\bbJ_{j})$ are pseudodifferential operators of order $j-k$, and $Q_{j}$ is a Green operator of order and class strictly lower than $j-1$. The operator $B$ is called a normal system of trace operators if the diagonal coefficients $S_{jj}$ are surjective.
\end{definition}
We then have \cite[Prop.~1.6.5--1.6.8]{Gru96}: 
\begin{proposition}
\label{prop:normality}
Let $B$ be a normal system of trace operators. Then:
\begin{itemize}
\item $B$ is surjective onto $\Gamma(\bbJ)$, and likewise onto its Sobolev completions.
\item $B$ admits a smooth right inverse within the calculus.
\item $\ker B$ is $L^{2}$-dense in $\Gamma(\bbE)$.
\end{itemize}
\end{proposition}

The second notion we will need is that of \emph{overdetermined ellipticity}. To state it, we note that the direct sum of a system of boundary operators $B:\Gamma(\bbE)\rightarrow\Gamma(\bbJ)$ of class $r$ with two interior operators of class zero, possibly of different orders $m_1$ and $m_2$,
\[
A_{i}:\Gamma(\bbE)\to\Gamma(\bbF_{i}), \qquad i\in\BRK{1,2},
\]
yields an operator within the calculus
\[
A_1\oplus A_2\oplus B:\Gamma(\bbE)\rightarrow\Gamma(\bbF_1)\oplus\Gamma(\bbF_2)\oplus\Gamma(\bbJ),
\]
which is the prototypical example of a Douglis--Nirenberg system, that is, a boundary-value problem of varying orders~\cite{DN55, Gru90}. A key property that such systems may possess is overdetermined ellipticity, which is a generalization of the classical notion of ellipticity, in which the symbol, adapted to incorporate boundary symbols and varying orders, is required to be injective rather than bijective.

To avoid delving into the machinery of symbol calculus, in this paper we use the fact that at the functional-analytic level, overdetermined ellipticity is equivalent to the existence of a continuous Sobolev extension that is semi-Fredholm \cite[Ch.~4.5]{Kat80}, i.e., one satisfying an a priori estimate for \emph{some} Sobolev norms:
\beq 
\|\psi\|_{s} \leq C \bigl(\|A_1 \psi\|_{s-m_{1}} + \|A_2 \psi\|_{s-m_{2}} +\sum_{j=0}^{r} \|B_j \psi\|_{s - j - 1/2} + \|\psi\|_0 \bigr),
\label{eq:Sobolev_estimate}
\eeq
where $s \in \Nzero$ is a sufficiently large Sobolev exponent.

Once such an estimate holds for some $s\in\Nzero$, it can be shown to hold for every $s> r-\tfrac{1}{2}$, and that the boundary value problem not only has a finite-dimensional kernel of smooth sections, but also admits a left inverse within the calculus. The choice of different norms on each component in \eqref{eq:Sobolev_estimate} is classically referred to as the choice of weights~\cite{DN55}. See~\cite[Ch.~2.1.2 \& Ch.~2.2.4]{Led25B} and~\cite[Sec.~2]{KL23} for further elaboration.

\subsection{Classical Hodge theory}The essential components of classical Hodge theory for the \emph{de Rham complex} were recognized long ago, later generalized to sections of vector bundles over compact manifolds with boundary, and unified within the framework of \emph{elliptic complexes} \cite{AS68}. Since this is a classical subject, we recall only the elements needed here, following the exposition in \cite[Ch.~12.A]{Tay11b} and our own review in \cite{Led25B}.

The components can be summarized in the following diagram:
\beq
\begin{xy}
(-30,0)*+{0}="Em1";
(-30,-20)*+{0}="Gm1";
(0,0)*+{\Gamma(\E_0)}="E0";
(30,0)*+{\Gamma(\E_1)}="E1";
(60,0)*+{\Gamma(\E_2)}="E2";
(90,0)*+{\Gamma(\E_3)}="E3";
(100,0)*+{\cdots}="E4";
(0,-20)*+{\Gamma(\bbJ_0)}="G0";
(30,-20)*+{\Gamma(\bbJ_1)}="G1";
(60,-20)*+{\Gamma(\bbJ_2)}="G2";
(90,-20)*+{\Gamma(\bbJ_3)}="G3";
(100,-20)*+{\cdots}="G4";
{\ar@{->}@/^{1pc}/^{A_0}"E0";"E1"};
{\ar@{->}@/^{1pc}/^{A_0^*}"E1";"E0"};
{\ar@{->}@/^{1pc}/^{A_1}"E1";"E2"};
{\ar@{->}@/^{1pc}/^{A_1^*}"E2";"E1"};
{\ar@{->}@/^{1pc}/^{A_2}"E2";"E3"};
{\ar@{->}@/^{1pc}/^{A_2^*}"E3";"E2"};
{\ar@{->}@/^{1pc}/^0"Em1";"E0"};
{\ar@{->}@/^{1pc}/^0"E0";"Em1"};
{\ar@{->}@/_{0pc}/^{B_0}"E0";"G0"};
{\ar@{->}@/_{0pc}/^{B_1}"E1";"G1"};
{\ar@{->}@/_{0pc}/^{B_2}"E2";"G2"};
{\ar@{->}@/^{1pc}/^{B^*_0}"E1";"G0"};
{\ar@{->}@/^{1pc}/^{B^*_1}"E2";"G1"};
{\ar@{->}@/^{1pc}/^0"E0";"Gm1"};
{\ar@{->}@/^{1pc}/^{B^*_2}"E3";"G2"};
{\ar@{->}@/_{0pc}/^{B^*_3}"E3";"G3"};
\end{xy}
\label{eq:elliptic_complex_diagram}
\eeq
where $\E_\alpha \to M$ and $\bbJ_\alpha \to \dM$ are sequences of vector bundles over the interior and the boundary, respectively, and the operators acting on the section spaces are, classically, differential operators (possibly of different orders). These are related to one another through \emph{Green's formulae}, valid for any $\psi \in \Gamma(\E_\alpha)$ and $\eta \in \Gamma(\E_{\alpha+1})$:
\beq
\bra A_\alpha\psi, \eta \ket_{L^2(M)} = \bra \psi, A_\alpha^*\eta \ket_{L^2(M)} + \bra B_\alpha\psi, B^*_\alpha\eta \ket_{L^2(\dM)},
\label{eq:Green_elliptic_intro}
\eeq
where $\bra \cdot, \cdot \ket_{L^2(M)}$ and $\bra \cdot, \cdot \ket_{L^2(\dM)}$ denote the $L^2$-pairings of interior and boundary sections, respectively, with respect to chosen Riemannian fiber metrics and volume forms. In the future, we omit these subscripts when there is no ambiguity. Also note that, although $A_{\alpha}^*$ is indeed the formal adjoint of
$A_{\alpha}$, $B_{\alpha}^*$ is not the formal adjoint of $B_{\alpha}$; we
abuse notation here.

\begin{definition}[Elliptic complex]
\label{def:elliptic_complex}
A diagram \eqref{eq:elliptic_complex_diagram} is called an \emph{elliptic complex} if the following are satisfied:
\begin{enumerate}[itemsep=0pt,label=(\alph*)]
\item The interior operators satisfy $A_{\alpha+1} A_{\alpha}=0$; that is, $(A_{\bullet})$ is a cochain complex.
\item The boundary-value problem associated with $(D_\alpha^*D_\alpha, T_\alpha)$ is elliptic, where $D_\alpha = A_{\alpha-1}^* \oplus A_\alpha$, so that $D_\alpha^*D_\alpha$ is the ``Laplacian":
\[
D_\alpha^*D_\alpha = A_{\alpha-1}A_{\alpha-1}^* + A^*_\alpha A_\alpha,
\]
and where $T_\alpha$ is a boundary operator defined by one of the following:
\begin{itemize}
\item[$(\NN)$] $T_\alpha = B^*_{\alpha-1} \oplus B^*_\alpha A_\alpha$;
\item[$(\DD)$] $T_\alpha = B_{\alpha} \oplus B_{\alpha-1}A^*_{\alpha-1}$.
\end{itemize}
\label{eq:classical_ellipticity}
\end{enumerate}
\end{definition}
Under these conditions, the elliptic complex is said to satisfy generalized \emph{Neumann conditions} (usually termed \emph{absolute}, due to their classical correspondence with absolute cohomology) in the case of $(\NN)$, and generalized \emph{Dirichlet conditions} (similarly termed \emph{relative}, due to the classical correspondence with relative cohomology) in the case of $(\DD)$.

The main result concerning elliptic complexes is an $L^{2}$-orthogonal, topologically direct decomposition of Fréchet spaces, the \emph{Hodge decomposition}. For Dirichlet conditions $(\DD)$, the decomposition takes the form:
\beq
\Gamma(\bbE_{\alpha+1})= 
\lefteqn{\overbrace{\phantom{\image(A_{\alpha}|_{\ker B_{\alpha}})\oplus \module_{\DD}^{\alpha+1}}}^{\ker(A_{\alpha+1}|_{\ker B_{\alpha+1}})}} \image(A_{\alpha}|_{\ker B_{\alpha}}) \oplus \underbrace{\module_{\DD}^{\alpha+1}\oplus \image(A^*_{\alpha+1})}_{\ker(A_{\alpha}^*)}
\label{eq:Hodge_intro_dirichlet}
\eeq
where the finite-dimensional space $\module_{\DD}^{\alpha+1} := \ker(A_{\alpha+1}, A_\alpha^*, B_{\alpha+1})$ satisfies
\beq 
\module_{\DD}^{\alpha+1} \simeq \ker(A_{\alpha+1}|_{\ker B_{\alpha+1}})\big/\image(A_{\alpha}|_{\ker B_{\alpha}}) \simeq \ker(A^*_{\alpha})\big/\image(A^*_{\alpha+1}).
\label{eq:cohomology_hodge_classic_dirichlet}
\eeq

The problem with this picture is that, despite the apparent generality of elliptic complexes, a cohomological formulation for problems beyond the de Rham complex is often out of reach. The reason is twofold. First, as alluded to in \secref{sec:main_results_intro}, constructing a sequence $(A_\bullet)$ that incorporates a given $A$ into a cochain complex is generally not possible unless restrictive geometric assumptions are imposed. Second, geometric problems often fail to satisfy the required ellipticity conditions, as they are typically overdetermined---again, usually due to gauge invariance and geometric constraints. Specifically, in the diagram \eqref{eq:elliptic_complex_diagram}, $A_{\alpha}$ and $A_{\alpha-1}$ may not have the same order, implying that the associated Laplacian $D^*_{\alpha}D_{\alpha}$ fails to be elliptic even in the interior. Indeed, as aptly put in \cite[Ch.~12.A, p.~465]{Tay11b}:
\begin{quotation}
``With this sketch of elliptic complexes done, it is time to deliver the bad news. The regularity (i.e., ellipticity) hypothesis is rarely satisfied, other than for the de Rham complex.''
\end{quotation}
\subsection{Generalized Hodge theory}\label{sec:Hodge_intro}Our generalization of elliptic complexes, namely the notion of an
\emph{elliptic pre-complex}, is designed to
systematically weaken the requirements on the diagram
\eqref{eq:elliptic_complex_diagram} above, thereby accommodating a broader
class of geometric examples while still allowing Hodge decompositions and
cohomology to manifest.

We state here only the Dirichlet version, developed in \cite{Led25B}, and in a
prototypical form suitable for our purposes. While still built on diagrams of
operators of the form \eqref{eq:elliptic_complex_diagram}, related to one
another by the Green formulae \eqref{eq:Green_elliptic_intro}, this notion not
only removes the requirement that the main sequence form a cochain complex, but
also weakens the demands of \defref{def:elliptic_complex} by replacing
ellipticity with the notion of overdetermined ellipticity: 
\begin{definition}[Elliptic pre-complex, Dirichlet conditions]
\label{def:elliptic_pre_complex}
Let $\alpha_{0} \in \Nzero \cup \BRK{\infty}$. A diagram of operators \eqref{eq:elliptic_complex_diagram}, satisfying Green's formulae \eqref{eq:Green_elliptic_intro}, is called an $\alpha_{0}$-elliptic pre-complex (based on Dirichlet conditions) if, for every $\alpha \leq \alpha_{0}$, the following conditions are satisfied:
\begin{enumerate}[itemsep=0pt,label=(\alph*)]
\item $B_{\alpha}$ and $B_{\alpha}^*$ are normal systems of trace operators.
\vspace{0.4em}
\item The diagram satisfies the following \emph{order reduction property}:
\[
\begin{gathered} 
\ord(A_{\alpha+1} A_{\alpha}) \leq \ord(A_{\alpha}), \qquad
\ker B_{\alpha-1} \subseteq \ker(B_{\alpha} A_{\alpha-1}).
\end{gathered}
\]
\item The system $A_{\alpha} \oplus A_{\alpha-1}^* \oplus B_{\alpha}$ is overdetermined elliptic.
\end{enumerate}
\end{definition}
We remark that, although perhaps not immediately apparent, items (a) and (b) are satisfied trivially in an elliptic complex as in \defref{def:elliptic_complex}, so that this is indeed a generalization. An analogous picture adapted to
Neumann conditions can also be formulated and, in fact, preceded the Dirichlet
picture chronologically. This theory was developed in~\cite{KL23} under the
term ``elliptic pre-complex,'' without distinguishing it as the Neumann case,
and was also subsequently extended in \cite{Led25B}. 


The main result for an $\alpha_{0}$-elliptic pre-complex is the following:
\begin{theorem}[Lifted complex]
\label{thm:lifted_complex}
For every $\alpha \leq \alpha_0$, there exists a continuous linear map of Fréchet spaces 
\[
\bA_{\alpha+1} : \Gamma(\bbE_{\alpha+1}) \to \Gamma(\bbE_{\alpha+2}),
\]
uniquely characterized by the following properties:
\[
\begin{aligned}
& (1) \quad && \bA_{\alpha+1}\bA_\alpha = 0,
& \qquad &&& \text{on } \ker B_\alpha, \\
& (2) \quad && \bA_{\alpha+1} = A_{\alpha+1},
& \qquad &&& \text{on } \ker \bA^*_{\alpha} .
\end{aligned}
\]
The operator $\bA_{\alpha+1}$ differs from $A_{\alpha+1}$ by a Green operator of order and class zero. The resulting sequence of operators, denoted by $(\bA_{\bullet})$, is called the lifted complex induced by $(A_{\bullet})$.
\end{theorem}

By the defining properties of the lifted complex, it follows in particular that $\bA_0 = A_0$ trivially. 

It is important to remark that none of the main results of the present paper rely on the technical aspects of the proof. However, for completeness, and since the treatment in \cite{Led25B} is much more general than what is required here, we include an appendix (\appref{app:Construction}) that distills a proof for the present case and its consequences. 

Since $\bA_{\alpha}-A_{\alpha}$ is of zero order and class, the lifted operators $\bA_{\alpha}$ and their adjoints $\bA_{\alpha}^*$ inherit the Green's formula \eqref{eq:Green_elliptic_intro} without extra boundary terms:
\beq
\bra \bA_{\alpha}\psi,\eta\ket
= \bra\psi,\bA^*_{\alpha}\eta\ket
+ \bra B_{\alpha}\psi,B_{\alpha}^*\eta\ket.
\label{eq:Green_forumula_correction}
\eeq
As a byproduct, for every $\alpha\leq \alpha_0$:
\beq 
\begin{aligned}
(1) \qquad & \bA_{\alpha}^* \, \bA_{\alpha+1}^* = 0 \;\;\text{identically}, \\
(2) \qquad & \ker B_{\alpha-1} \subseteq \ker(B_{\alpha}\bA_{\alpha-1}).
\end{aligned}
\label{eq:lifted_adjoint_relations} 
\eeq  
These relations, in turn, give rise to generalized Hodge decompositions:  
\begin{theorem}[Hodge decomposition]
\label{thm:Hodge_decomposition}
For every $\alpha < \alpha_{0}$, there exists a topologically direct, $L^{2}$-orthogonal decomposition of Fréchet spaces given by
\beq
\Gamma(\bbE_{\alpha+1}) =
\lefteqn{\overbrace{\phantom{\image{\bA_{\alpha}|_{\ker B_{\alpha}}}\oplus \module_{\D}^{\alpha+1}}}^{\ker(\bA_{\alpha+1}|_{\ker B_{\alpha+1}})}}
\image{(\bA_{\alpha}|_{\ker B_{\alpha}})}
\oplus
\underbrace{\module_{\D}^{\alpha+1}\oplus \image{(\bA^*_{\alpha+1}})}_{\ker (\bA_{\alpha}^*)}.
\label{eq:Hodge_intro}
\eeq
The $L^{2}$-orthogonal projections onto the summands in \eqref{eq:Hodge_intro} are Green operators of order and class zero.

Moreover, the module $\module_{\D}^{\alpha+1} 
:= \ker(\bA_{\alpha+1},\bA_{\alpha}^*,B_{\alpha+1})$ is finite-dimensional and satisfies: 
\beq
\module_{\D}^{\alpha+1}
\simeq
\ker(\bA_{\alpha+1}|_{\ker{B_{\alpha+1}}})
\big/\image(\bA_{\alpha}\vert_{\ker B_{\alpha}})\simeq \ker(\bA_{\alpha}^*)\big/\image(\bA_{\alpha+1}^*).
\label{eq:cohomology_corrected}
\eeq
\end{theorem}

We remark that since the projections onto the constituent spaces belong to the calculus, the decompositions extend to Sobolev regularity. This is not required for the present work, as the entire analysis is carried out in the smooth Fréchet topology.

The explicit characterization of the lifting in \thmref{thm:lifted_complex} then allows us to obtain reduced expressions for the adjoint and the cohomology:

\begin{proposition}
\label{prop:coadjoint}
In the setting of \thmref{thm:Hodge_decomposition}, let 
\[
\tbP_{\alpha-1} \colon \Gamma(\bbE_\alpha) \to \Gamma(\bbE_\alpha)
\] 
be the $L^2$-orthogonal projection onto $\ker(\bA_{\alpha-1}^*)$, so $\id - \tbP_{\alpha-1}$ is the projection onto $\image(\bA_{\alpha-1}|_{\ker B_{\alpha-1}})$. Then, for every $\alpha\leq \alpha_0$,  
\[
\bA_\alpha = A_\alpha\tbP_{\alpha-1} \Textand
\bA_\alpha^* = \tbP_{\alpha-1}A_\alpha^*,
\]
and the cohomology \eqref{eq:cohomology_corrected} can be rewritten in the following form:
\beq
\module_{\D}^{\alpha+1} = \ker(A_{\alpha+1}, \tbP_{\alpha-1}A^*_\alpha, B_{\alpha+1}).
\label{eq:cohomology_expression}
\eeq
\end{proposition}

We emphasize that the fact that the cohomology modules $\module_{\D}^{\alpha+1}$ in \eqref{eq:cohomology_corrected} can be expressed in terms of the original, non-lifted operators is an important feature of the theory. Without this property, the raw expression for the cohomology in \eqref{eq:cohomology_corrected} would be difficult to interpret. This characterization is a direct consequence of how the lifting is carried out in \thmref{thm:lifted_complex}.

We also remark that, for $\alpha=-1$ and $\alpha=0$, since $\tbP_{-2}=0$ and $\tbP_{-1}=\id$, the cohomologies there are expressible purely in terms of the original operators: 
\beq 
\module^0_{\D}=\ker(A_0,B_0), \qquad \module^1_{\D}=\ker(A_1,A_0^*,B_1).
\label{eq:lower_cohomolgies}
\eeq 
In general, although \eqref{eq:cohomology_expression} implies that it always holds that 
\[
\ker(A_{\alpha+1}, A_\alpha^*, B_{\alpha+1}) \subseteq \module^{\alpha+1}_{\D},
\]
there exist counterexamples to the reverse inclusion. 

Iterating the Hodge decompositions \eqref{eq:Hodge_intro}, together with the Green's formulae \eqref{eq:Green_forumula_correction} and the defining relations of the lifted complex $(\bA_{\bullet})$, yields solvability and uniqueness for the original operators, once supplemented with the canonical gauge and boundary conditions:

\begin{theorem}[Cohomological formulation]
\label{thm:cohomology_dirichlet}
Let $\alpha< \alpha_0$ and $\eta \in \Gamma(\bbE_{\alpha+1})$. Then the boundary-value problem
\[
\begin{aligned}
&A_\alpha \omega = \eta, \qquad \bA^*_{\alpha-1}\omega = 0 
\qquad &&\text{in } M, \\
&\qquad\quad\, B_{\alpha}\omega = 0 
\qquad &&\text{on } \dM
\end{aligned}
\]
admits a solution $\omega\in\Gamma(\bbE_{\alpha})$ if and only if: 
\[
\bA_{\alpha+1}\eta = 0, \qquad B_{\alpha+1}\eta = 0, \qquad \eta \perp_{L^2} \module^{\alpha+1}_{\D}.
\]
The solution is unique modulo $\module_{\D}^\alpha$.
\end{theorem}

The definition of an elliptic pre-complex can be generalized in several ways to broaden the scope of the theory even further. For our purposes here, we will need a technical variant, presented and developed in \cite[Ch.~4.1.3]{Led25B}.

\begin{definition}[Disrupted elliptic pre-complex]
\label{def:disrupted_pre_complex}
A diagram $(A_{\bullet})$ as in \eqref{eq:elliptic_complex_diagram} is called a disrupted $\alpha_0$-elliptic pre-complex (based on Dirichlet conditions) if all requirements of \defref{def:elliptic_pre_complex} hold, with the following differences:
\begin{enumerate}
\item \emph{(Finiteness)} $A_{\alpha} = 0$ and $B_{\alpha}=0$ for all $\alpha \geq \alpha_0 \in \Nzero$.
\vspace{0.4em}
\item \emph{(Disruption)} Overdetermined ellipticity fails at the penultimate level:
\[
A_{\alpha_0-1} \oplus A_{\alpha_0-2}^* \oplus B_{\alpha_0-1}
\qquad \text{is not overdetermined elliptic}.
\]
\end{enumerate}
\end{definition}
\begin{proposition}
\label{prop:disrupted_valid}
All the results above hold verbatim in the disrupted case, except for the finite dimensionality of $\module^{\alpha_0-1}_{\D}$ when $\alpha = \alpha_0-1$.
\end{proposition}
In particular, the Hodge decompositions \eqref{eq:Hodge_intro} remain valid as stated; namely, the summands are closed subspaces in the smooth Fréchet topology, and the projections onto them belong to the calculus. In broad terms, this holds because the ranges of the operators in the decomposition remain closed, even without overdetermined ellipticity at the penultimate level, thereby preserving the $L^{2}$-splitting and the associated regularity. We do, however, lose the finite dimensionality of the kernel. A proof of \propref{prop:disrupted_valid} is given in the appendix (\secref{app:disrupted_proof}).
\section{Setup for the Hodge theory}
\label{sec:TecSetup} 
\subsection{Linearized Ricci curvature}
\label{sec:setup}
As premised in \secref{sec:cohomology}, the starting point of our study is to view the Ricci tensor as a nonlinear differential operator:
\[
g \mapsto \Ric_g : \scrM(M) \rightarrow \SM.
\]
As discussed in \secref{sec:main_results_intro} and \secref{sec:outlook}, we then fix an arbitrary tensorial endomorphism $\tGamma:\SM\rightarrow\SM$, which we will write for short as $\tGamma\in\End(\SM)$. We recall that a prime example is provided by endomorphisms arising from connections~\cite[pp.~93, 214]{Ham82} on the space of Riemannian metrics that are tensorial in their $\SM$-arguments, defined by
\beq 
\tGamma\sigma := \tGamma(\Ric_{g}, \sigma), \qquad
\tGamma \colon \scrM(M) \times \SM \times \SM \to \SM.
\label{eq:connection}
\eeq 
This makes the linear map $\D_{\Gamma}\Ric_{g} \colon \SM \rightarrow \SM$, 
\[
\D_{\Gamma}\Ric_{g}\sigma := \dertZero\Ric_{g+t\sigma} + \tGamma\sigma, 
\]
the covariant derivative of $g \mapsto \Ric_{g}$, viewed as a vector field on $\scrM(M)$. Despite this geometric interpretation, we will not rely on the assumption that $\tGamma$ arises from a connection; the fact that it is a tensorial endomorphism is the only property we require.

We also recall the explicit formula for the linearized Ricci tensor (see, e.g., \cite[p.~4]{And08} and \cite{FM75,PW21}), on which we will rely:
\beq
\D \Ric_g \sigma = \tfrac{1}{2} {\nabg}^* \nabg \sigma - \Def \delBianchi \mathrm{B}_{g} \sigma + \tfrac{1}{2} \calR_g \sigma,
\label{eq:Ricci_variation}
\eeq
where $\calR_{g} \colon \SM \rightarrow \SM$ denotes the Lichnerowicz curvature operator, and $\mathrm{B}_g \colon \SM \rightarrow \SM$ is the Bianchi contraction
\[
\mathrm{B}_{g}\sigma = \sigma - \tfrac{1}{2}\trace_{g}\sigma\,g.
\]
The explicit expression for $\calR_{g}$ shows that when $g$ is viewed as an element of $\SM$ and the trace as a map $\trace_{g} \colon \SM \rightarrow C^{\infty}(M)$, we have $\trace_{g} \calR_{g} = 0$ and $\calR_{g} g = 0$; thus:
\beq 
\mathrm{B}_{g}\calR_{g} = \calR_{g}\mathrm{B}_{g} = \calR_{g}. 
\label{eq:Lic_prop}
\eeq
We also recall the Killing operator $\Def \colon \XM \rightarrow \SM$ and its adjoint, the tensor divergence $\delBianchi \colon \SM \rightarrow \XM$, defined respectively by
\[
\Def{X} = \tfrac{1}{2}\calL_{X}g, \qquad (\delBianchi\sigma,X)_{g} = -\sum_{i} (\nabg_{E_i}\sigma)(E_i, X),
\]
where $\{E_i\}$ is any local orthonormal frame.
\subsection{Double forms and Green's formulae} 
\label{sec:double_forms}
To cast $\D_{\tGamma}\Ric_{g}$ and its associated operators $\Def$ and $\deltag$ within the generalized Hodge theory developed in \secref{sec:Hodge}, we must first identify the Green's formulae satisfied by these operators, together with the boundary systems naturally associated to them. This becomes clear from the requirements of the theory, as well as from the structure of the diagram \eqref{eq:elliptic_complex_diagram}.

For this purpose, and for several of the computations to follow—not only the derivation of Green's formulae, but also the manipulation of linearized curvature and its constraints, and later the Bochner formulae discussed in \secref{sec:main_results_intro}—we find that it is both convenient and conceptually clarifying to work in the language of vector-valued forms taking values in the exterior algebra. These are known in the literature as \emph{double forms}.

The notion of double forms traces back to de Rham's double currents, and was introduced into the Riemannian setting by Calabi \cite{Cal61} in the study of the Killing operator, before being developed extensively by Kulkarni \cite{Kul72} in the study of curvature symmetries. We recall here the part of this framework needed for our purposes, following our previous work~\cite{KL21a}.

On a Riemannian manifold with boundary $(M,g)$ of dimension $n=\dim M$, and for $k,m\in\Nzero$, let
\[
\Lkm{k}{m} := \Lambda^k T^*M \otimes \Lambda^m T^*M
\]
denote the graded algebra of $(k,m)$-covectors. This structure naturally inherits several operations from the exterior algebra, namely the wedge product and Hodge-dual isomorphisms, and has a natural involution:
\[
\begin{aligned}
&(\cdot)^T : \Lkm{k}{m} \to \Lkm{m}{k}, \quad 
&&\wedge : \Lkm{k}{m} \times \Lkm{\ell}{n} \to \Lkm{k+\ell}{m+n} \\ 
&\starG : \Lkm{k}{m}\to \Lkm{d-k}{m}, \quad 
&&\starGV : \Lkm{k}{m} \to \Lkm{k}{d-m},
\end{aligned}
\]
where $\starGV \psi = (\starG \psi^T)^T$. We denote by $\Wkm{k}{m}$ the corresponding space of sections, namely the module of $(k,m)$-double forms. In this language, the space $\SM$ of symmetric tensor fields is identified with the subspace
\[
\SM=\{\sigma\in\Wkm{1}{1} : \sigma^T=\sigma\}.
\]
The interior products are likewise inherited from the exterior algebra:
\[
i_X \colon \Lkm{k}{m} \to \Lkm{k-1}{m} \Textand i_X^V \colon \Lkm{k}{m} \to \Lkm{k}{m-1},
\]
with the analogous definitions. Thus, given an orthonormal frame $\{E_i\}$, the metric trace operator is expressed as
\beq 
\trace_g \colon \Lkm{k}{m}\to \Lkm{k-1}{m-1}
\qquad\qquad
\trace_g \psi = \sum_{i=1}^n i_{E_i} i^V_{E_i} \psi.
\label{eq:trace}
\eeq 

As Calabi and Kulkarni pointed out, the natural symmetries of the Riemann curvature tensor allow it to be viewed, in this framework, as a nonlinear differential operator
\[
\Rm \colon \scrM(M)\rightarrow \Wkm{2}{2}, \qquad g\mapsto\Rm_g.
\]
Its linearization at a reference metric $g\in\scrM(M)$ is then the second-order differential operator~\cite[p.~560]{Tay11b}
\[
\D\Rm_g \colon \SM \rightarrow \Wkm{2}{2}, \qquad \D\Rm_g\sigma:=\dertZero\Rm_{g+t\sigma}.
\]
We next record a more explicit form of this operator. Since double forms are particular instances of vector-valued forms, the exterior covariant derivatives and their adjoints~\cite[Ch.~9]{Pet16} apply to them as well, giving the operations
\beq 
\dg \colon \Wkm{k}{m} \to \Wkm{k+1}{m}
\quad\text{and}\quad
\deltag \colon \Wkm{k+1}{m} \to \Wkm{k}{m},
\label{eq:exterior1}
\eeq 
where the background connection is induced by the Riemannian metric. Acting through the involution, these in turn define the operators
\beq 
\dgV \colon \Wkm{k}{m} \to \Wkm{k}{m+1}
\quad\text{and}\quad
\deltagV \colon \Wkm{k}{m+1} \to \Wkm{k}{m},
\label{eq:exterior2}
\eeq 
where $\dgV \psi = (\dg \psi^T)^T$ and $\deltagV \psi = (\deltag \psi^T)^T$. 

In particular, the Killing operator takes the form
\beq 
\Def{X}=\dgV\omega+(\dgV \omega)^{T}, \qquad \text{where } \omega=\tfrac{1}{2}X^{\flat_{g}}\in\Wkm{1}{0}.
\label{eq:Killing_exterior_relations}  
\eeq 
The adjoint of this operator, the tensor divergence, can be identified with $\delBianchi\sigma\in\Wkm{0}{1}\simeq\XM$ up to the musical isomorphism; we shall therefore use this identification whenever no ambiguity can arise. For the linearized curvature, a direct computation gives~\cite[p.~560]{Tay11b}
\beq 
2\,\D\Rm_g=\Hg-D_g,
\label{eq:DRm}
\eeq 
where $D_g$ is a tensorial term depending on $\Rm_g$, and $\Hg \colon \SM \rightarrow\Wkm{2}{2}$ is the second-order differential operator
\beq 
\Hg:=\tfrac{1}{2}(\dg\dgV+\dgV\dg).
\label{eq:curlcurl} 
\eeq

We now introduce boundary operators on double forms, so we can formulate the relevant Green's formulae for the geometric operators above. Recall first the second fundamental form and the shape operator of the boundary~\cite[Ch.~3]{Pet16}; in the convention of \defref{def:normal_system}, these are
\[
\Ah_{g}= (\nabg\dr)^{\flat}=\mathrm{Hess}_g r, \qquad \Sh_{g}=\nabg\dr.
\]
Let $\jmath \colon \dM\hookrightarrow M$ be the boundary inclusion. For a double form $\eta\in\Wkm{k}{m}$, we define its mixed boundary projections by
\beq 
\begin{aligned}
&\PttG\eta = \jmath^*\eta = \jmath^* \Dng{0}\eta\in\plWkm{k}{m}, 
\\&\PnnG\eta=\jmath^*i_{\dr}i^{V}_{\dr}\eta = \jmath^* i_{\dr}i^{V}_{\dr}\Dng{0}\eta\in\plWkm{k-1}{m-1}
\\&\PntG\eta=\jmath^*i^{V}_{\dr}\eta= \jmath^* i^{V}_{\dr}\Dng{0}\eta \in\plWkm{k}{m-1}
\\ &\PtnG\eta=\jmath^*i_{\dr}\eta = \jmath^* i_{\dr}\Dng{0}\eta \in\plWkm{k-1}{m},
\end{aligned}
\label{eq:boundary_tn}
\eeq 
where the subscripts $\frakn\frakn$, $\frakt\frakn$, $\frakn\frakt$, and $\frakt\frakt$ denote the normal-normal, tangential-normal, normal-tangential, and tangential-tangential projections, respectively. The absence of a subscript $g$ in $\PttD$ reflects the fact that this projection is metric-independent, whereas the remaining projections depend on $g$ through the choice of the inward normal $\dr$. For a symmetric tensor $\sigma \in \SM$, the tensor $\PttG \sigma\in\SdM$ is precisely its boundary pullback. 

We can use these to express the boundary terms in the standard Green's formula for the exterior covariant derivative and its adjoint~\cite[p.~180]{Tay11a}. In the notation of \eqref{eq:Green_elliptic_intro}, for any $\omega\in\Wkm{k}{m}$ and $\eta\in\Wkm{k+1}{m}$, the formula reads
\beq  
\bra \dg\omega,\eta\ket=\bra \omega,\deltag\eta\ket-\bra \PttD\omega,\PntG\eta\ket-\bra \PtnG\omega,\PnnG\eta\ket.
\label{eq:Greens_exterior}  
\eeq

We next use this structure to derive a Green's formula for $\D\Rm_{g}$ in \eqref{eq:DRm}. We find that the natural boundary operator in this context is the linearized second fundamental form. Indeed, in analogy with the maps $g\mapsto\Ric_g$ and $g\mapsto\Rm_{g}$, the second fundamental form may be regarded as a map
\[
\Ah \colon \scrM(M)\rightarrow S^2(\dM), \qquad g\mapsto\Ah_{g},
\]
whose linearization at $g\in\scrM(M)$ is the first-order differential operator
\[
\D\Ah_{g} \colon \SM \rightarrow S^2(\dM), \qquad \D\Ah_{g}\sigma:=\dertZero\Ah_{g+t\sigma}. 
\]
To express this operator in a form compatible with \eqref{eq:DRm}, we introduce the boundary operators \cite[Ch.~3]{KL21a}
\beq 
\begin{aligned}
\frakT_g\sigma &= \tfrac{1}{2} \brk{\PntG \dg\sigma-\dg \PntG\sigma} + \tfrac{1}{2} \brk{\PtnG\dgV\sigma- \dgV \PtnG\sigma}, \\
\frakT_g^*\eta &= -\tfrac{1}{2}\brk{ \PtnG\deltag\eta+\deltag \PtnG\eta} - \tfrac{1}{2}\brk{\PntG\deltagV\eta+\deltagV \PntG\eta}.
\end{aligned}
\label{eq:frakT}
\eeq 
An explicit computation~\cite[p.~755]{KL21a} shows that
\beq
\frakT_{g}\sigma=-d_{\gamma}^{V}\PtnG\sigma-(d_{\gamma}^{V}\PtnG\sigma)^{T}-(\PnnG\sigma)\Ah_{g}+\Sh_{g}\PttD\sigma+\PttD\nabg_{\dr}\sigma.
\label{eq:frakT_explicit} 
\eeq 
Here $\Sh_{g}\colon S^2(\dM)\rightarrow S^2(\dM)$ denotes the symmetrized contraction
\[
(\Sh_{g}\mu)(X;Y)=\tfrac{1}{2}\mu(\Sh_{g}(X);Y)+\tfrac{1}{2}\mu(\Sh_{g}(Y);X).
\]
Comparing this formula with the known expression for $\D\Ah_{g}$ (e.g., \cite[pp.~5--6]{And08}) and with the expression for $\Defd$ in \eqref{eq:Killing_exterior_relations}, we obtain
\beq 
2\,\D\Ah_{g}=\frakT_g+\Sh_{g}\PttG.
\label{eq:DA}
\eeq

The definition of $\frakT_{g}$ is tailored precisely to the leading second-order term in \eqref{eq:DRm}. Consequently, a direct iteration of Green's formula \eqref{eq:Greens_exterior} yields
\beq 
\bra \Hg\psi,\eta\ket=\bra\psi,\Hg^*\eta\ket+\bra\PttG\psi,\frakT_g^*\eta\ket-\bra\frakT_g\psi,\PnnG\eta\ket,
\label{eq:Hg_Green}
\eeq 
where $\Hg^* \colon \Wkm{k+1}{m+1} \rightarrow \Wkm{k}{m}$ is defined by
\beq 
\Hg^*=\tfrac{1}{2}(\deltag\deltagV+\deltagV\deltag).
\label{eq:divdiv} 
\eeq

We now apply these identities to obtain the Green's formulae associated with the operators appearing in the boundary value problem \eqref{eq:Einstein_boundary_opening}:
\[
\begin{aligned}
& \D_{\Gamma}\Ric_{g}\sigma = T, \qquad
&& \delBianchi \sigma = 0
\qquad && \text{in } M, \\[3pt]
& \PttD \sigma = 0, \qquad
&& \D\Ah_{g}\sigma = 0
\qquad && \text{on } \dM.
\end{aligned}
\]
For the Killing operator and tensor divergence, \eqref{eq:Greens_exterior} gives, for every $X \in \XM$ and $\sigma \in \SM$,
\beq 
\bra \Def X, \sigma \ket
= \bra X, \delBianchi \sigma \ket
- \bra X|_{\dM}, \PnD_{g}\sigma \ket,
\label{eq:Green_forumla_killing} 
\eeq 
where the normal contraction $\PnD_{g} \colon \SM \to \XM|_{\dM}$ and the restriction $(\cdot)|_{\dM}:\XM\rightarrow \XM|_{\dM}$ are identified as
\beq 
\begin{split} 
&\PnD_{g}\sigma = (i_{\dr}\Dng{0}\sigma)^{\sharp_{g}}\simeq(\PnnG\sigma\oplus\PntG\sigma), 
\\& \sigma|_{\dM}=\Dng{0}\sigma\simeq (\PttG\sigma\oplus\PtnG\sigma).
\end{split} 
\label{eq:Pn_res_iso} 
\eeq 

To state the Green's formula for $\D\Ric_{g}$, we introduce the tensorial operations
\beq 
\begin{gathered} 
\mathrm{E}_{g} \psi = -\trace_{g} \psi + \tfrac{1}{2} (\trace_{g} \trace_{g} \psi)\,g, 
\qquad 
\mathrm{C}_{g} \sigma = -\sigma + \trace_{g} \sigma\, g.
\end{gathered} 
\label{eq:E_C_contractions}
\eeq
The map $\mathrm{C}_{g} \colon \SM\rightarrow \SM$ is a symmetric isomorphism whenever $\dim M \geq 2$, while $\mathrm{E}_{g}\colon \Wkm{2}{2}\rightarrow \Wkm{1}{1}$ is nontrivial whenever $\dim{M}\geq 3$. Moreover, these operations give the relation
\beq
\begin{aligned} 
&\Ein_{g} := \mathrm{B}_{g}\Ric_{g}=\mathrm{E}_{g}\Rm_{g}, \qquad \dim{M}\geq 3,
\end{aligned} 
\label{eq:Ein_B_Ric}
\eeq
where $\Ein_{g}\in\SM$ is the Einstein tensor. 

With this notation in place, \eqref{eq:Hg_Green}, together with the relation \eqref{eq:Ein_B_Ric}, gives the desired Green's formula for Ricci curvature:
\begin{proposition}
\label{prop:normal_derivative_DA}
Let $\gamma:=\PttD g$ be the pullback metric. Then for every $\sigma,\eta\in \SM$, the following Green's formula holds: 
\beq 
\bra \mathrm{B}_{g}\D\Ric_{g}\sigma, \eta \ket
= \bra \sigma, \mathrm{B}_{g}\D\Ric_{g}\eta \ket
+ \bra \PttD \sigma, \mathrm{C}_{\gD}\frakT_{g}\eta \ket
- \bra \frakT_{g}\sigma, \mathrm{C}_{\gD}\PttD \eta \ket.
\label{eq:Green_formula_Einstein} 
\eeq
The leading normal derivative of $\frakT_g \colon \SM \rightarrow S^2(\dM)$ is $\jmath^*\Dng{1}$. 
\end{proposition}

The Green's formula \eqref{eq:Green_formula_Einstein} appears in the literature in several notations (see, e.g., \cite[Prop.~2.4]{AH24}). In its classical form, it expresses the fact that the linearized Einstein tensor is formally self-adjoint up to lower-order terms, as also noted in \cite{Hin24} in the case $\Ein_{g}=0$. We defer the proof of the formula to \secref{sec:Einstein_green}, where it will be derived from \eqref{eq:Ein_B_Ric} through a duality between the Riemann curvature tensor and the Einstein tensor, which we take a moment to state here.

For the statement, observe that $g\in\SM$, and hence may be wedged with itself to define the operator
\[
g^{N}=\underbrace{g\wedge\dots\wedge g}_{N\text{ times}}:\Wkm{k}{m}\rightarrow \Wkm{k+N}{m+N}.
\]
\begin{theorem} 
\label{thm:duality_Riem}
The following identity holds when $\dim{M}:=n\geq 3$:
\[
\Ein_{g}=\frac{1}{(n-3)!}\starG \starGV(g^{n-3}\Rm_{g}).
\]
\end{theorem}
We shall prove this theorem in \secref{sec:duality_for_curvature}, as we believe it is not only new but of independent interest, and present further applications there.

Finally, incorporating the fixed tensorial endomorphism term $\Gamma\in\End(\SM)$ from \secref{sec:setup} into \eqref{eq:Green_formula_Einstein}, and observing that such tensorial terms integrate by parts without producing boundary contributions, we obtain
\beq 
\bra \mathrm{B}_{g}\D_{\tGamma}\Ric_{g}\sigma, \eta \ket
= \bra \sigma, \mathrm{B}_{g}\D_{\ttGamma}\Ric_{g}\eta \ket
+ \bra \PttD \sigma, \mathrm{C}_{\gD}\frakT_{g}\eta \ket
- \bra \frakT_{g}\sigma, \mathrm{C}_{\gD}\PttD \eta \ket,
\label{eq:Green_formula_Einstein2} 
\eeq
where
\[
\D_{\ttGamma}\Ric_{g}\sigma:=\D\Ric_{g}\sigma+\ttGamma\sigma, \qquad \ttGamma:=\mathrm{B}_{g}^{-1}\tGamma^*\mathrm{B}_{g}.
\]
Indeed, since $\mathrm{B}_{g}$ is a symmetric isomorphism, we have
\[
(\mathrm{B}_{g}\D_{\tGamma}\Ric_{g})^*=\mathrm{B}_{g}\D\Ric_{g}+\tGamma^*\mathrm{B}_{g},
\]
and, setting $\ttGamma:=\mathrm{B}_{g}^{-1}\Gamma^*\mathrm{B}_{g}$, this becomes
\[
(\mathrm{B}_{g}\D_{\tGamma}\Ric_{g})^*=\mathrm{B}_{g}\D_{\ttGamma}\Ric_{g}=\mathrm{B}_{g}\D\Ric_{g}+\ttGamma.
\]
\subsection{Einstein constraint equations}
\label{sec:Einstein_Constraint} 
%

We next use the framework of double forms, and in particular the boundary projections in \eqref{eq:boundary_tn} together with the operations in \eqref{eq:E_C_contractions}, to formulate an extended system of Einstein constraint equations on the boundary. In the present notation, the standard expressions from the literature (see, e.g., \cite{CK93}) take the form
\beq
\begin{aligned}
&\PnnG\Ein_{g}
= \mathrm{Sc}_{\gD} - |\Ah_{g}|_{\gD}^{2} + (\trace_{\gD} \Ah_{g})^2, \\[0.5em]
&\PtnG\Ein_{g}
= \delta^{V}_{\gD} \Ah_{g} + d\trace_{\gD}\Ah_{g}, \\[0.5em]
&\PttD\Ein_{g}
= \Ein_{\gD} + \mathrm{C}_{\gD}\PnnG\Rm_{g}
+ \tfrac{1}{2} \mathrm{E}_{\gD}(\Ah_{g} \wedge \Ah_{g}).
\end{aligned}
\label{eq:Einstein_constraints_general}
\eeq
When $n>3$, the third constraint in \eqref{eq:Einstein_constraints_general} can equivalently be expressed in terms of the electric part of the Weyl tensor, $\PnnG\mathrm{Wey}_{g}$, as
\beq
(\tfrac{n-3}{n-2})\PttD\Ein_{g}
= \Ein_{\gD} - \PnnG\mathrm{Wey}_{g}
+ \tfrac{1}{2} \mathrm{E}_{\gD}(\Ah_{g} \wedge \Ah_{g}).
\label{eq:electric_part} 
\eeq
In \secref{sec:constraint_duality}, we derive the third constraint in \eqref{eq:Einstein_constraints_general}, as well as its equivalence with \eqref{eq:electric_part}, by means of the duality in \thmref{thm:duality_Riem}. 

These constraint equations will play an important role at several points in the analysis. In particular, they show that an extended Cauchy data set of the metric $g$ on the boundary determine $\Ric_{g}|_{\dM}$. We now take a moment to make this precise. For the analysis, consider the variation of $g\mapsto\PnnG\Rm_{g}$:
\[
\D(\PnnG\Rm_{g})\sigma:=\dertZero\PnnD_{g+t\sigma}\Rm_{g+t\sigma}. 
\]
This defines a second-order trace operator $\D(\PnnG\Rm_{g}):\SM\rightarrow \SdM$. 
\begin{proposition}
\label{prop:normal_trace_weyl}
For every $g \in \scrM(M)$, the boundary operator
\beq 
\PttD \oplus \D\Ah_{g} \oplus \D(\PnnG\Rm_{g})
: \SM \to \SdM \times \SdM \times \SdM
\label{eq:boundary_operator} 
\eeq 
is a normal system of trace operators.
\end{proposition}
\begin{proof}
We show that the leading normal derivatives of \eqref{eq:boundary_operator} are, up to factors, $
\jmath^*\Dng{0}\oplus\jmath^*\Dng{1}\oplus\jmath^*\Dng{2}$.
Since the tensor pullback $\jmath^* \colon \SM|_{\dM} \rightarrow \SdM$ is surjective, this proves that \eqref{eq:boundary_operator} satisfies the criteria of \defref{def:normal_system}.

By \propref{prop:normal_derivative_DA}, the leading normal derivatives of $\PttD \oplus \D\Ah_{g}$ are $\jmath^* \Dng{0} \oplus \jmath^* \Dng{1}$. It remains to identify the leading normal derivative of $\D(\PnnG\Rm_{g})$. By \eqref{eq:DRm} and the chain rule, this is the same as that of $\PnnG\Hg$; a computation shows that its leading normal derivative is $\jmath^*\Dng{2}=\jmath^*\nabg_{\dr}\nabg_{\dr}$ (\cite[p.~756]{KL21a}).
\end{proof}

Now fix a metric $g \in \scrM(M)$, and consider the level set
\[
\scrM^{00}(M,g) := \{h \in \scrM(M) : \PttD h = \gD, \quad \Ah_{h} = \Ah_g, \quad \PnnD_{h}\Rm_{h} = \PnnG\Rm_{g}\}.
\]

By comparing with \eqref{eq:Einstein_constraints_general}, and using the facts that $\Ein_{g}=\mathrm{B}_{g}\Ric_{g}$ and that $\mathrm{B}_{g}$ is a tensorial isomorphism, we obtain
\[
h \in \scrM^{00}(M,g) \Longrightarrow 
\Ric_{h}|_{\dM} = \Ric_{g}|_{\dM}. 
\label{eq:equive_M0_Ric}
\]
We next show that $\scrM^{00}(M,g)$ is a smooth infinite-dimensional submanifold of $\scrM(M)$. For later use, we also record the analogous statement for:
\[
\begin{aligned} 
\scrM^{0}(M,g)
&:= \{h \in \scrM(M)~:~\PttD h = \gD,\quad \Ah_{h} = \Ah_g\}.
\end{aligned} 
\]
\begin{proposition}
\label{prop:submanifolds}
The inclusions $\scrM^{00}(M,g) \subset \scrM^{0}(M,g) \subset \scrM(M)$ define infinite-dimensional smooth Fr\'echet submanifolds, with tangent spaces at each point $h\in \scrM^{00}(M,g)$ given by
\[
\ker(\PttD, \D\Ah_{h}, \D\PnnD_{h}\Rm_{h})
\subset
\ker(\PttD, \D\Ah_{h})
\subset
\SM.
\]
\end{proposition}
\begin{proof}
For conciseness, we invoke the Nash--Moser implicit function theorem rather than its elementary Banach-space counterpart. The spaces $\scrM^{0}(M,g)$ and $\scrM^{00}(M,g)$ are level sets of the maps
\[
h\mapsto(\PttD h,\Ah_{h}), \qquad h \mapsto (\PttD h, \Ah_{h}, \PnnD_{h}\Rm_{h}).
\]
These are nonlinear differential operators between spaces of sections of vector bundles, and hence are tame smooth maps~\cite[Cor.~2.2.7]{Ham82}. By \propref{prop:normal_trace_weyl}, their linearizations are normal systems of trace operators. Therefore, by \propref{prop:normality}, they are surjective at every $h \in \scrM(M)$ and admit smooth tame right inverses. It follows from \cite[Thm.~1.1.3]{Ham82} that the maps above are submersions, and hence their level sets are smooth submanifolds, with tangent spaces given by the associated kernels.

Finally, these submanifolds are clearly infinite-dimensional. Indeed, by \propref{prop:normality}, the kernels above are $L^{2}$-dense in $L^{2}\SM$; and the level sets are nonempty, since they contain $g$.
\end{proof}
\section{Hodge theory for linearized Ricci curvature}
\label{sec:Hodge_Einstein}
\subsection{The elliptic pre-complex}
For every $g\in\scrM(M)$ and $\tGamma\in\End(\SM)$, consider now the following diagram:
\beq
\begin{xy}
(-30,0)*+{0}="Em1";
(-15,0)*+{\XM}="E0";
(15,0)*+{\SM}="E1";
(45,0)*+{\SM}="E2";
(75,0)*+{\XM}="E3";
(90,0)*+{0}="E4";
(-15,-25)*+{\XM|_{\dM}}="G0";
(15,-25)*+{\SdM\oplus\SdM}="G1";
(45,-25)*+{\XM|_{\dM}}="G2";
(75,-25)*+{0}="G3";
{\ar@{->}@/^{1pc}/^{\Def}"E0";"E1"};
{\ar@{->}@/^{1pc}/^{\delBianchi}"E1";"E0"};
{\ar@{->}@/^{1pc}/^{\mathrm{B}_{g}\D_{\tGamma}\Ric_{g}}"E1";"E2"};
{\ar@{->}@/^{1pc}/^{\mathrm{B}_{g}\D_{\ttGamma}\Ric_{g}}"E2";"E1"};
{\ar@{->}@/^{1pc}/^{\delBianchi}"E2";"E3"};
{\ar@{->}@/^{1pc}/^{\Def}"E3";"E2"};
{\ar@{->}@/^{0.5pc}/^{0}"E3";"E4"};
{\ar@{->}@/^{0.5pc}/^{0}"E4";"E3"};
{\ar@{->}@/^{0.5pc}/^{0}"Em1";"E0"};
{\ar@{->}@/^{0.5pc}/^{0}"E0";"Em1"};
{\ar@{->}@/_{0pc}/^{|_{\dM}}"E0";"G0"};
{\ar@{->}@/_{0pc}/^{\PttD \oplus \frakT_{g}}"E1";"G1"};
{\ar@{->}@/_{0pc}/^{\PnD_{g}}"E2";"G2"};
{\ar@{->}@/^{0.7pc}/^{\PnD_{g}}"E1";"G0"};
{\ar@{->}@/^{0.7pc}/^{\cdots}"E2";"G1"};
{\ar@{->}@/^{1pc}/^{|_{\dM}}"E3";"G2"};
{\ar@{->}@/_{0pc}/^{0}"E3";"G3"};
\end{xy}
\label{eq:disrupted_complex_diagram}
\eeq

In view of the preceding analysis, this diagram fits into the framework of \eqref{eq:elliptic_complex_diagram}. Our main result concerning it is the following.
\begin{proposition}
\label{prop:Eisntein_disrupted}
Let $(M,g)$ be a compact Riemannian manifold with boundary and let $\tGamma\in\End(\SM)$ be a tensorial endomorphism. Then:
\begin{enumerate}[label=(\alph*), itemsep=2pt]
\item The diagram \eqref{eq:disrupted_complex_diagram} is a $1$-elliptic pre-complex in the sense of \defref{def:elliptic_pre_complex}. 
\item If, in addition, $\tGamma\in\End(\SM)$ satisfies the condition
\beq 
\PnD_{g}\mathrm{B}_{g}\D_{\tGamma}\Ric_{g}\sigma=0, \qquad \text{for } \qquad \sigma\in\ker(\PttD,\D\Ah_{g}),
\label{eq:vanishing_RicBody} 
\eeq 
then the diagram \eqref{eq:disrupted_complex_diagram} is a disrupted $3$-elliptic pre-complex in the sense of \defref{def:disrupted_pre_complex}.
\item If, in addition to (b), $\dim M = 3$ or $\dM=\emptyset$, the diagram \eqref{eq:disrupted_complex_diagram} is a $3$-elliptic pre-complex in the sense of \defref{def:elliptic_pre_complex}.
\end{enumerate}
\end{proposition}
We take a moment to note that, as discussed in the introduction around \eqref{eq:compatibilityRic}, by applying the chain rule to the constraint equations \eqref{eq:Einstein_constraints_general}, one finds that \eqref{eq:vanishing_RicBody} reduces to the following conditions for every $\sigma\in\SM$: 
\beq
\begin{aligned}
\PnnD_{g}\mathrm{B}_{g}\tGamma\sigma
&= \dertZero \PnnD_{g+t\sigma}\mathrm{B}_{g+t\sigma}(\Ric_{g}|_{\dM}), \\
\PtnD_{g}\mathrm{B}_{g}\tGamma\sigma
&= \dertZero \PtnD_{g+t\sigma}\mathrm{B}_{g+t\sigma}(\Ric_{g}|_{\dM})
= \dertZero \PtnD_{g+t\sigma}(\Ric_{g}|_{\dM}).
\end{aligned}
\label{eq:vanishing_RicBody2} 
\eeq 
In the latter identity, we used the fact that  $\PtnD_{g+t\sigma}\mathrm{B}_{g+t\sigma}=\PtnD_{g+t\sigma}$, which follows by expanding $\mathrm{B}_{g+t\sigma}=\id-\tfrac{1}{2}\trace_{g+t\sigma}(\cdot)(g+t\sigma)$ and using $\PtnD_{g+t\sigma}(g+t\sigma)=0$, as the metric has no mixed tangential-normal components. 

Since all operators here are tensorial, this condition indeed depends only on $\Ric_{g}|_{\dM}$ and $\sigma|_{\dM}$. We remark that these expressions can be further expanded and simplified by opening the variation formulae for the trace and normal vector field, but this is not required for our purposes.

We postpone the proof to \secref{sec:proof_pre_complex} and first discuss the consequences. Since the diagram \eqref{eq:disrupted_complex_diagram} has only three nontrivial segments, the lifted complex provided by \thmref{thm:lifted_complex} has the same form in all cases covered by \propref{prop:Eisntein_disrupted}. Thus, for every $g \in \scrM(M)$, and independently of the case under consideration, the lifted complex is well defined and given by
\beq
\begin{aligned}
&\bA_{0}=\Def \colon \XM \rightarrow \SM, \\
&\bA_{1}=\mathpzc{T}_{g} \colon \SM \rightarrow \SM, \\
&\bA_2=\pzcdel_{g} \colon \SM \rightarrow \XM.
\end{aligned}
\label{eq:lifted_eind}
\eeq
The defining relations for the lifted operators in \thmref{thm:lifted_complex} then become
\beq
\begin{aligned}
\text{for } \alpha = 1:\qquad
& \mathpzc{T}_{g} \Def = 0
&&\qquad \text{on } \ker(\cdot|_{\dM}), \\
&
\mathpzc{T}_{g}
= \mathrm{B}_{g}\D_{\tGamma}\Ric_{g}
&& \qquad\text{on } \ker (\delBianchi), \\[0.2em]
\text{for } \alpha = 2:\qquad
& \pzcdel_{g} \mathpzc{T}_{g} = 0
&&\qquad \text{on } \ker(\PttD, \D \Ah_{g}), \\
&
\pzcdel_{g}
= \delBianchi
&& \qquad\text{on } \ker (\mathpzc{T}_{g}^{*}).
\end{aligned}
\label{eq:lifted_ricci_relations}
\eeq
Although we omit a subscript $\Gamma$ for conciseness, we emphasize that these operators also depend on the choice of $\Gamma$. 

%

The Hodge decompositions furnished by \thmref{thm:Hodge_decomposition} now take the following explicit forms, according to the three settings in \propref{prop:Eisntein_disrupted}. 

\textbf{Case $(a)$: Hodge theory for the $1$-elliptic pre-complex.}

For every $g \in \scrM(M)$, the $1$-elliptic pre-complex yields the decompositions
\beq
\begin{aligned}
&\alpha = -1:\quad
&& \XM
= \scrE^{0}_{M}(g,\tGamma)\oplus\image{(\delBianchi)},
\\[0.2em]
&\alpha = 0:\quad
&& \SM
=
\lefteqn{\overbrace{\phantom{\image{\Def|_{\ker (\cdot)|_{\dM}}}\oplus \scrE^{1}(M,g,\tGamma)}}^{\ker(\mathpzc{T}_{g}|_{\ker(\PttD,\frakT_{g})})}}
\image{(\Def|_{\ker (\cdot)|_{\dM}})}
\oplus
\underbrace{\scrE^{1}(M,g,\tGamma) \oplus \image{(\mathpzc{T}_{g}^*})}_{\ker(\delBianchi)} .
\end{aligned}
\label{eq:Hodge_1} 
\eeq
Here we denote the cohomology modules by the script letter $\scrE$, rather than by $\module$, to indicate the Einstein setting. We also include the parameters $(M,g,\tGamma)$ in the notation, to emphasize their dependence on $M$ and on the choices of $g$ and $\tGamma$.

Since these are the first two segments of the complex, the corresponding cohomology modules in \eqref{eq:cohomology_expression} then reduce, by \eqref{eq:lower_cohomolgies}, to
\beq
\begin{aligned}
\scrE^{0}(M,g,\tGamma) &= \ker(\delta_g^*,(\cdot)|_{\dM}),
\\\scrE^{1}(M,g,\tGamma) &= \ker(\mathrm{B}_{g}\D_{\Gamma}\Ric_{g}, \delBianchi, \PttD, \frakT_{g}) \\
&= \ker(\D_{\Gamma}\Ric_{g}, \delBianchi, \PttD, \D\Ah_{g}).
\end{aligned}
\label{eq:cohomology_Einstein_bianchi1}
\eeq
In the second equality we have used that $\mathrm{B}_{g}$ is an isomorphism, and that $\ker(\PttD,\frakT_{g})=\ker(\PttD,\D\Ah_{g})$ by \eqref{eq:DA}; hence the two kernels coincide.

\textbf{Case $(b)$: Hodge theory for the disrupted $3$-elliptic pre-complex.}

If, in addition, $\tGamma$ satisfies \eqref{eq:vanishing_RicBody}, then the disrupted $3$-elliptic pre-complex yields, by \propref{prop:disrupted_valid}, the further decompositions
\beq
\begin{aligned}
&\alpha = 1:\quad
&& \SM
=
\lefteqn{\overbrace{\phantom{\image{\mathpzc{T}_{g}|_{\ker(\PttD,\D\Ah_{g})}}\oplus \scrE^{2}(M,g,\tGamma)}}^{\ker(\pzcdel_{g}|_{\ker\PnD_{g}})}}
\image{(\mathpzc{T}_{g}|_{\ker(\PttD,\D\Ah_{g})})}
\oplus
\underbrace{\scrE^{2}(M,g,\tGamma) \oplus \image{(\pzcdel_{g}^*})}_{\ker (\mathpzc{T}_{g}^*)},
\\[0.6em]
&\alpha = 2:\quad
&& \XM
=
\image{(\pzcdel_{g}|_{\ker(\PnD_{g})})}
\oplus
\scrE^{3}(M,g,\tGamma).
\end{aligned}
\label{eq:Hodge_decompostions_Ricci}
\eeq
By \propref{prop:disrupted_valid}, the module
\[
\scrE^{3}(M,g,\tGamma) = \ker (\pzcdel_{g}^*)
\]
is finite-dimensional. By contrast, the module
\beq 
\scrE^{2}(M,g,\tGamma)=\ker(\mathpzc{T}_{g}^*,\pzcdel_{g},\PnD_{g})
\label{eq:cohomology_2} 
\eeq 
need not be finite-dimensional in general, since it lies at the penultimate level of the disrupted complex. Nevertheless, let
\[
\tbP_{\g}:\SM\rightarrow \SM
\] 
denote the $L^{2}$-orthogonal projection onto $\ker\delBianchi$ in the Hodge decomposition \eqref{eq:Hodge_1}. By \thmref{thm:Hodge_decomposition}, this projection is a Green operator of order and class zero. Then \propref{prop:coadjoint} gives the reduced expression
\beq
\begin{split}
\scrE^{2}(M,g,\tGamma)&= \ker(\delBianchi,\tbP_{g}\mathrm{B}_{g}\D_{\ttGamma}\Ric_{g}, \PnD_{g}). 
\end{split}
\label{eq:cohomology_Einstein_bianchi2}
\eeq

\medskip
\noindent
\textbf{Case $(c)$: Hodge theory for the $3$-elliptic pre-complex.}

If, in addition, $\dim M = 3$ or $\dM=\emptyset$, then \propref{prop:Eisntein_disrupted} shows that the diagram defines a genuine elliptic pre-complex, since the conditions of \defref{def:elliptic_pre_complex} are fully satisfied. Consequently, the decompositions in \eqref{eq:Hodge_decompostions_Ricci} remain valid verbatim, and the module $\scrE^{2}(M,g,\tGamma)$ is finite-dimensional.
\subsection{Proofs of main results in \secref{sec:main_results_intro}}
\label{sec:main_proof}
By means of the isomorphism $\mathrm{B}_{g}^{-1} \colon \SM \to \SM$ (we remind that $\dim M \geq 3$), we can now define
\beq
\begin{aligned}
\pzcD_{\tGamma} \mathpzc{Ric}_{g}
:= \mathrm{B}_{g}^{-1}\mathpzc{T}_{g}
\colon \SM \to \SM.
\end{aligned}
\label{eq:lifted_ricci}
\eeq
We retain $\pzcdel_g \colon \SM \to \XM$ as defined in \eqref{eq:lifted_operators}. We show that these definitions of $\pzcD_{\tGamma} \mathpzc{Ric}_{g}$ and $\pzcdel_{g}$ now yield \thmref{thm:lifting_intro}.
\begin{PROOF}{\thmref{thm:lifting_intro}}
The definition in \eqref{eq:lifted_ricci}, together with the relations in \eqref{eq:lifted_ricci_relations}, translates directly into the required properties of $\pzcD_{\tGamma}\mathpzc{Ric}_{g}$ and $\pzcdel_{g}$ stated in \thmref{thm:lifting_intro}, namely
\[
\begin{aligned}
&\;\pzcD_{\Gamma}\pzcDRic \Def = 0
&& \quad \text{on } \quad
&& \ker(\cdot|_{\dM}), \\[1pt]
&\;\pzcD_{\Gamma}\pzcDRic := \D_{\Gamma}\Ric_{g}
&& \quad \text{on } \quad
&& \ker (\delBianchi), \\[3pt]
&\;\pzcdel_{g} \pzcD_{\Gamma}\pzcDRic = 0
&& \quad\text{on }\quad
&& \ker(\PttD, \D \Ah_{g}), \\[1pt]
&\;\pzcdel_{g}\mathrm{B}_{g} := \delBianchi\mathrm{B}_{g}
&& \quad\text{on }\quad
&& \ker(\pzcD_{\Gamma}\pzcDRic^{*}).
\end{aligned}
\]
The uniqueness clause is inherited from the corresponding uniqueness statement in \thmref{thm:lifted_complex} for $\mathpzc{T}_{g}$ and $\pzcdel_{g}$ due to the fact that $\mathrm{B}_{g}$ is an isomorphism, and $\ker(\PttD,\frakT_{g})=\ker(\PttD,\D\Ah_{g})$ by \eqref{eq:DA}. 
\end{PROOF}

The remaining proofs are now consequences of the Hodge decompositions produced in \secref{sec:Hodge_Einstein}.
\begin{PROOF}{\thmref{thm:Saint_Venant_opening}}
The Hodge decomposition \eqref{eq:Hodge_1} implies that
\[
\ker(\mathpzc{T}_{g}|_{\ker(\PttD, \frakT_{g})})=\image\bigl(\Def|_{\ker(\cdot|_{\dM})}\bigr)\oplus \scrE^{1}(M,g,\tGamma).
\]
Since $\mathrm{B}_{g}$ is an isomorphism, the equation $\mathpzc{T}_{g}\sigma = 0$ is equivalent, by \eqref{eq:lifted_ricci}, to $\pzcD_{\tGamma}\mathpzc{Ric}_{g}\sigma = 0$. Moreover, \eqref{eq:DA} gives $\ker(\PttD,\frakT_{g})=\ker(\PttD,\D\Ah_{g})$. Hence
\[
\ker(\mathpzc{T}_{g}|_{\ker(\PttD, \frakT_{g})}) = \ker(\pzcD_{\tGamma} \mathpzc{Ric}_{g}|_{\ker(\PttD, \D\Ah_{g})}). 
\]
Comparing this identity with the expression in \eqref{eq:cohomology_Einstein_bianchi1} completes the proof.

\end{PROOF}

\begin{PROOF}{\thmref{thm:Saint_Venant_opening2}}
Applying \thmref{thm:cohomology_dirichlet} with $\alpha = 0$ to the $1$-elliptic pre-complex shows that, for a given $\sigma \in \SM$, the system
\[
\begin{aligned}
&\Def X = \sigma \qquad  &&\text{in } M, \\
&X|_{\dM} = 0 
\qquad &&\text{on } \dM
\end{aligned}
\]
admits a solution if and only if
\[
\mathpzc{T}_{g}\sigma = 0, \qquad
\PttD\sigma = 0, \qquad \frakT_{g}\sigma=0, \qquad
\sigma \perp_{L^2} \scrE^{1}(M,g,\tGamma).
\]
As above, the condition $\mathpzc{T}_{g}\sigma = 0$ is equivalent to $\pzcD_{\tGamma}\mathpzc{Ric}_{g}\sigma = 0$ by \eqref{eq:lifted_ricci}, while $\ker(\PttD,\frakT_{g})=\ker(\PttD,\D\Ah_{g})$ by \eqref{eq:DA}, yielding the asserted criterion.

\end{PROOF}

\begin{PROOF}{\thmref{thm:Einstein_boundary_opening}} 
Substituting the definition of the lifted Ricci operator \eqref{eq:lifted_ricci}, the Hodge decomposition in \eqref{eq:Hodge_decompostions_Ricci} gives
\[
\begin{split} 
\ker(\pzcdel_{g}|_{\ker \PnD_{g}})=\mathrm{im}\bigl(\mathrm{B}_{g}\pzcD_{\tGamma}\pzcDRic|_{\ker(\PttD,\D\Ah_{g})}) \oplus\scrE^{2}(M,g,\tGamma). 
\end{split} 
\]
The required explicit expression for $\scrE^{2}(M,g,\tGamma)$ is the one obtained in \eqref{eq:cohomology_Einstein_bianchi2}. 

TThe solvability statement then follows directly, since \thmref{thm:cohomology_dirichlet} remains applicable to the disrupted elliptic pre-complex by \propref{prop:disrupted_valid}. In terms of the lifted operators in \eqref{eq:lifted_ricci_relations}, given $\tilde{T} \in \SM$, the system
\beq
\begin{aligned}
&\mathrm{B}_{g}\D_{\tGamma} \Ric_{g} \sigma = \tilde{T}, \qquad
&&\delBianchi\sigma = 0
\qquad &&\text{in } M, \\[0.3em]
&\PttD \sigma = 0, \qquad
&&\frakT_g \sigma = 0,
\qquad &&\text{on } \dM,
\end{aligned}
\label{eq:Einstein_boundary_Covarian_proof}
\eeq
admits a solution $\sigma \in \SM$ if and only if
\[
\pzcdel_{g}\tilde{T} = 0, \qquad
\PnD_{g}\tilde{T} = 0, \qquad
\tilde{T} \perp_{L^2} \scrE^{2}(M,g,\tGamma).
\]
Since $\ker(\PttD,\frakT_{g})=\ker(\PttD,\D\Ah_{g})$, substituting $\tilde{T}=\mathrm{B}_{g}T$ for $T\in \SM$, as in the statement of the theorem, gives the desired result.
\end{PROOF} 
\subsection{Proof of \propref{prop:Eisntein_disrupted}} \label{sec:proof_pre_complex}
The proof that the diagram \eqref{eq:disrupted_complex_diagram} is an elliptic pre-complex, in the sense specified by \propref{prop:Eisntein_disrupted}, is organized in several steps, verifying the ingredients required by \defref{def:elliptic_pre_complex} and \defref{def:disrupted_pre_complex} one by one. 
\begin{proposition}
The operators in the diagram \eqref{eq:disrupted_complex_diagram} satisfy the Green's formulae and the normality conditions required in part (a) of \defref{def:elliptic_pre_complex}.
\end{proposition}
\begin{proof}
By the definition of a normal system of trace operators in \defref{def:normal_system}, it suffices, for the verification of normality, to identify the leading normal derivative of each summand in the relevant boundary systems.

We begin with the level $\alpha=0$. Reading from the diagram \eqref{eq:disrupted_complex_diagram}, the corresponding boundary operators are
\[
B_0 = (\cdot)|_{\dM} = \Dng{0} 
\Textand
B_0^* = \PnD_{g} = (i_{\dr}\Dng{0})^{\sharp_g}.
\]
The relevant Green's formula is \eqref{eq:Green_forumla_killing}. These boundary systems are normal of class $r=1$, since $\id$ and $i_{\dr}$ is surjective.

At the level $\alpha=1$, the boundary operators appearing in \eqref{eq:disrupted_complex_diagram} are
\[
B_1 = \PttG \oplus \frakT_g
\Textand
B_1^* = (-\mathrm{C}_{\gD}\frakT_g) \oplus \mathrm{C}_{\gD}\PttG.
\]
By \propref{prop:normal_derivative_DA}, their leading normal terms are
\[
\jmath^*\Dng{0} \oplus
\jmath^*\Dng{1} 
\Textand
\mathrm{C}_{\gD}\jmath^*\Dng{0} \oplus
\mathrm{C}_{\gD}\jmath^*\Dng{1}. 
\]
Both systems are normal of class $r=2$, again by the surjectivity of the boundary pullback, together with the fact that $\mathrm{C}_{\gD}$ is an isomorphism.

Finally, we note that the normality conditions and Green's formula required at the level $\alpha=2$ are identical to the those at $\alpha=0$, with the roles of $B_0$ and $B_0^*$ interchanged.
\end{proof} 

\begin{proposition}
\label{prop:order_reduction_Einstein}
In the setting of \propref{prop:Eisntein_disrupted}, the diagram \eqref{eq:disrupted_complex_diagram} satisfies the order-reduction properties required in part (b) of \defref{def:elliptic_pre_complex}. 
\end{proposition}
\begin{proof} 
The only nontrivial levels are $\alpha = 1$ and $\alpha = 2$; the remaining levels are padded by zero operators on one side or the other, so the claim there holds trivially.

\textbf{Order reductions for case (a) in \propref{prop:Eisntein_disrupted}.} 

For every $g\in\scrM(M)$, the required order-reduction property at the level $\alpha=1$ amounts to
\[
\ord(\mathrm{B}_{g}\D_{\tGamma}\Ric_{g}\Def) \leq \ord(\Def), 
\qquad 
\ker(\cdot|_{\dM})\subseteq \ker((\PttD\oplus\frakT_{g})\Def). 
\]
By \eqref{eq:DA}, the inclusion on the right is equivalent to
\beq 
\ker(\cdot|_{\dM})\subseteq \ker((\PttD\oplus\D\Ah_{g})\Def),
\label{eq:second_fudnemental_Lie} 
\eeq 
which follows from the linearized symmetries of the boundary data. Although this can be found in the literature, e.g., \cite[pp.~5--6]{And08}, we recall the argument for completeness. Consider the nonlinear map
\[
g\mapsto(\PttG g,\Ah_{g}).
\]
This map is invariant under diffeomorphisms $\varphi: M \to M$ satisfying $\varphi|_{\dM} = \id$, namely
\beq
(\PttG\varphi^*g,\Ah_{\varphi^*g}) = (\PttG g,\Ah_{g}).
\label{eq:boundary_symmetries_intro} 
\eeq
The infinitesimal generator of a family $\varphi_t:M\rightarrow M$ of such diffeomorphisms is a vector field $X\in\XM$ satisfying $X|_{\dM}=0$. Linearizing the invariance relation \eqref{eq:boundary_symmetries_intro} at $t=0$ gives
\[
(\PttG \oplus \D\Ah_{g})(\dertZero \varphi_{t}^* g) = 0
\]
for every such vector field. Since
\[
\dertZero \varphi_t^* g  = \calL_X g = 2\Def{X}
\]
by the definition of the Killing operator, this proves the inclusion
\[
\ker((\cdot)|_{\dM})\subseteq\ker((\PttG\oplus\D\Ah_{g})\Def).
\]

It remains to verify the order inequality. Since $\tGamma$ and $\mathrm{B}_{g}$ are tensorial, it suffices to show that
\[
\ord(\D\Ric_{g}\Def)\leq \ord(\Def).
\]
This is also known (e.g., \cite[Sec.~5]{Ham84}), but we again recall the argument for completion. The Ricci curvature is equivariant under pullback by local diffeomorphisms $\phi : M \to M$:
\[
\Ric_{\phi^*g}=\phi^*\Ric_{g}.
\]
Since $\Def{X} = \tfrac{1}{2} \calL_{X} g$, linearizing this identity along a one-parameter family of local diffeomorphisms $\phi_{t}$ generated by $X$ gives
\[
\D\Ric_{g} \Def{X} = \tfrac{1}{2} \calL_{X} \Ric_{g}.
\]
Thus, although the operator $X\mapsto\D\Ric_g\Def{X}$ is a priori of third order, it agrees with the first-order operator $X\mapsto\tfrac{1}{2}\calL_{X} \Ric_g$. Hence
\[
\ord(\D\Ric_g\Def)\leq \ord(\Def),
\]
as required.

\textbf{Order reductions for cases (b)--(c) in \propref{prop:Eisntein_disrupted}.} 

At the level $\alpha=2$, the required order-reduction property becomes
\[
\ord(\delBianchi\mathrm{B}_{g}\D_{\tGamma}\Ric_{g}) \leq \ord(\mathrm{B}_{g}\D_{\tGamma}\Ric_{g}), 
\qquad 
\ker(\PttD,\frakT_{g}) \subseteq \ker(\PnD_{g}\mathrm{B}_{g}\D_{\tGamma}\Ric_{g}).
\]
Again, since $\tGamma$ and $\mathrm{B}_{g}$ are tensorial, the order inequality reduces to proving
\[
\ord(\delBianchi\mathrm{B}_{g}\D\Ric_{g}) \leq \ord(\D\Ric_{g}).
\]
This follows directly by linearizing the differential Bianchi identity \eqref{eq:gauge_equivariance} and applying the chain rule; see again \cite[Sec.~5]{Ham84}. Indeed, if $g_{t}\in\scrM(M)$ is a family of metrics with $g_0=g$ and $\dertZero g_{t}=\sigma$, then linearizing
\[
\delta_{g_{t}}\mathrm{B}_{g_t}\Ric_{g_{t}}=0
\]
gives
\[
\delBianchi\mathrm{B}_{g}\D\Ric_{g}\sigma
= -(\dertZero\delta_{g_{t}}\mathrm{B}_{g_t}\Ric_g).
\]
The right-hand side is a first-order differential operator in $\sigma$, whereas $\D\Ric_{g}$ is of second order.

We now turn to the boundary inclusion. By \eqref{eq:DA}, it is enough to prove
\[
\ker(\PttD,\D\Ah_{g}) \subseteq \ker(\PnD_{g}\mathrm{B}_{g}\D_{\tGamma}\Ric_{g})
= \ker(\PnnG\mathrm{B}_{g}\D_{\tGamma}\Ric_{g}, \PntG\mathrm{B}_{g}\D_{\tGamma}\Ric_{g}),
\]
where, as in \eqref{eq:Pn_res_iso}, we have identified $\PnD_{g}$ with $\PnnG\oplus\PntG$. We prove this under the assumption \eqref{eq:vanishing_RicBody} on $\tGamma$.

Recall the submanifold $\scrM^{0}(M,g)$ from \propref{prop:submanifolds}, and let $g_{t} \in \scrM^{0}(M,g)$ be a smooth family such that $g_{0}=g \in \scrM^{00}(M,g)\subset \scrM^{0}(M,g)$ and $\dertZero g_{t}=\sigma \in \ker(\PttD, \D\Ah_{g})$. Expanding
\[
\D_{\tGamma}\Ric_{g}=\D\Ric_g+\tGamma
\]
and applying the chain rule, we obtain
\[
\begin{split}
\PnnD_{g}\mathrm{B}_{g}\D_{\tGamma}\Ric_{g}\sigma
&= \PnnD_{g}\mathrm{B}_{g}\D\Ric_{g}\sigma+\PnD_{g}\mathrm{B}_{g}\tGamma\sigma \\
&= \dertZero(\PnnD_{g_{t}}\Ein_{g_{t}})
- \dertZero(\PnnD_{g+t\sigma}\mathrm{B}_{g+t\sigma}\Ric_{g})
+\PnnD_{g+t\sigma}\mathrm{B}_{g+t\sigma}\tGamma\sigma \\
&= 0.
\end{split}
\]
Here we have used the assumption \eqref{eq:vanishing_RicBody}, which gives
\[
\dertZero(\PnnD_{g+t\sigma}\mathrm{B}_{g+t\sigma}\Ric_{g})
= \PnnD_{g+t\sigma}\mathrm{B}_{g+t\sigma}\tGamma\sigma,
\]
and also the fact that, since $g_{t}\in\scrM^{0}(M,g)$, the Einstein constraint equations \eqref{eq:Einstein_constraints_general} imply
\[
\dertZero\PnnD_{g_{t}}\Ein_{g_{t}}
= \dertZero (\mathrm{Sc}_{\gD} - |\Ah_{g}|_{\gD}^{2} + (\trace_{\gD} \Ah_{g})^2)=0.
\]
The proof that $\PtnG\mathrm{B}_{g}\D_{\tGamma}\Ric_{g}\sigma=0$ is identical, using the corresponding $\PtnG$-condition. This proves the required inclusion.
\end{proof}

We next establish the required overdetermined ellipticities.
\begin{proposition}
\label{prop:overdetermined_elliptiicty_einstein}
For every $g \in \scrM(M)$, the overdetermined ellipticities required in part {(c)} of \defref{def:elliptic_pre_complex} and in \propref{prop:Eisntein_disrupted} are satisfied for the diagram \eqref{eq:disrupted_complex_diagram}, possibly except at the level $\alpha=2$.
\end{proposition}
\begin{proof}
For $\alpha=0$ and $\alpha = 3$, the would-be overdetermined elliptic systems are
\[
A_{0}\oplus B_{0}=\Def\oplus(\cdot)|_{\dM}, \qquad A_{2}^*=\Def,
\]
where $A_{2}^*$ is not supplemented by additional adjoints or boundary terms, since $A_{\alpha}=0$ when either $\alpha \geq 3$ or $\alpha=-1$. These systems are overdetermined elliptic because the Killing operator is overdetermined elliptic even without boundary conditions. More concretely, following the equivalence of overdetermined ellipticity to a-priori estimates \secref{sec:OD}, this is precisely the content of Korn's inequality:
\[
\|X\|_{s+1} \leq C \bigl(\|\Def X\|_{s} + \|X\|_{0} \bigr), \qquad s\in\Nzero.
\]

Since $\alpha=2$ is the disrupted level, it remains only to address $\alpha = 1$. There, the would-be overdetermined elliptic system is
\[
A_{1}\oplus A_{0}^*\oplus B_{1}
=
\mathrm{B}_{g}\D_{\tGamma}\Ric_{g}\oplus\delBianchi\oplus(\PttD\oplus\frakT_{g}).
\]
Because $\mathrm{B}_{g}$ and $\Gamma$ are tensorial, this system is equivalent, for ellipticity, to
\beq 
\mathrm{B}_{g}\D\Ric_{g}\oplus\delBianchi\oplus(\PttD\oplus\frakT_{g}).
\label{eq:overdetermined_Ricci} 
\eeq
The overdetermined ellipticity of \eqref{eq:overdetermined_Ricci} can be established in several ways. One may verify the Lopatinskii--Shapiro condition directly, following the technical analysis outlined in \cite[Sec.~5.5]{KL23} and \cite[Ch.~4.1]{Led25B}; see also the computations in \cite[Prop.~3.1]{And08}.

Here, instead, we use the duality between the Einstein tensor and the Riemann tensor stated in \thmref{thm:duality_Riem}. Since the verification is rather technical, and is most naturally carried out in the language of double forms, we defer the computation to \secref{sec:ellipticity}.
\end{proof}

The remaining case in part (c) of \propref{prop:Eisntein_disrupted} is addressed by the following.
\begin{proposition}
\label{prop:disrupted_OD_einstein}
When either $\dim M = 3$ or $\dM=\emptyset$, the system
\[
A_{2}\oplus A_{1}^*\oplus B_{2}=\delBianchi\oplus\mathrm{B}_{g}\D_{\ttGamma}\Ric_{g} \oplus \PnD_g
\]
is overdetermined elliptic.
\end{proposition}

Unless $\dim M = 3$, a count of degrees of freedom shows that this system does not supply enough boundary equations for the problem to be determined, let alone overdetermined: the operator $\PnD_{g}$ contributes only $n:=\dim M$ degrees of freedom, whereas the unknown lies in $\SM$ and so has $\tfrac{n(n+1)}{2}$ degrees of freedom. Consequently, in general, \eqref{eq:disrupted_complex_diagram} does not define an elliptic pre-complex in the strict sense, but rather a disrupted one.

\begin{proof}
Since $\mathrm{B}_{g}$ and $\Gamma$ are tensorial and $\mathrm{B}_{g}\D\Ric_{g}$ is formally self-adjoint due to \eqref{eq:Green_formula_Einstein}, it suffices to check the overdetermined ellipticity of
\[
\delBianchi \oplus \mathrm{B}_{g}\D\Ric_{g} \oplus \PnD_g.
\]
When $\dM=\emptyset$, the boundary operators vanish and the statement reduces to the overdetermined ellipticity of $\delBianchi \oplus \mathrm{B}_{g}\D\Ric_{g}$ (since $\D\Ric_{g}$ is formally self-adjoint by \eqref{eq:Green_formula_Einstein}), which is established in the previous proposition.

When $\dim M = 3$, since $\D\Ric_g$ determines $\D\Rm_{g}$ up to lower-order terms, the required overdetermined ellipticity is equivalent to that of the system
\[
\delta_{g} \oplus \Hg \oplus \PnD_{g},
\]
where we expanded $\D\Rm_{g}$ as in \eqref{eq:DRm} and neglected lower order terms. A proof that this system is overdetermined elliptic is given in \cite[Sec.~5.5]{KL23}.
\end{proof} 
\section{Cohomology analysis} 
\label{sec:uniqueness_result}

\subsection{First vanishing results}

We begin by recording two vanishing theorems which can be retained from existing results in the literature. 

\begin{proposition}
\label{prop:vanish_killing}
For $M$ with nonempty boundary and any $g \in \scrM(M)$,
\[
\scrE^{0}(M,g,\tGamma) = \ker(\Def,(\cdot)|_{\dM}) = \{0\}.
\]
\end{proposition}
\begin{proof}
The claim follows from the standard fact that if $S \subset M$ is a codimension-one submanifold and $X \in \ker (\Def)$ satisfies $X|_{S} = 0$, then $X=0$. See, for example, \cite[Thm.~8.1.5]{Pet16}.
\end{proof}

The second shows that our framework also recovers the existing uniqueness results for Einstein manifolds with boundary in the literature (cf.~\cite[Prop.~2.4]{And08} or \cite[Prop.~3.3]{AH22}).
\begin{proposition}
\label{prop:vanish_Ricci}
Let $M$ be a manifold with nonempty boundary such that $\pi_{1}(M,\dM)=\{0\}$. Then, for every $g \in \scrM(M)$ satisfying $\Ric_{g}=\lambda g$, and for $\tGamma=-\lambda\id$ with $\lambda\in\mathbb{R}$,
\[
\scrE^{1}(M,g,\tGamma) =\ker(\D\Ric_{g}-\lambda\id,\delBianchi,\PttD,\D\Ah_{g})= \{0\}.
\]
\end{proposition}
\begin{proof}
By either \cite[Prop.~2.4]{And08} or \cite[Prop.~3.3]{AH22}, the assumptions $\pi_{1}(M,\dM)=\{0\}$ and $\Ric_{g}=\lambda g$ imply that if $\sigma \in \SM$ satisfies
\[
\begin{gathered}
\D\Ric_{g}\sigma-\lambda\sigma=0,\\
\PttD\sigma=0, \qquad \D\Ah_{g}\sigma=0,
\end{gathered}
\]
then $\sigma=\Def X$ for some $X \in \XM$ with $X|_{\dM}=0$. We apply this result to $\sigma \in \scrE^{1}(M,g,\tGamma)$ with $\tGamma=-\lambda\id$, which is justified by the explicit expression \eqref{eq:cohomology_Einstein_bianchi1}. The condition $\delBianchi\sigma=0$ then becomes $\delBianchi\Def X=0$. Together with $X|_{\dM}=0$ and the Green's formula \eqref{eq:Green_forumla_killing}, this implies that $\Def{X}=0$ and $X|_{\dM}=0$; hence $X=0$ by the previous proposition. 
\end{proof}

\subsection{Bochner technique for the first cohomology}

Throughout the next two sections, we use commutation formulae between the exterior covariant derivatives and their adjoints \eqref{eq:exterior1}--\eqref{eq:exterior2} with the boundary projections in \eqref{eq:boundary_tn}. These formulae were derived and detailed with care in \cite[Sec.~4]{KL21a}. Although somewhat technical in appearance, they arise from nothing but straightforward applications of the Gauss and Weingarten formulae, together with careful tracking of how covariant derivatives commute with boundary pullback. We therefore omit their statement and derivation here, and reference them directly when needed.

We shall also depart from the notation $\bra\cdot,\cdot\ket$ and write $L^{2}$-couplings of tensor fields explicitly, as is customary in Bochner formulae (cf.~\cite{Pet16,PW21}). In particular, we shall use the following Green's formula for the covariant derivative, valid for arbitrary tensor fields (cf.~\cite[Sec.~2.5, p.~694]{KL21a}):
\beq
\int_{M}({\nabg}^*\omega,\eta)_{g}\,d\Volume_{g}
=\int_{M}(\omega,\nabg\eta)_{g}\,d\Volume_{g}
+\int_{\dM}(\omega,dr\otimes\eta)_{g}\,d\Volume_{\gD}.
\label{eq:Greens_outward_normal}
\eeq

The first Bochner formula we derive is the following.
\begin{proposition}
\label{prop:Bochner1}
Let $g \in \scrM(M)$, $\tGamma\in\End(\SM)$, $\sigma\in\scrE^{1}(M,g,\tGamma)$, and 
\beq
\begin{gathered} 
v:=\PtnG\sigma\in\Omega^{1}(\dM), \qquad
\phi:=\PnnG\sigma\in C^{\infty}(\dM), \\
\tau:=\trace_{g}\sigma\in C^{\infty}(M).
\label{eq:v_phi_tau}
\end{gathered} 
\eeq
Then the following formula holds:
\[
\begin{split} 
&\int_M |\nabg\sigma|_{g}^2+|d\tau|_{g}^{2}+(\mathcal (\calR_{g}+\,\bar{\tGamma})\sigma,\sigma)_{g}\,d\Volume_{g}=
\\&\qquad\qquad
\int_{\dM}
\trace_{\gD}\mathrm{A}_{g}\,\phi^2
+2\,\trace_{\gD}\mathrm{A}_{g}\,|v|_{\gD}^2
+2\,(\Sh_{g}v,v)_{\gD}
\,d\Volume_{\gD}.
\end{split} 
\]
where $\bar{\tGamma}:\SM\rightarrow \SM$ is given by
\beq 
\bar{\tGamma}\sigma:=2\,\tGamma\sigma+\trace_{g}(\tGamma\sigma)g.
\label{eq:Gamma_bar}   
\eeq 
\end{proposition}

As is classical in applications of Bochner formulae, this identity immediately yields a vanishing result whenever the left-hand side is nonnegative and the right-hand side is nonpositive, recovering \thmref{thm:vanishing_intro}:
\begin{PROOF}{\thmref{thm:vanishing_intro}}
Assume that $\mathrm{A}_{g} \leq 0$, so that the shape operator $\Sh_{g}:\plWkm{0}{1}\rightarrow \plWkm{0}{1}$ is non-positive. Then under $\calR_{g}+\tGamma\geq 0$, the two sides of the identity in \propref{prop:Bochner1} have opposite signs, hence must both vanish. Consequently, $\nabg\sigma=0$ and $(\calR_{g}+\,\bar{\tGamma})\sigma=0$ on $M$, while on $\dM$ 
\[
\trace_{\gD}\mathrm{A}_{g}\,\phi^{2}=0
\Textand
\trace_{\gD}\mathrm{A}_{g}\,|v|^{2}_{\gD}=0.
\]

If $(\calR_{g}+\,\bar{\tGamma})|_{p} > 0$ at some point $p \in M$, then $\sigma=0$ at $p$. Since $\nabg\sigma=0$ on all of $M$, parallel transport gives $\sigma=0$ identically. Alternatively, if $\trace_{\gD}\mathrm{A}_{g}|_{q}\neq 0$ at some point $q \in \dM$, then $\phi=0$ and $v=0$ at $q$. Together with the condition $\PttD\sigma=0$, which holds for every $\sigma \in \scrE^{1}(M,g,\tGamma)$, and by unraveling the definitions of $\phi,v$ in \eqref{eq:v_phi_tau}, this implies that $\sigma=0$ at $q$. Since $\sigma$ is parallel, it follows again that $\sigma=0$ throughout $M$.
\end{PROOF}

\begin{PROOF}{\propref{prop:Bochner1}}
By the explicit expression for $\scrE^{1}(M,g,\tGamma)$ in \eqref{eq:cohomology_Einstein_bianchi1}, and since $\delta_{g}\sigma=0$, the Ricci variation formula \eqref{eq:Ricci_variation} shows that the equation $\D_{\Gamma}\Ric_{g}\sigma=0$ reduces to
\[
{\nabg}^*\nabg\sigma-\nabg{d}\trace_{g}\sigma+\calR_{g}\sigma+2\,\tGamma\sigma=0.
\]
Integrating this equation against $\sigma$ gives
\beq
\int_{M}({\nabg}^*\nabg\sigma-\nabg{d}\trace_{g}\sigma+\calR_{g}\sigma+2\,\tGamma\sigma,\sigma)_{g}\,d\Volume_{g}=0.
\label{eq:interior_Bochner_1}
\eeq
Using Green's formula \eqref{eq:Greens_outward_normal}, together with $\delta_{g}\sigma={\nabg}^*\sigma=0$, the interior term produced by integration by parts of $\nabg d\trace_{g}\sigma$ vanishes, leaving: 
\[
\begin{split} 
&\int_M |\nabg\sigma|_{g}^2+(\mathcal (\calR_{g}+2\,\tGamma)\sigma,\sigma)_{g}\,d\Volume_{g}=
\\&\qquad\qquad-\int_{\dM}(\nabg_{\dr}\sigma,\sigma)_{g}\,d\Volume_{\gD}-\int_{\dM}(d\trace_{g}\sigma,i_{\dr}\sigma)_{g}\,d\Volume_{\gD}.
\end{split} 
\]
We now analyze the boundary terms. We first note that, on the boundary, expanding the trace \eqref{eq:trace} in an orthonormal frame with $e_i=\dr$ gives
\beq 
(\trace_{g}\sigma)|_{\dM}
=\sum_{i} i^{V}_{e_i} i_{e_i}\sigma
= i^{V}_{\dr} i_{\dr}\sigma + \sum_{i\ne 1} i^{V}_{e_i} i_{e_i}\sigma
= \PnnG\sigma + \trace_{\gD}(\PttD\sigma).
\label{eq:trace_trick}
\eeq 
We remark that the latter computation will be used repeatedly applied on different tensors through the analysis. 

Now, by decomposing the boundary integrands orthogonally as in \eqref{eq:Pn_res_iso} into their tangential and normal components, fiberwise as in \cite[p.~704]{KL21a}, and using the symmetry of $\sigma$ together with the condition $\PttD\sigma=0$, we obtain
\beq 
\begin{split}
&(\nabg_{\dr}\sigma,\sigma)_{g}
=2(\PtnG\nabg_{\dr}\sigma,\PtnG\sigma)_{\gD}+(\PnnG\nabg_{\dr}\sigma)\cdot(\PnnG\sigma),
\\
&(d\trace_{g}\sigma,i_{\dr}\sigma)_{g}
=(\nabg_{\dr}\trace_{\g}\sigma)(\PnnG\sigma)+(d\PnnG\sigma,\PtnG\sigma)_{\gD}.
\end{split}
\label{eq:orthogonal_boundary_bochner} 
\eeq
Moreover, since $\trace_{g}$ commutes with covariant derivatives, the same computation above in \eqref{eq:trace_trick} applied to $\nabg_{\dr}\sigma$ gives
\beq 
\begin{split} 
\nabg_{\dr}\trace_{g}\sigma=\trace_{g}\nabg_{\dr}\sigma=\PnnG\nabg_{\dr}\sigma+\trace_{\gD}\PttD\nabg_{\dr}\sigma. 
\end{split} 
\label{eq:derivative_trace_boch} 
\eeq
Next, using \eqref{eq:DA}, we have on $\dM$
\[
\PttD\sigma=0, \qquad 
\frakT_{g}\sigma=0, \qquad 
(\delBianchi\sigma)|_{\dM}=0.
\]
Applying $\PttG$ and $\PtnG$ to the last boundary identity, and rewriting the resulting expressions in terms of the notation \eqref{eq:v_phi_tau} by means of the boundary commutation formulae in \cite[pp.~705--706]{KL21a} and the expression for $\frakT_{g}$ in \eqref{eq:frakT_explicit}, we obtain, after algebraic simplification,
\beq
\begin{gathered}
\PttD\nabg_{\dr}\sigma=d^{V}_{\gD}v+(d^{V}_{\gD}v)^{T}+\phi\,\mathrm{A}_{g}, \qquad
\PnnG\nabg_{\dr}\sigma=\,\delta_{\gD}v-(\trace_{\gD}\mathrm{A}_{g})\,\phi, \\
\PtnG\nabg_{\dr}\sigma=-(\trace_{\gD}\Ah_{g})\,v-\Sh_{g}v.
\end{gathered}
\label{eq:boundary_expressions}
\eeq
Recalling that $d^{V}_{\gD}v+(d^{V}_{\gD}v)^{T}=2\Defd{v^{\sharp_{\gD}}}$ by \eqref{eq:Killing_exterior_relations}, we obtain
\[
\trace_{\gD}\bigl(d^{V}_{\gD}v+(d^{V}_{\gD}v)^{T}\bigr) = -2\delta_{\gD}v.
\]
Comparing the first two identities in \eqref{eq:boundary_expressions} with \eqref{eq:derivative_trace_boch} therefore gives
\beq 
(\nabg_{\dr}\trace_{g}\sigma)|_{\dM}=-\delta_{\gD}v.
\label{eq:normal_refined} 
\eeq

Substituting \eqref{eq:boundary_expressions} and \eqref{eq:normal_refined} into the boundary integral above, and using the boundary decompositions in \eqref{eq:orthogonal_boundary_bochner}, we obtain
\[
\begin{split}
&\int_M |\nabg\sigma|_{g}^2+(\mathcal (\calR_{g}+2\,\tGamma)\sigma,\sigma)_{g}\,d\Volume_{g}
=
\\&\qquad\qquad\int_{\dM}
(\trace_{\gD}\mathrm{A}_{g})\,\phi^2
-\phi\,\delta_{\gD}v
+2\,(\trace_{\gD}\mathrm{A}_{g})\,|v|_{\gD}^2
+2\,(\Sh_{g}v,v)_{\gD}
\,d\Volume_{\gD}
\\&\qquad\qquad+\int_{\dM}
\phi\,\delta_{\gD}v-(d\phi,v)_{\gD}
\,d\Volume_{\gD}.
\end{split}
\]
The last integral vanishes by integrating $d$ by parts against $\delta_{\gD}$. Thus,
\[
\begin{split}
&\int_M |\nabg\sigma|_{g}^2+(\mathcal (\calR_{g}+2\,\tGamma)\sigma,\sigma)_{g}\,d\Volume_{g}
=
\\&\qquad\qquad\int_{\dM}
(\trace_{\gD}\mathrm{A}_{g})\,\phi^2
+2\,(\trace_{\gD}\mathrm{A}_{g})\,|v|_{\gD}^2
+2\,(\Sh_{g}v,v)_{\gD}
\,d\Volume_{\gD}-\int_{\dM}
\phi\,\delta_{\gD}v
\,d\Volume_{\gD}.
\end{split}
\]
It remains to eliminate the coupled term $\phi\,\delta_{\gD}v$. For this, we take the trace of the interior equation in \eqref{eq:interior_Bochner_1} and use $\trace_{g}\calR_{g}\sigma=0$, as in \eqref{eq:Lic_prop}. With $\tau:=\trace_{g}\sigma$, this yields
\[
{\nabg}^*d\tau+\trace_{g}(\tGamma\sigma)=0.
\]
Applying Green's formula \eqref{eq:Greens_outward_normal} to this identity gives
\[
\int_{M}|d\tau|^{2}_{g}+\tau\,\trace_{g}(\tGamma\sigma)\,d\Volume_{g}
=
-\int_{\dM}\tau(\nabg_{\dr}\tau)\,d\Volume_{\gD}.
\]
Using $\tau=\trace_{g}\sigma$ and comparing with \eqref{eq:normal_refined} and \eqref{eq:trace_trick}, we find
\[
\int_{M}|d\tau|^{2}_{g}+\tau\,\trace_{g}(\tGamma\sigma)\,d\Volume_{g}
=
\int_{\dM}\phi\,\delta_{\gD}v\,d\Volume_{\gD}.
\]
Adding this identity to the preceding integral cancels the coupled boundary term. Finally, since
\[
\tau\,\trace_{g}(\tGamma\sigma)
=(\trace_{g}(\tGamma\sigma)g, \sigma)_{g},
\]
we have
\[
(2\,\tGamma\sigma,\sigma)_{g}+\tau\,\trace_{g}(\tGamma\sigma)
=(\overbar{\tGamma}\sigma,\sigma)_{g}, 
\]
yielding the required formula in \propref{prop:Bochner1}.
\end{PROOF}
\subsection{Bochner technique for the second cohomology}

As discussed in \secref{sec:main_results_intro}, although the explicit expression of $\scrE^{2}(M,g,\Gamma)$ contains non-local terms, it still yields identities making the application of the Bochner technique possible. We begin by recording these identities in full generality.
\begin{proposition}
\label{prop:E2basic}
Under the assumptions of part (b) in \propref{prop:Eisntein_disrupted}, let $\eta \in \scrE^{2}(M, g, \tGamma)$. Then the following identities hold:
\beq 
\begin{gathered} 
\delBianchi\eta=0, \qquad \PnnG\eta=0, \qquad \PtnG\eta=0, 
\\\PttD\delBianchi\mathrm{B}_{g}\D\Ric_{g}\eta=-\dertZero\PttD\delta_{g+t\eta}\mathrm{B}_{g+t\eta}(\Ric_{g}|_{\dM}). 
\end{gathered}
\label{eq:boundaryE2}   
\eeq 
Here the notation on the right-hand side of the final identity emphasizes that this term depends only on $\Ric_{g}|_{\dM}$. 

Moreover, the following integral identity is satisfied:
\beq 
\int_{M} ( {\nabg}^*\nabg\eta-\nabg d\trace_{g}\eta-\Delta_{g}\trace_{g}\eta\,g + (\calR_{g}+2\tGamma^*\mathrm{B}_{g})\eta, \eta )_{g}\,d\Volume_{g}=0.
\label{eq:integral}
\eeq
\end{proposition}

\begin{proof}
The first three identities in \eqref{eq:boundaryE2} are just the conditions $\delta_{g}\eta=0$ and $\PnD_{g}\eta=0$, with the latter rewritten as in \eqref{eq:Pn_res_iso}.  

For the boundary identity in the second line, we use the linearization of the differential Bianchi identity and the chain rule to write
\[
\begin{split} 
\delBianchi\mathrm{B}_{g}\D\Ric_{g}\eta
&=\dertZero(\delta_{g+t\eta}\mathrm{B}_{g+t\eta}\Ric_{g+t\eta}-\delta_{g+t\eta}\mathrm{B}_{g+t\eta}\Ric_{g})
\\
&=-\dertZero\delta_{g+t\eta}\mathrm{B}_{g+t\eta}\Ric_{g}
\\
&=\dertZero\trace_{g+t\eta}\nabla^{g+t\eta}\mathrm{B}_{g+t\eta}\Ric_{g}
\\
&=-\sum_{i}(\nabla^{g}_{\eta^{\sharp_{g}}(e_i)}\mathrm{B}_{g}\Ric_{g})(e_i;\cdot)+\dertZero(\nabla^{g+t\eta}_{e_i}\mathrm{B}_{g+t\eta}\Ric_{g}(e_i;\cdot)). 
\end{split} 
\]
Here we used the chain rule, the identity $\delta_{g+t\eta}\mathrm{B}_{g+t\eta}\Ric_{g+t\eta}=0$, and the variation formula for the trace. We also write $\eta^{\sharp_{g}}:TM\rightarrow TM$ for the symmetric endomorphism obtained by raising one index of $\eta$. 

Choose now an orthonormal frame with $e_1=\dr$. The condition $\PnD_{g}\eta=0$ is equivalent to $\eta^{\sharp_{g}}(\dr)=0$ on the boundary. Hence the first term involves only tangential derivatives, and therefore depends only on $\Ric_{g}|_{\dM}$. For the second term, the variation formula for the Levi-Civita connection (cf.~\cite[p.~560]{Tay11b}), the fact that $\eta$ has no normal boundary components, and the tensoriality of $\mathrm{B}_{g+t\eta}$ imply that it too depends only on the boundary value $\Ric_{g}|_{\dM}$. This proves the stated boundary identity.


It remains to prove the integral identity \eqref{eq:integral}. Taking the $L^{2}$-inner product of $\tbP_{g}\mathrm{B}_{g}\D_{\ttGamma}\Ric_{g}\eta=0$ with $\eta$, and using that $\tbP_{g}$ is the $L^{2}$-orthogonal projection onto $\ker\delta_{g}$ while $\eta\in\ker\delta_{g}$, we may remove the projection inside the $L^{2}$-coupling, obtaining: 
\[
\int_{M}(\mathrm{B}_{g}\D_{\ttGamma}\Ric_{g}\eta, \eta)_{g}\,d\Volume_{g}
=
\int_{M}(\tbP_{g}\mathrm{B}_{g}\D_{\ttGamma}\Ric_{g}\eta, \eta)_{g}\,d\Volume_{g}
=0.
\]
Expanding $\D_{\ttGamma}\Ric_{g}\eta$ by the Ricci variation formula \eqref{eq:Ricci_variation}, as in the first step of the proof of \propref{prop:Bochner1}, and using that $\mathrm{B}_{g}{{\nabg}^*\nabg\eta}={\nabg}^*\nabg\mathrm{B}_{g}\eta$ since $\mathrm{B}_{g}$ is parallel, we obtain, using $\mathrm{B}_{g}\calR_{g}=0$ (cf.~\eqref{eq:Lic_prop}) and the definition $\ttGamma=\mathrm{B}_{g}^{-1}\Gamma^*\mathrm{B}_{g}$,
\[
\mathrm{B}_{g}\D_{\ttGamma}\Ric_{g}\eta=\tfrac{1}{2}{\nabg}^*\nabg\mathrm{B}_{g}\eta-\tfrac{1}{2}\mathrm{B}_{g}\nabg d\trace_{g}\eta+\tfrac{1}{2}\calR_{g}\eta+\tGamma^*\mathrm{B}_{g}\eta. 
\]
The expansion $\mathrm{B}_{g}\eta=\eta-\tfrac{1}{2}\trace_{g}\eta\,g$ then gives the required integral identity, where we use the identities: 
\[
\Delta_{g}\trace_g\eta={\nabg}^*\nabg\trace_{g}\eta=-\trace_{g}\nabg d\trace_{g}\eta={\nabg}^*\nabg(\trace_{g}\eta\,g).
\]
\end{proof}

The next proposition is perhaps the heart of the Bochner-technique argument. To carry it out, we impose an additional condition on $\tGamma$, equivalent by the chain rule to the condition presented in \eqref{eq:tgamma_condition}. Under this assumption, by comparing Weitzenb\"ock identities with the vector potential provided by the nonlocal constraint, the preceding identities yield a further boundary condition that will be key to the rest of the analysis.
\begin{proposition}
\label{prop:Sptn}
In the previous setting, let $\eta\in\scrE^{2}(M,g,\tGamma)$ and suppose that, in comparison with \eqref{eq:boundaryE2}, $\tGamma$ satisfies
\beq 
\PttD\delta_{g}\ttGamma\eta=\dertZero\PttD\delta_{g+t\eta}\mathrm{B}_{g+t\eta}(\Ric_{g}|_{\dM}).
\label{eq:divGammacond} 
\eeq 
Then
\[
\Sh_{g}\PtnG\D_{\ttGamma}\Ric_{g}\eta=0.
\]
\end{proposition}

\begin{proof}
Set the quantity: 
\[
\upsilon:=\mathrm{B}_{g}\D_{\ttGamma}\Ric_{g}\eta\in\SM.
\]
Since $\PtnG\mathrm{B}_{g}=\PtnG$, as $\PtnG g=0$ and $\mathrm{B}_{g}=\id-\tfrac{1}{2}\trace(\cdot)g$, the claim is reduced to showing that $\Sh_{g}\PtnG\upsilon$ vanishes. Since $\eta\in\scrE^{2}(M,g,\tGamma)$, we have in particular that $\tbP_{g}\upsilon=0$, which is to say: 
\beq 
(\id-\tbP_{g})\upsilon=\upsilon.
\label{eq:PttDRic}
\eeq
The operator $\id-\tbP_{g} \colon \SM \rightarrow \SM$ is the $L^{2}$-orthogonal projection onto $\mathrm{im}(\Def|_{\ker(\cdot)|_{\partial M}})$. Hence, using \eqref{eq:Killing_exterior_relations}, we may write
\[
\upsilon=\Def{X}=\dgV\omega+(\dgV\omega)^{T},
\qquad
\omega:=\tfrac{1}{2}X^{\flat_{g}}\in\Wkm{1}{0}.
\]
The identity $\PttD\Def{X}=0$ for such vector fields, proved in \eqref{eq:second_fudnemental_Lie}, gives
\[
\PttD\upsilon=0.
\]
The last two boundary identities in \eqref{eq:boundaryE2}, combined with \eqref{eq:divGammacond}, then yield in addition
\[
\PttD\delta_{g}\upsilon=0.
\]
Since $X|_{\dM}=0$ is equivalent to $\omega|_{\dM}=0$, the condition $\PttD\upsilon=0$, together with the commutation formula in \cite[p.~706]{KL21a}, then gives
\beq
\label{eq:Ptn-upsilon-general}
\begin{gathered} 
\PtnG\upsilon=\PttG\nabg_{\dr}\omega, \qquad -\PttD\deltag\omega=\PnnG\upsilon=(\nabg_{\dr}\omega)(\dr)|_{\dM}.
\end{gathered} 
\eeq
From this point on, it is useful to keep track of the function $\phi\in C^{\infty}(\dM)$ defined by
\beq
\label{eq:phi-def-general}
\phi:=-\PttD\deltag\omega=\PnnG\upsilon=(\nabg_{\dr}\omega)(\dr)|_{\dM}.
\eeq
We proceed by expanding the relation $\PttD\delta_{g}\upsilon=0$. Since $\PttD\upsilon=0$, the boundary commutation formulae in \cite[pp~705--706]{KL21a} gives
\beq
\label{eq:delta-upsilon-expanded-general}
\PtnG\nabg_{\dr}\upsilon
=
-(\trace_{\gD}\Ah_{g})\,\PtnG\upsilon-\Sh_{g}\PtnG\upsilon.
\eeq

We now expand the left-hand side. Let $Y\in\frakX(M)|_{\dM}$ be a tangential vector field extended so that $\nabg_{\dr}Y=0$. Then, on $\dM$,
\[
\begin{split}
(\PtnG\nabg_{\dr}\upsilon)(Y)
&=(\nabg_{\dr}\upsilon)(\dr;Y)
=\dr(\upsilon(\dr;Y))=\dr(\nabg_{\dr}\omega(Y)+\nabg_{Y}\omega(\dr))
\\
&=(\nabg_{\dr}\nabg_{\dr}\omega)(Y)+\nabg_{\dr}\nabg_{Y}\omega(\dr)
\\
&=(\nabg_{\dr}\nabg_{\dr}\omega)(Y)+\nabg_{Y}\nabg_{\dr}\omega(\dr)
\\
&=(\nabg_{\dr}\nabg_{\dr}\omega)(Y)+Y(\nabg_{\dr}\omega(\dr))-(\PttG\nabg_{\dr}\omega,\nabg_{Y}\dr)_{\gD}
\\
&=(\nabg_{\dr}\nabg_{\dr}\omega)(Y)+d\phi(Y)-(\Sh_{g}\PtnG\upsilon)(Y),
\end{split}
\]
where we used $\omega|_{\dM}=0$ to commute derivatives on the boundary, then \eqref{eq:phi-def-general}, and finally \eqref{eq:Ptn-upsilon-general}. Comparing this identity with \eqref{eq:delta-upsilon-expanded-general}, we obtain
\beq
\label{eq:Ptt-nablan2-omega-general}
d\phi=-\PttD\nabg_{\dr}\nabg_{\dr}\omega-(\trace_{\gD}\Ah_{g})\,\PtnG\upsilon.
\eeq

We next compare this with the Weitzenb\"ock formula for the $1$-form $\omega$; see \cite[Ch.~9.4.1]{Pet16}. Since $\omega|_{\dM}=0$, we have
\beq
\label{eq:weitzenbock-omega-general}
({\nabg}^*\nabg \omega)|_{\dM}
=
(\deltag d\omega+d\deltag\omega)|_{\dM}.
\eeq

For the left-hand side, choose an orthonormal frame with $e_1=\dr$ and $\nabla^{\gD}e_i=0$ at the point under consideration. Then
\[
{\nabg}^*\nabg\omega
=
-\nabg_{\dr}\nabg_{\dr}\omega
-\sum_{i\ne 1}(\nabg_{e_{i}}\nabg_{e_{i}}\omega-\nabg_{\nabg_{e_{i}}e_{i}}\omega).
\]
Since $\omega|_{\dM}=0$, we have $\nabg_{e_{i}}\omega=0$ on $\dM$ for $i\ne 1$, and hence the terms $\nabg_{e_{i}}\nabg_{e_{i}}\omega$ vanish. Moreover, since $\nabg_{e_i}e_i=-\Ah_{g}(e_i,e_i)\dr$ on the boundary,
\[
(\nabg_{\nabg_{e_{i}}e_{i}}\omega)|_{\dM}
=
-\Ah_{g}(e_{i},e_{i})\,\PttG\nabg_{\dr}\omega
=
-\Ah_{g}(e_{i},e_{i})\,\PtnG\upsilon.
\]
Thus,
\[
\PttD{\nabg}^*\nabg\omega
=
-\PttD\nabg_{\dr}\nabg_{\dr}\omega-(\trace_{\gD}\Ah_{g})\,\PtnG\upsilon.
\]
Using \eqref{eq:Ptt-nablan2-omega-general}, we conclude that
\beq
\label{eq:Ptt-rough-laplacian-general}
\PttD{\nabg}^*\nabg\omega= d\phi.
\eeq

Now apply $\PttD$ to \eqref{eq:weitzenbock-omega-general}. Since $\phi=-\PttD\deltag\omega$ and $d$ commutes with pullback, this gives
\beq
\label{eq:weitzenbock-boundary-general}
\PttD{\nabg}^*\nabg\omega
=
\PttD\deltag d\omega- d\phi.
\eeq
Comparing \eqref{eq:Ptt-rough-laplacian-general} with \eqref{eq:weitzenbock-boundary-general}, we find
\beq
\label{eq:Ptt-delta-domega-two-dphi}
\PttD\deltag d\omega=2d\phi.
\eeq

It remains to compute $\PttD\deltag d\omega$ directly. Since $\omega|_{\dM}=0$, the commutation formulae in \cite[p.~706]{KL21a} give
\[
\PttD d\omega=0,
\qquad
\PntG d\omega=\PttG\nabg_{\dr}\omega=\PtnG\upsilon.
\]
Hence
\beq
\label{eq:Ptt-delta-domega-general}
\PttD\deltag d\omega
=
-(\trace_{\gD}\Ah_{g})\,\PtnG\upsilon-\PntG\nabg_{\dr}d\omega.
\eeq

To compute the last term, let again $Y$ be tangential with $\nabg_{\dr}Y=0$. Then
\[
\begin{split}
(\PntG\nabg_{\dr}d\omega)(Y)
&=
(\nabg_{\dr}d\omega)(\dr,Y)
=
\dr(d\omega(\dr,Y))
\\
&=
\dr(\nabg_{\dr}\omega(Y)-\nabg_{Y}\omega(\dr))
\\
&=
(\nabg_{\dr}\nabg_{\dr}\omega)(Y)-\nabg_{Y}\nabg_{\dr}\omega(\dr)
\\
&=
(\nabg_{\dr}\nabg_{\dr}\omega)(Y)-Y(\phi)+(\PttG\nabg_{\dr}\omega,\nabg_{Y}\dr)_{\gD}
\\
&=
(\nabg_{\dr}\nabg_{\dr}\omega)(Y)-d\phi(Y)+(\Sh_{g}\PtnG\upsilon)(Y),
\end{split}
\]
where we again used $\omega|_{\dM}=0$ to commute derivatives on the boundary, then \eqref{eq:phi-def-general}, and finally \eqref{eq:Ptn-upsilon-general}. Therefore, by \eqref{eq:Ptt-nablan2-omega-general},
\[
\PntG\nabg_{\dr}d\omega
=
-(\trace_{\gD}\Ah_{g})\,\PtnG\upsilon-2d\phi+\Sh_{g}\PtnG\upsilon.
\]
Substituting this into \eqref{eq:Ptt-delta-domega-general}, we obtain
\beq
\label{eq:Ptt-delta-domega-final-general}
\PttD\deltag d\omega
=
2d\phi-\Sh_{g}\PtnG\upsilon.
\eeq

Finally, comparing \eqref{eq:Ptt-delta-domega-two-dphi} with \eqref{eq:Ptt-delta-domega-final-general}, we conclude that
\[
\Sh_{g}\PtnG\upsilon=0,
\]
which completes the proof.
\end{proof}

With these preparations, by iterating Green's formula \eqref{eq:Greens_outward_normal} upon the integral identity \eqref{eq:integral}, and using the conditions $\PnD_{g}\eta=0$ and $\deltag\eta={\nabg}^*\eta=0$, we find that all boundary terms cancel except the purely tangential ones:
\[
\begin{split} 
\int_{M} |\nabg\eta|_{g}^2 - &|d\trace_{g}\eta|_{g}^2 + ((\calR_{g}+2\tGamma^*\mathrm{B}_{g})\eta, \eta )_{g} \,d\Volume_{g}
\\&= -\int_{\dM}(\xi,\mu)_{\gD}\,d\Volume_{\gD}+\int_{\dM}(\PnnG\nabg_{\dr}\eta+\trace_{\gamma}\xi)(\trace_{\gamma}\mu)\,d\Volume_{\gD},
\end{split} 
\]
where we have defined
\[
\mu := \PttD \eta\in\SdM,
\qquad
\xi := \PttD \nabg_{\dr} \eta\in\SdM.
\]
This identity shows that, even when $\dM=\emptyset$, positivity of the interior expression cannot be secured merely by imposing positivity assumptions on $\calR_{g}+2\tGamma^*\mathrm{B}_{g}$, unless $d\trace_{g}\eta=0$. 

For this reason, as discussed before \thmref{thm:vanishing_intro2}, we restrict to elements of $\scrE^{2}(M,g,\tGamma)$ satisfying $d\trace_{g}\eta=0$, so the integral identity \eqref{eq:integral} reduces to
\beq 
\begin{split} 
&\int_{M} |\nabg\eta|_{g}^2+ ((\calR_{g}+2\tGamma^*)\eta, \eta )_{g} \,d\Volume_{g}=-\int_{\dM}(\xi,\mu)_{\gD}\,d\Volume_{\gD}.
\end{split} 
\label{eq:integral2} 
\eeq  

In order to evaluate the resulting boundary term, we first derive a set of identities in a slightly more general form. This will make clear where the full specialization to the setting of \thmref{thm:vanishing_intro2} is needed.

\begin{proposition}
\label{prop:relations_E2}
Let
$\eta \in \scrE^{2}(M,g,\tGamma)\cap \ker d\trace_{g}$, so $\trace_{g}\eta=c\in\bbR$.
Then
\beq
\begin{gathered}
\trace_{\gD}\mu=c, \qquad \trace_{\gD}\xi=-\PnnG\nabg_{\dr}\eta=-(\mathrm{A}_{g}, \mu)_{\gD}, 
\\ H_{\gD}^*\mu+(\Ah_{g}, \xi)_{\gD}+(\Ah^{2}_{g}, \mu)_{\gD}=-\trace_{\gD}\PttD(2\tGamma^*\eta+\tfrac{1}{2}\calR_{g}\eta).
\end{gathered}
\label{eq:relations_E2}
\eeq
Moreover, if the assumptions of \propref{prop:Sptn} are satisfied and the shape operator $\Sh_{g}$ is non-singular, then
\beq 
\tfrac{1}{2}d_{\gD}\xi
=  -\dertZero d_{\gD+t\mu}\Ah_{g}-d_{\gD}\Sh_{g}\mu-\tfrac{1}{n-2}\gD\wedge\PtnG\tGamma^*\eta+\PtnG\D\mathrm{Wey}_{g}\eta.
\label{eq:relation_E22}
\eeq 
\end{proposition}
The proof relies on Weitzenb\"ock identities and linearized constraints; namely, on comparing the Ricci decomposition of the curvature tensor with the Codazzi equations on the boundary.
\begin{proof}
The identity $\trace_{\gD}\mu=c$ in the first row of \eqref{eq:relations_E2} follows directly from $\trace_{g}\eta=c$, $\PnnG\eta=0$, and the same computation as in \eqref{eq:trace_trick}:
\[
c=(\trace_{g}\eta)|_{\dM}=\PnnG\eta+\trace_{\gD}\PttD\eta=\trace_{\gD}\mu.
\]
We next prove that
\[
\trace_{\gD}\xi=-\PnnG\nabg_{\dr}\eta=-(\mathrm{A}_{g}, \mu)_{\gD}.
\]
To observe this, on the one hand, since the trace and the covariant derivative commute, the same computation as in \eqref{eq:trace_trick} gives
\[
\begin{split}
\trace_{\gD}\xi=\trace_{\gD}\PttD\nabg_{\dr}\eta
&=(\trace_{g}\nabg_{\dr}\eta)|_{\dM}-\PnnG\nabg_{\dr}\eta\\
&=(\nabg_{\dr}\trace_{g}\eta)|_{\dM}-\PnnG\nabg_{\dr}\eta
=-\PnnG\nabg_{\dr}\eta.
\end{split}
\]
On the other hand, applying $\PtnG$ to $\delBianchi\eta=0$, and using $\PnnG\eta=0$ and $\PtnG\eta=0$ together with the commutation formulae in \cite[pp.~705--706]{KL21a}, yields
\[
\PnnG\nabg_{\dr}\eta=(\mathrm{A}_{g}, \mu)_{\gD}.
\]
By comparing, this proves the identities on the first row of \eqref{eq:relations_E2}.

For the identity in the second row, we apply $\PttG$ to $\delBianchi\eta=0$. Again using the commutation formulae in \cite[pp.~705--706]{KL21a}, we obtain
\[
\delta_{\gD}\mu=\PtnD\nabg_{\dr}\eta=\PnnG\dgV\eta.
\]
Applying $\delta_{\gD}^{V}$ to this identity, and using the commutation formulae once more together with $\PnnG\eta=0$ and $\PntG\eta=0$, gives after simplification
\[
\begin{split} 
H_{\gD}^*\mu
&= -\PnnG\deltagV\dgV\eta-(\Ah_g,\PtnG\dgV\eta)_{\gD}
\\
&= -\PnnG\deltagV\dgV\eta-(\Ah^{2}_g,\mu)_{\gD}-(\Ah_g,\xi)_{\gD}.
\end{split} 
\]
It remains to evaluate $\PnnG\deltagV\dgV\eta$. Since $\trace_{g}\eta=c$ implies $\mathrm{B}_{g}\eta=\eta-\tfrac{1}{2}cg$, the Ricci variation formula \eqref{eq:Ricci_variation}, together with $\delta_{g}\eta=0$, gives
\[
(\mathrm{B}_{g}\D_{\ttGamma}\Ric_{g}\eta)|_{\dM}
=(\tfrac{1}{2}{\nabg}^*\nabg\eta+\tfrac{1}{2}\calR_{g}\eta+\tGamma^*\eta)|_{\dM}.
\]
Thus, the boundary condition in the second row of \eqref{eq:boundaryE2} reads
\[
\PttD{\nabg}^*\nabg\eta+\PttD\calR_{g}\eta+2\PttD\tGamma^*\eta=0.
\]
Using the Weitzenb\"ock formula \cite[Ch.~9.4.1]{Pet16} to relate ${\nabg}^*\nabg$ to $\deltag\dg+\dg\deltag$ and its transposed version, and using again $\deltag\eta=0$, we obtain
\[
{\nabg}^*\nabg\eta=\deltagV\dgV\eta-\tfrac{1}{2}\mathcal{R}_{g}\eta.
\]
Substitution into the previous expression gives, after rearranging,
\[
\PttD\deltagV\dgV\eta=-\tfrac{1}{2}\PttD\calR_{g}\eta-2\PttD\tGamma^*\eta.
\]
Applying $\trace_{\gD}$, and using $\trace_{g}{\nabg}^*\nabg\eta={\nabg}^*\nabg\trace_{g}\eta=0$ and $\trace_{g}\calR_{g}\eta=0$, together with the computation in \eqref{eq:trace_trick}, we find
\[
\PnnG\deltagV\dgV\eta
=\trace_{\gD}\PttD(2\tGamma^*\eta+\tfrac{1}{2}\calR_{g}\eta).
\]
Combining this with the expression for $H_{\gD}^*\mu$ proves the second row of \eqref{eq:relations_E2}.

We now prove \eqref{eq:relation_E22}. Under the assumptions of \propref{prop:Sptn}, and assuming that the shape operator is non-singular, we have
\[
0=\PtnG\D_{\ttGamma}\Ric_{g}\eta=\PtnG\D\Ric_{g}\eta+\PtnG\tGamma^*.
\]
Here we used that $\ttGamma\eta=\mathrm{B}_{g}^{-1}(\tGamma^*\eta-\tfrac{1}{2}cg)$ since $\trace_{g}\eta=c$, and hence $\PtnG\mathrm{B}_{g}^{-1}\tGamma^*=\PtnG\tGamma^*$ because $\PtnG g=0$.

We next write the Ricci decomposition for $\Rm_{g}$, which is the defining relation for $\mathrm{Wey}_{g}$ (cf.~\cite[p.~109]{Pet16}; note that the wedge product differs from the Kulkarni--Nomizu product by a sign):
\[
\Rm_{g}=-\tfrac{1}{n-2}\,g\wedge\Ric_{g}+\tfrac{1}{2(n-2)(n-1)}\,\mathrm{Sc}_{g}\,g\wedge g+\mathrm{Wey}_{g}.
\]
Applying $\PtnG$ to both sides and using $\PtnG g=0$, we obtain
\[
\PtnG\Rm_{g}=\tfrac{1}{n-2}\gD\wedge\PtnG\Ric_{g}+\PtnG\mathrm{Wey}_{g}.
\]
This identity remains true with $g$ replaced by $g+t\eta$. Linearizing it, and using that $\dertZero \PntD_{g+t\eta}=0$ because $\eta$ has no normal components, gives
\[
\PtnG\D\Rm_{g}\eta=\tfrac{1}{n-2}\gD\wedge\PtnG\D\Ric_{g}\eta+\tfrac{1}{n-2}\mu\wedge\PtnG\Ric_{g}+\PtnG\D\mathrm{Wey}_{g}\eta.
\]
Substituting $\PtnG\D\Ric_{g}\eta=-\PtnG\tGamma^*$ and linearizing the Codazzi equations, which can be written as (\cite[Ch.~4]{KL21a}):
\[
\PtnG\D\Rm_{g}=d_{\gD}\mathrm{A}_{g},
\]
we obtain
\[
\dertZero d_{\gD+t\mu}\mathrm{A}_{g}+d_{\gD}\D\Ah_{g}\eta
=-\tfrac{1}{n-2}\gD\wedge\PtnG\tGamma^*+\PtnG\D\mathrm{Wey}_{g}\eta.
\]
Finally, expanding $\D\Ah_{g}\eta$ by \eqref{eq:DA}, and using the explicit expression in \cite[p.~755]{KL21a} for $\frakT_{g}\eta$ together with $\PnnG\eta=0$ and $\PntG\eta=0$, gives
\[
d_{\gD}\D\Ah_{g}\eta=\tfrac{1}{2}d_{\gD}\xi+d_{\gD}\Sh_{g}\mu.
\]
By substituting this into the preceding identity and rearranging we are done. 
\end{proof}
The identities in \eqref{eq:relations_E2} do not, by themselves, appear sufficient to control the boundary term in \eqref{eq:integral2}, even when $\Ah_{g}$ is definite. We therefore specialize to the setting of \thmref{thm:vanishing_intro2}: namely, $\mathrm{A}_{g} = -\sqrt{\kappa}\,\gD$ for some constant $\kappa > 0$, the boundary $(\dM,\gD)$ is a round sphere, and $\PtnG\D\mathrm{Wey}_{g}\eta=0$. This completes the setup needed for the proof of \thmref{thm:vanishing_intro2}.

\begin{PROOF}{\thmref{thm:vanishing_intro2}}
Since $\tGamma^*=\tGamma$, we have $\PtnG\tGamma^*\eta=\PtnG\mathrm{B}_{g}\tGamma\eta0$ by the assumption that $\tGamma$ satisfies \eqref{eq:vanishing_RicBody2}, together with the fact that $\dertZero\PtnD_{g+t\eta}=0$, since $\eta$ has no normal components. We also have $\trace_{\gamma}\PttD(2\tGamma+\tfrac{1}{2}\calR_{g})=0$ by the second condition in \eqref{eq:vanishingCon2}. 

Now, under the assumption that $\Ah_{g}=-\sqrt\kappa \gamma$, we readily find, by the well-known variation formula for the connection \cite[p.~559]{Tay11b}, that
\[
-\dertZero d_{\gD+t\mu}\mathrm{A}_{g}-d_{\gamma}\Sh_{g}\mu=-\sqrt{\kappa}d_{\gamma}\mu+\sqrt{\kappa}d_{\gamma}\mu=0.
\]
We also have $\trace_{\gD}\xi=-(\Ah_{g},\mu)_{\gD}=\sqrt{\kappa}\trace_{\gD}\mu=\sqrt{\kappa}c$, along with
\[
(\Ah_{g},\xi)_\gD+(\Ah_{g}^{2},\mu)_{\gD}=-\sqrt{\kappa}\trace_{\gamma}\xi+ \kappa\trace_{\gamma}\mu=0.
\]
Thus, under these assumptions, the relations \eqref{eq:relations_E2} and \eqref{eq:relation_E22}, together with the already imposed condition \eqref{eq:vanishing_RicBody}, reduce to
\beq 
\begin{gathered}
\trace_{\gD}\xi=\sqrt{\kappa}c, \qquad d_{\gD}\xi = 0, \qquad H_{\gD}^*\mu=0.
\end{gathered} 
\label{eq:refined_sphere} 
\eeq
To exploit these identities, we apply the classical fact (e.g., \cite{Cal61} or \cite[Sec.~5]{BE69}) that the condition $d_{\gD}\xi=0$ on a sphere of constant curvature is equivalent to the existence of a scalar potential $\chi \in C^{\infty}(\dM)$ such that
\[
\xi=H_{\gD}\chi+\kappa\,\chi\gD.
\]
Taking $\trace_{\gD}$ gives
\[
\sqrt{\kappa}c=-\Delta_{\gamma}\chi+(n-1)\kappa\,\chi.
\]
Integrating over $\dM$, it follows that
\[
\tfrac{c}{(n-1)\sqrt{\kappa}}\Volume_{\gD}(\dM)=\int_{\dM}\chi\,d\Volume_{\gamma}.
\]
With this in hand, substituting the representation of $\xi$ into \eqref{eq:integral2} yields
\[
\begin{split}
\int_{M} |\nabg\eta|_{g}^2+ ((\calR_{g}+2\tGamma)\eta, \eta )_{g} \,d\Volume_{g}
&=-\int_{\dM}(H_{\gD}\chi+\kappa\,\chi\gD,\mu)_{\gD}  \,d\Volume_{\gD}
\\
&=-\int_{\dM}\chi(H_{\gD}^*\mu+\kappa\trace_{\gD}\mu)\,d\Volume_{\gD}
\\
&=-\int_{\dM}\chi\,\kappa c\,d\Volume_{\gD}
=-\tfrac{c^2\sqrt{\kappa}}{n-1}\Volume_{\gD}(\dM),
\end{split}
\]
where we used Green's formula \eqref{eq:Hg_Green} on the closed boundary manifold, the fact that the adjoint of multiplication by $\gD$ is $\trace_{\gD}$, and finally \eqref{eq:refined_sphere} together with the derived expression for the integral of $\chi$.

Thus, under the assumption $\calR_{g}+2\tGamma\geq 0$, one side is non-positive while the other is non-negative, and hence we must have $\nabg\eta = 0$ everywhere and $c=0$. If the boundary is empty, the additional assumption that $\calR_{g}+2\tGamma>0$ at some point implies $\eta=0$. If the boundary is nonempty, then since $\eta$ has no normal components, it follows that $\nabla^{\gD}\mu = 0$. Since the sphere is irreducible, by holonomy, the only parallel symmetric tensor fields are of the form $\mu=\lambda\,\gD$ for some constant $\lambda \in \mathbb{R}$ (cf.~\cite[p.~386]{Pet16}). Since $\trace_{\gD}\mu=c=0$, we obtain $\lambda=0$, and therefore $\mu=\PttD\eta=0$. Together with $\PnD_{g}\eta=0$, this implies $\eta|_{\dM}=0$. Since $\eta$ is parallel, we conclude that $\eta=0$.

\end{PROOF}
\section{Duality for curvature}
\label{sec:duality_for_curvature} 
\subsection{Duality between the Einstein and Riemann curvatures}
\label{sec:duality}
In this section, we prove \thmref{thm:duality_Riem} and derive several applications. To this end, we shall use a further refinement of the notion of double forms, namely Bianchi forms \cite{Kul72}, which we recall; see \cite{KL21a,KL23} for detailed surveys.

In the setting of \secref{sec:double_forms}, if $(\vartheta^{i})$ is the dual orthonormal co-frame for $(E_i)$, the {Bianchi sum} is the map
\[
\frakG\psi = \sum_{i=1}^n \vartheta^i \wedge i^V_{E_i}\psi, \qquad \frakG:\Lkm{k}{m} \to \Lkm{k+1}{m-1}.
\]
For $k \leq m$, this allows us to define the bundle of {Bianchi $(k,m)$-covectors}:
\[
\Gkm{k}{m} = \Lkm{k}{m} \cap \ker \frakG.
\]
In particular, it is shown in \cite{Kul72} that $\Gkm{1}{1}$ coincides with the bundle of symmetric $(2,0)$-covectors, while $\Gkm{2}{2}$ is the bundle of $(4,0)$-covectors satisfying the {algebraic Bianchi identities}, i.e., algebraic curvature tensors. 

We denote by $\Ckm{k}{m} := \tGamma(\Gkm{k}{m})$ the associated space of section, which is the module of {Bianchi forms}. In this framework, one has $\Ein_{g} \in \Ckm{1}{1} = \SM$ and $\Rm_{g} \in \Ckm{2}{2}$, related by
\beq 
\Ein_{g} = \mathrm{E}_{g} \Rm_{g},
\label{eq:Einstein_Rm}
\eeq 
where $\mathrm{E}_{g}:\Ckm{2}{2} \to \Ckm{1}{1}$ and $\mathrm{C}_{g}:\Ckm{1}{1} \to \Ckm{1}{1}$ were introduced in \eqref{eq:E_C_contractions}. In the notation of \thmref{thm:duality_Riem}, we shall prove the following:

\begin{proposition}
\label{prop:Eisntein_duality}
Let $\psi \in \Ckm{2}{2}$ and $\sigma \in \Ckm{1}{1}$. Then,
\beq
\begin{split}
&\starG \starGV (g^{\,n-3} \psi) = (n-3)! \, \mathrm{E}_{g}\psi, \qquad n \geq 3, \\[0.3em]
&\starG \starGV (g^{\,n-2}  \sigma) = (n-2)! \, \mathrm{C}_{g}\sigma, \qquad n \geq 2.
\end{split}
\label{eq:Einstein_relation}
\eeq
\end{proposition}
Taking $\psi=\Rm_{g}$ then gives the proof of \thmref{thm:duality_Riem}.

\begin{proof}
Since the identities in \eqref{eq:Einstein_relation} are tensorial, it suffices to verify them fiber-wise. By \cite[Thm.~3.1]{Kul72}, $\Gkm{k}{k}$ is spanned fiber-wise by elements of the form
\[
\omega^{1} \wedge \cdots \wedge \omega^{k} \wedge (\omega^{1} \wedge \cdots \wedge \omega^{k})^{T}, 
\qquad \omega^{i} \in \Lkm{1}{0}.
\]
For $k=1$, this is the polarization identity for symmetric bilinear forms, and for $k=2$ it is its analogue for algebraic curvature tensors (\cite[~3.4.29]{Pet16}).

Assume that $(\omega^{i})_{i=1}^{k}$ are independent covectors. After Gram--Schmidt, we may take $\omega^{i} = \vartheta^{i}$, the first $k$ elements of an orthonormal co-frame $(\vartheta^{j})_{j=1}^{n}$ with dual frame $(E_{j})_{j=1}^{n}$. For brevity, write
\[
(\vartheta^{i_{1}} \wedge \cdots \wedge \vartheta^{i_{k}}) \wedge (\vartheta^{i_{1}} \wedge \cdots \wedge \vartheta^{i_{k}})^{T} 
= (\vartheta^{i_{1}} \wedge \cdots \wedge \vartheta^{i_{k}})^{2}.
\]

Now take $\psi = (\vartheta^{1} \wedge \cdots \wedge \vartheta^{k})^{2}$ and assume $n \geq k+1$. Using \cite[Lem.~3.4]{KL21a}, together with the action of the Hodge star on an oriented orthonormal basis, we obtain
\[
\begin{split}
\starG \starGV \bigl(g^{\,n-k-1} (\vartheta^{1} \wedge \cdots \wedge \vartheta^{k})^{2}\bigr) 
&= \trace_{g}^{\,n-k-1}\bigl(\starG \starGV (\vartheta^{1} \wedge \cdots \wedge \vartheta^{k})^{2}\bigr) \\
&= \trace_{g}^{\,n-k-1}(\vartheta^{k+1} \wedge \cdots \wedge \vartheta^{n})^{2}.
\end{split}
\]
Since $\vartheta^{j}(E_{i}) = \delta_{i}^{j}$, a combinatorial argument using \eqref{eq:trace} gives
\beq 
\trace_{g}^{\,n-k-1}(\vartheta^{k+1} \wedge \cdots \wedge \vartheta^{d})^{2} 
= (n-k-1)! \sum_{j=k+1}^{n} (\vartheta^{j})^{2}.
\label{eq:duality_1} 
\eeq 

On the other hand, for $k=2$, applying $\mathrm{E}_{g}:\Ckm{2}{2}\to\Ckm{1}{1}$ yields
\beq 
\mathrm{E}_{g}\big((\vartheta^{1} \wedge \vartheta^{2})^{2}\big) 
= -(\vartheta^{1})^{2} - (\vartheta^{2})^{2} + g 
= \sum_{j=3}^{n} (\vartheta^{j})^{2},
\label{eq:duality_2} 
\eeq 
where we used $g = \sum_{j=1}^{n} (\vartheta^{j})^{2}$. Comparing \eqref{eq:duality_1} with \eqref{eq:duality_2} proves the first identity in \eqref{eq:Einstein_relation}.

Similarly, for $k=1$, applying $\mathrm{C}_{g}:\Ckm{1}{1}\to\Ckm{1}{1}$ gives
\beq 
\mathrm{C}_{g}\big((\vartheta^{1})^{2}\big) 
= -(\vartheta^{1})^{2} + \trace_{g} (\vartheta^{1})^{2}\, g
= -(\vartheta^{1})^{2} + g 
= \sum_{j=2}^{n} (\vartheta^{j})^{2}.
\label{eq:duality_3} 
\eeq 
Comparing \eqref{eq:duality_1} with \eqref{eq:duality_3} proves the second identity in \eqref{eq:Einstein_relation}.
\end{proof}
\subsection{Application: Einstein's Green's formula}
\label{sec:Einstein_green} 
As a first application of the duality outlined above, we provide a proof for \propref{prop:normal_derivative_DA}. 

Recall again  $\Hg : \SM = \Ckm{1}{1} \to \Ckm{2}{2}$, the leading-order differential term in the expression linearized curvature tensor $\D\Rm_{g}$ in \eqref{eq:DRm}. By applying the chain rule to \eqref{eq:Einstein_Rm}, we may therefore write
\beq
\D\Ein_{g} = \tfrac{1}{2}\bigl(\mathrm{E}_{g} \Hg + E D_{g}\bigr),
\label{eq:Ein_variation}
\eeq
where $E D_{g}$ denotes a tensorial operation.
\begin{PROOF}{\propref{prop:normal_derivative_DA}}
Note that $\mathrm{B}_{g}\D\Ric_{g}$ is indeed formally self-adjoint by virtue of $\mathrm{B}_{g}$ being parallel and the properties of the Lichnerowicz curvature operator in \eqref{eq:Lic_prop}. Since $\mathrm{B}_{g}\D\Ric_{g}$ differs from $\D\Ein_{g}$ only by lower-order terms (by the chain rule applied to \eqref{eq:Ein_B_Ric}), the result follows—once $(\mathrm{E}_{g} \Hg)^* = \mathrm{E}_{g} \Hg$ has been established—by iterating the Green's formula for $\Hg$ in \eqref{eq:Hg_Green} and the identities \eqref{eq:Einstein_relation}. We note that the tensorial term integrates without a boundary contribution; hence, all these operators share the same boundary integral in their corresponding Green's formulas.

To show that $(\mathrm{E}_{g} \Hg)^* = \mathrm{E}_{g} \Hg$, observe that by \eqref{eq:Einstein_relation} and the dualities between $\Hg$ and $\Hg^*$, between $\starG \starGV$ and itself, and between $\trace_{g}$ and $g \wedge$, 
\[
(\mathrm{E}_{g}\Hg)^* = \tfrac{1}{(n-3)!} (\starG \starGV \, g^{n-3} \Hg )^* = \tfrac{1}{(n-3)!} \Hg^* \, \trace_{g}^{\,n-3} \, \starG \starGV.
\]
Since $\Hg^*$ commutes with $\trace_{g}$ \cite[Prop.~3.10]{KL21a}, and $\starG \starGV \Hg^* = \Hg \starG \starGV$:
\[
(\mathrm{E}_{g}\Hg)^* = \tfrac{1}{(n-3)!} \trace_{g}^{\,n-3} \, \starG \starGV \, \Hg.
\]
The proof is completed by invoking the relation $
\trace_{g}^{\,n-3} \, \starG \starGV = \starG \starGV \, g^{n-3}$. 

\end{PROOF}
\subsection{Application: Einstein's constraint equation}
\label{sec:constraint_duality} 
\begin{proposition}
\label{prop:electric_idenity} 
The following constraint equations hold,
\beq
\begin{aligned} 
&\PttD\Ein_{g} = \Ein_{\gD} + \mathrm{C}_{\gD} \PnnG \Rm_{g} + \tfrac{1}{2}\mathrm{E}_{\gD}(\Ah_{g} \wedge \Ah_{g}), \qquad &&\text{for } n\geq 3, \\[0.5em]
&\PttD\Ein_{g} = \mathrm{C}_{\gD} \PnnG \Rm_{g} + \mathrm{E}_{\gD}\PttD\Rm_{g}, \qquad &&\text{for } n\geq 3, \\[0.5em]
&(\tfrac{n-3}{n-2}) \PttD \Ein_{g} = \Ein_{\gD} - \PnnG \mathrm{Wey}_{g} + \tfrac{1}{2}\mathrm{E}_{\gD}(\Ah_{g} \wedge \Ah_{g}),  \qquad &&\text{for } n>3.
\end{aligned}
\label{eq:electric_constraints}
\eeq
\end{proposition}
\begin{proof}
Applying $\PttD$ on \eqref{eq:Einstein_Rm}, and expanding $\PttD \Ein_{g}$ together with the commutation relation $\PttD \trace_{g} - \trace_{\gD}\PttD = \PnnG$ and the Gauss equations \cite[p.~704]{KL21a}, yields after some manipulation:
\[
\begin{split}
\PttD \Ein_{g} 
&=-\PnnG \Rm_{g} + \trace_{\gD} \PnnG \Rm_{g}\,\gD - \trace_{\gD} \Rm_{\gD} + \tfrac{1}{2}\trace_{\gD}\trace_{\gD}\Rm_{\gD}\,\gD \\
&\qquad - \tfrac{1}{2}\trace_{\gD}(\Ah_{g}\wedge\Ah_{g}) + \tfrac{1}{4}\trace_{\gD}\trace_{\gD}(\Ah_{g}\wedge\Ah_{g})\,\gD.
\end{split}
\]
Regrouping terms, we obtain the first and second identities in \eqref{eq:electric_constraints}. 

For the third identity, let $\mathrm{P}_{g} \in \Ckm{1}{1}$ be the {Schouten tensor} of $g$, and recall the Kulkarni–Nomizu product and the Weyl tensor \cite[p.~109]{Pet16}. Since the Kulkarni-Nomizu product differs from the wedge by a sign \cite{Kul72}:  
\beq
\Rm_{g} = -g \wedge \mathrm{P}_{g} + \mathrm{Wey}_{g}.
\label{eq:orthogonal_decompostion_Rm} 
\eeq
Thus,
\[
\begin{aligned}
\mathrm{C}_{\gD} \PnnG \Rm_{g} 
&= -\mathrm{C}_{\gD}(\PttD \mathrm{P}_{g} + \PnnG \mathrm{P}_{g}\,\gD) + \mathrm{C}_{\gD}\PnnG \mathrm{Wey}_{g}.
\end{aligned}
\]
On the other hand, combining \eqref{eq:orthogonal_decompostion_Rm},\eqref{eq:E_C_contractions},\eqref{eq:Einstein_relation} and $\trace_{g}\mathrm{Wey}_{g}=0$, gives
\[
\Ein_{g} = -(n-2)\,\mathrm{C}_{g}\mathrm{P}_{g}, \qquad \mathrm{C}_{\gD}\PnnG \mathrm{Wey}_{g}=-\PnnG \mathrm{Wey}_{g}.
\]
Hence,
\[
\PttD \Ein_{g} = -(n-2)\,\mathrm{C}_{\gD}(\PttD \mathrm{P}_{g} + \PnnG \mathrm{P}_{g}\,\gD).
\]

Comparing this with the expression for $\mathrm{C}_{\gD}\PnnG \Rm_{g}$ above and using the first identity in \eqref{eq:electric_constraints}, we obtain the second.
\end{proof} 
\subsection{Ellipticity of divergence-free gauged Einstein equations}
\label{sec:ellipticity}
We now establish the overdetermined ellipticity referenced in the proof of \propref{prop:overdetermined_elliptiicty_einstein}, specifically that of the system \eqref{eq:overdetermined_Ricci}:
\[
\mathrm{B}_{g}\D\Ric_{g}\oplus\delBianchi\oplus\PttD\oplus\frakT_{g}.
\]
By the equivalence between overdetermined ellipticity and semi-Fredholmness (\secref{sec:OD}), and since $\mathrm{E}_{g}\Hg$ is the leading-order term in \eqref{eq:Ein_variation} (and thus in $\mathrm{B}_{g}\D\Ric_{g}$ via \eqref{eq:Ein_B_Ric}), it suffices to prove
\beq
\begin{split}
\|\sigma\|_{3} &\lesssim  
\|\mathrm{E}_{g}\Hg\sigma\|_{1} + \|\delBianchi\sigma\|_{2} + \|\PttD\sigma\|_{5/2}
+\|\frakT_{g}\sigma\|_{3/2}+ \|\sigma\|_{2}
\end{split}
\label{eq:Einstien_estimate} 
\eeq
for all $\sigma\in \SM$, with ``$\lesssim$'' denoting inequality up to a $\sigma$-independent constant. We establish this by composing semi-Fredholm operators and discarding compact perturbations, under which semi-Fredholmness is stable. 

By \propref{prop:normal_derivative_DA}, $\PttD \oplus \frakT_{g}$ is a normal system of trace operators, admitting a right inverse $I : W^{5/2,2}\plCkm{1}{1} \oplus W^{3/2,2}\plCkm{1}{1} \rightarrow W^{3,2}\Ckm{1}{1}$ (\propref{prop:normality}). Replacing $\sigma$ with $\sigma - I(\PttD\sigma, \frakT_{g}\sigma)$ in \eqref{eq:Einstien_estimate}, we therefore may, instead of \eqref{eq:Einstien_estimate}, demonstrate that for all $\sigma \in \ker(\PttD, \frakT_{g})$: 
\beq
\begin{split}
\|\sigma\|_{3} &\lesssim  
\|\mathrm{E}_{g}\Hg\sigma\|_{1} + \|\delBianchi\sigma\|_{2} + \|\sigma\|_{2}. 
\end{split}
\label{eq:Einstien_estimate1} 
\eeq
(See the procedure in \cite[Thm. 5.13]{KL21a} for an example of how adding back the extension yields the non-homogeneous bound.)

Since the Hodge star $\starG\starGV$ is a Sobolev isometry \cite[p.~40]{Sch95b}, \propref{prop:Eisntein_duality} and the continuity of
\[
\begin{gathered} 
\delBianchi: W^{1,2} \Ckm{1}{1} \to L^{2} \Ckm{0}{1},
\\
\PnnG: W^{1,2} \Ckm{1}{1} \to W^{1/2,2} \scrC_{\dM}^{0,0}, \qquad 
\PntG: W^{1,2} \Ckm{1}{1} \to W^{1/2,2} \scrC_{\dM}^{0,1},
\end{gathered}
\]
yield the lower bound
\[
\|\mathrm{E}_{g}\Hg\sigma\|_{1} \gtrsim \|g^{n-3}\Hg\sigma\|_{1}
\gtrsim \|\delBianchi g^{n-3}\Hg\sigma\|_{0}
+ \|\PnnG g^{n-3}\Hg\sigma\|_{1/2}
+ \|\PntG g^{n-3}\Hg\sigma\|_{1/2}.
\]

Iterating the commutation relations \cite[Sec.~5.2]{KL23},
\[
\delBianchi g^{1} + g^{1}\delBianchi = -\dBianchiV, \qquad 
\PnnG g^{1} = \PttD + \gD^{1}\PnnG, \qquad 
\PntG g^{1} = -g^{1}\PntG,
\]
provides that this becomes, after adding $\|\delta_{g}\sigma\|_{2}+\|\sigma\|_{2}$: 
\beq
\begin{split}
\|\mathrm{E}_{g}\Hg\sigma\|_{1}+\|\delta_{g}\sigma\|_{2}+\|\sigma\|_{2}
&\gtrsim \|g^{n-3}\delBianchi \Hg\sigma +(n-3)! g^{n-4}\dBianchiV\Hg\sigma\|_{0} \\
&\quad + \|\gD^{n-4}\PttD\Hg\sigma +(n-3)! \gD^{n-3}\PnnG\Hg\sigma\|_{1/2} 
\\
&\quad + \|\gD^{n-3}\PntG \Hg\sigma\|_{1/2}+\|\delta_{g}\sigma\|_{2}+\|\sigma\|_{2}.
\end{split}
\label{eq:Eg_lower_bound} 
\eeq

Now, as $\ord(\dBianchiV\Hg) \leq 1$ \cite[Sec.~5.5]{KL23}, $\dBianchiV\Hg : W^{3}\Ckm{1}{1} \to W^{0}\Ckm{2}{3}$ is compact and thus negligible for semi-Fredholmness. Similarly, applying the commutation formulae \cite[p.~706--707]{KL21a} for $\sigma\in\ker(\PttD, \frakT_{g})$, the mappings $\sigma\mapsto\PttD\Hg\sigma:W^{3,2}\cap\ker(\PttD, \frakT_{g})\rightarrow W^{5/2,2}$ and $\sigma\mapsto\PntG \Hg\sigma:W^{3,2}\cap\ker(\PttD, \frakT_{g})\rightarrow W^{3/2,2}$ reduce to compact operations.  

Consequently, returning to \eqref{eq:Eg_lower_bound}, it suffices to show that the system
\beq 
g^{n-3} \delBianchi \Hg \oplus \delBianchi \oplus \PttD \oplus \frakT_{g} \oplus g^{n-3} \PnnG \Hg
\label{eq:effective_system}
\eeq 
is overdetermined elliptic. Indeed, appending the negligible compact terms 
\[
g^{n-4} \dBianchiV \Hg, \qquad \gD^{n-4} \PttD \Hg, \qquad \gD^{n-3} \PntG \Hg,
\]
leaves overdetermined ellipticity unaltered and yields the extended system
\[
(g^{n-3} \delBianchi \Hg + g^{n-4} \dBianchiV \Hg)
\oplus \delBianchi
\oplus \PttD
\oplus \frakT_{g}
\oplus (g^{n-3} \PnnG \Hg + \gD^{n-4} \PttD \Hg)
\oplus \gD^{n-3} \PntG \Hg.
\]
The a priori estimate for this extended system corresponds to the right-hand side of \eqref{eq:Eg_lower_bound} for $\sigma \in \ker(\PttD, \frakT_{g})$, exactly yielding \eqref{eq:Einstien_estimate1} as it controls $\|\sigma\|_{3}$.

It therefore remains to verify the overdetermined ellipticity of \eqref{eq:effective_system}. We first consider the simpler system
\beq 
g^{n-3} \delBianchi \oplus \dBianchi \oplus \PnnG \oplus \gD^{n-4} \PntG,
\label{eq:simpler_system} 
\eeq 
which is clearly overdetermined elliptic: the injective bundle maps $g^{n-3}$ and $\gD^{n-4}$ (cf.~\cite[Prop.~2.5]{Kul72}) define semi-Fredholm mappings, and here they are composed upon the standard covariant de-Rham system, classically known to be overdetermined elliptic (cf.~\cite[Prop.~4.15]{Led25B}):
\[
\delBianchi\oplus \dBianchi\oplus \PnD_{g},
\]
Thus, \eqref{eq:simpler_system} is a composition of semi-Fredholm operators, hence semi-Fredholm itself. Composing then \eqref{eq:simpler_system} on the left with the overdetermined elliptic system from \cite[Sec.~5.5]{KL23},
\[
\Hg \oplus \delBianchi \oplus \PttD \oplus \frakT_{g},
\]
yields a system equivalent to \eqref{eq:effective_system}; the compactness of $\dBianchi \Hg$ and negligibility of $\PntG \Hg$ relative to $\PttD$ and $\frakT_{g}$ ensure these lower-order terms do not alter the semi-Fredholm analysis. Retracing these steps completes the proof.
\appendix

\section{Construction of the lifted complex}
\label{app:Construction}

This appendix aims to distill the technical machinery required for the Hodge theory presented in \secref{sec:Hodge_intro} to take form; namely, to provide the proofs of \thmref{thm:lifted_complex}, \thmref{thm:Hodge_decomposition}, \thmref{thm:cohomology_dirichlet}, and \propref{prop:disrupted_valid}. We do not discuss the theory or its various definitions here, and emphasize that this appendix is included with the purpose of making the paper and the theory as self-contained and clear as possible. 

As explained in \secref{sec:Hodge_intro}, the developments here are essentially an adaptation of \cite[Sec.~4]{KL23} and \cite[Ch.~3.3]{Led25B} to the specific setting of \secref{sec:Hodge_intro}. However, since \cite{Led25B} includes much more than is required here, we deemed it appropriate to assemble this appendix.
\subsection{Adapted Green operators} 
We rely on \cite[Sec.~2]{KL23} for the introduction of Green operators, the Boutet de Monvel calculus, and the notion overdetermined ellipticity. We deviate slightly from the notation used there, and, for simplicity, let $\OP(m,r)$ denote the class of Green operators of order $m\in\bbR$ and class $r\in\bbZ$ between sections of vector bundles $\tGamma(\bbE)\to\tGamma(\bbF)$ (usually implied by context) (cf.~\cite[Sec.~2]{KL23} or \cite[Ch.~3.1]{Led25B}).

The operators of the lifted complex in \thmref{thm:lifted_complex} will be shown to be {adapted Green operators} in the sense of \cite[Def.~3.1]{KL23}. For this purpose, we briefly recall the relevant notions and constructions from \cite[Sec.~3.1]{KL23}. We remind the reader that the theory in \cite{KL23} is formulated for {Neumann} conditions, although this terminology is not used there; cf.~\cite[Ch.~1.2]{Led25B} and \cite[Ch.~3.1]{Led25B} for full generality.

We recall from \cite[Def.~3.1]{KL23} that an \emph{adapted Green operator} $\fbD\in\OP(m,0)$ is an operator $\fbD:\Gamma(\bbE)\rightarrow\Gamma(\bbF)$ in the calculus that is related by Green's formula to its adjoint $\fbD^*\in\OP(m,0)$ (which is also an adapted Green operator), with corresponding normal systems of boundary operators $B,B^*$:
\[
\bra \fbD\psi,\theta\ket=\bra \psi,\fbD^*\theta\ket+\bra B\psi,B^*\theta\ket.  
\]
Given an adapted Green operator $\fbD$, we set
\[
\begin{aligned} 
&\scrR(\fbD):=\image\fbD, \quad &&\scrR(\fbD;\bB):=\image(\fbD|_{\ker\bB}), 
\\
&\scrN(\fbD):=\ker\fbD, \quad &&\scrN(\fbD,\bB):=\ker(\fbD,\bB). 
\end{aligned} 
\]
We define the Sobolev versions as in \cite[Def.~3.2]{KL23} for every $s\geq 0$, restricting here to the case $p=2$, which suffices for our purposes:
\[
\scrR^{s,2}(\fbD), \qquad \scrR^{s,2}(\fbD;\bB), \qquad \scrN^{s,2}(\fbD), \qquad \scrN^{s,2}(\fbD,\bB). 
\]
For the definition of $\scrN^{0,2}(\fbD,\bB)$, one must address the fact the boundary operator does not extend continuously to $L^2$-sections; see \cite[Def.~3.2]{KL23}. In general, we have the $L^{2}$-orthogonal decompositions (\cite[Prop.~3.4]{KL23}):
\beq
\begin{split}
L^{2}(\bbF)=\overline{\scrR^{0,2}(\fbD)}\oplus\scrN^{0,2}(\fbD^*,\bB^*), \qquad 
L^{2}(\bbF)=\overline{\scrR^{0,2}(\fbD;\bB)}\oplus\scrN^{0,2}(\fbD^*).
\end{split}
\label{eq:L2_decom} 
\eeq

An adapted Green operator $\fbD$ is said to induce a {Dirichlet auxiliary decomposition} (to distinguish from the {Neumann} case, \cite[Def.~3.3]{KL23}) if there is a topologically direct, $L^{2}$-orthogonal decomposition of Fréchet spaces
\beq 
\tGamma(\bbF)=\scrR(\fbD;\bB)\oplus\scrN(\fbD^*),
\label{eq:auixllary_decompostion} 
\eeq 
together with an operator $\bP\in\OP(-m,0)$ such that $\tbP:=\fbD\bP\in\OP(0,0)$ is the projection onto $\scrR(\fbD;\bB)$ in this decomposition. We refer to \cite[pp.~32--33]{KL23} and \cite[Ch.~3.2]{Led25B} for further discussion of this construction.

\subsection{Setup of the proof} 
The proof of \thmref{thm:lifted_complex} proceeds by induction on $\alpha\leq \alpha_0$, following the general lines of \cite[Thm.~3.6]{KL23} as presented in \cite[Sec.~4]{KL23}. Again, the argument is contained in the more general treatment of \cite[Ch.~3.3]{Led25B} and, as there, is organized into five stages:
\begin{enumerate}
\item Stage 1: establishing the base of the induction and setting up the induction step, namely the lifted operator $\fbD_{\alpha}$ and the properties it must satisfy.
\item Stages 2--3: proving the required estimates, overdetermined ellipticity, and closed-range properties for the lifted operator.
\item Stage 4: showing that each lifted operator $\fbD_{\alpha}$ induces an auxiliary decomposition.
\item Stage 5: defining the next operator $\fbD_{\alpha+1}$, as characterized by \thmref{thm:lifted_complex}, and proving that it lies in the calculus and differs from $\bD_{\alpha+1}$ by an $\OP(0,0)$ term. 
\end{enumerate}
Afterwards, in \appref{app:Hodge_proof}, once all relevant operators have been lifted as required by \thmref{thm:lifted_complex}, we establish the Hodge decompositions for Dirichlet conditions (\thmref{thm:Hodge_decomposition}), which in the notation above reads as:
\beq 
\tGamma(\bbE_{\alpha+1}) =
\lefteqn{\overbrace{\phantom{\scrR{(\bA_{\alpha};\ker B_{\alpha})}\oplus \module_{\D}^{\alpha+1}}}^{\scrN(\bA_{\alpha+1},B_{\alpha+1})}}
\scrR{(\bA_{\alpha}; B_{\alpha})}
\oplus
\underbrace{\module_{\D}^{\alpha+1}\oplus \scrR{(\bA^*_{\alpha+1}})}_{\scrN (\bA_{\alpha}^*)}.
\label{eq:Hodge_smooth_proof} 
\eeq 

In \appref{app:disrupted_proof}, we treat the disrupted case and prove \propref{prop:disrupted_valid}. 

\subsection{Stage 1: Base and setup of induction step}
\label{sec:stage1}
\subsubsection{Induction base} 
For the base of the induction, it is convenient to set
\[
\fbD_{-1}:=0, \qquad \fbD_{0}:=\bD_{0},
\]
and to start at the level $\alpha=-1$. At this initial level, the induction base requires the following conditions:

\begin{enumerate}[itemsep=0pt,label=(\alph*)]
\item $\fbD_{-1}$ induces an auxiliary decomposition:
\[
\tGamma(\bbE_0) = \scrR(\fbD_{-1};\bB_{-1}) \oplus \scrN(\fbD_{-1}).
\]
\item It holds that
\[
\scrR(\fbD_{-1};\bB_{-1}) \subseteq \scrN(\fbD_0).
\]
\item The following relation holds:
\[
\fbD_0 = \bD_0 \quad \text{ on } \quad \scrN(\fbD_{-1}^*).
\]
\item $\fbD_0 = \bD_0 + \bC_0$ is an adapted Green operator, with
\[
\bC_0=-A_{0}A_{-1}\bP_{-1}.
\]
\end{enumerate}

Since $\fbD_{-1}=0$, these requirements are satisfied trivially: indeed, $\scrR(\fbD_{-1};\bB_{-1})=\{0\}$ while $\scrN(\fbD_{-1}^*)=\tGamma(\bbE_{0})$ as $\fbD_{-1}^*=0$ and $\bB_{-1}^*=0$; condition (c) then implies that $\fbD_0=\bD_0$ identically, which is indeed the case. Condition (d) is also satisfied trivially, since $\fbD_0=\bD_0$ forces $\bC_0=0$.

\subsubsection{Induction Hypothesis} 
Throughout, let $\alpha\leq \alpha_0$ and set
\[
m_{\alpha}:=\ord(\bD_{\alpha}).
\] 
Since $\fbD_{\alpha}$ is a lifting of $\bD_{\alpha}$ in the sense of \thmref{thm:lifted_complex}, then, as the difference is in $\OP(0,0)$, we have $\fbD_{\alpha}\in\OP(m_{\alpha},0)$.

With this understood, interpreting the conditions in \thmref{thm:lifted_complex} and the ingredients of an auxiliary decomposition \eqref{eq:auixllary_decompostion}, the induction hypothesis is that $\fbD_{\alpha}$ and $\fbD_{\alpha-1}$ have been defined so that the following hold:

\begin{enumerate}[itemsep=0pt,label=(\alph*)]
\item $\fbD_{\alpha-1}$ induces a Dirichlet auxiliary decomposition as in \eqref{eq:auixllary_decompostion}. We emphasis that this corresponds to the existence of $\bP_{\alpha-1}\in\OP(-m_{\alpha-1},0)$ such that $\tbP_{\alpha-1}=\fbD_{\alpha-1}\bP_{\alpha-1}\in\OP(0,0)$ is the continuous projection onto the range in the following $L^2$-orthogonal, topologically direct decompositions:
\beq
\begin{aligned}
\tGamma(\bbE_{\alpha}) = \scrR(\fbD_{\alpha-1}; \bB_{\alpha-1}) \oplus \scrN(\fbD^*_{\alpha-1}).
\end{aligned}
\label{eq:aux_induction}
\eeq
For convenience, we also include in the induction hypotheses that by intersecting with $\ker\bB_{\alpha}$ we have: 
\beq
\begin{aligned}
\tGamma(\bbE_{\alpha}) \cap \ker \bB_{\alpha}
= \scrR(\fbD_{\alpha-1}; \bB_{\alpha-1})
\oplus (\scrN(\fbD_{\alpha-1}^*) \cap \ker \bB_{\alpha}).
\label{eq:aux_induction_Refined}
\end{aligned}
\eeq
\item The following containment holds:
\beq
\begin{aligned}
\scrR(\fbD_{\alpha-1}; \bB_{\alpha-1}) \subseteq \scrN(\fbD_\alpha),
\end{aligned} 
\label{eq:containments_induction} 
\eeq
with the further identity:
\beq
\scrR(\fbD_{\alpha-1}; \bB_{\alpha-1}) \subseteq \ker\bB_{\alpha}.
\label{eq:containments_inductionM} 
\eeq

\item The following relations hold:
\beq
\begin{aligned}
\fbD_\alpha = \bD_\alpha \text{ on } \scrN(\fbD_{\alpha-1}^*),
\end{aligned} 
\label{eq:relations_induction}
\eeq
\item $\fbD_{\alpha} = \bD_{\alpha} + \bC_{\alpha}$ is an adapted Green operator, with boundary operators $B_{\alpha},B_{\alpha}^*$, $\bC_{\alpha} \in \OP(0,0)$ and
\beq
\bC_{\alpha} = -\bD_{\alpha} \bD_{\alpha-1} \bP_{\alpha-1} = -\bD_{\alpha} \tbP_{\alpha-1}.
\label{eq:correction_induction}
\eeq
\end{enumerate}

\subsubsection{Induction step} Under the induction hypothesis, the remainder of the next four subsections is devoted to proving the following:
\begin{enumerate}[itemsep=0pt,label=(\alph*)]
\item The operator $\fbD_{\alpha}$ induces an auxiliary decomposition as in \eqref{eq:auixllary_decompostion}: specifically, there exists a map $\bP_{\alpha}\in\OP(-m_{\alpha},0)$ such that the operator $\tbP_{\alpha}=\fbD_{\alpha}\bP_{\alpha}\in\OP(0,0)$ is the continuous projection onto the range in the following $L^2$-orthogonal, topologically direct decomposition:
\beq
\begin{aligned}
\tGamma(\bbE_{\alpha+1}) = \scrR(\fbD_{\alpha}; \bB_{\alpha}) \oplus \scrN(\fbD^*_{\alpha}).
\end{aligned}
\label{eq:aux_induction_step}
\eeq
We also establish the refined version:
\beq
\begin{aligned} 
\tGamma(\bbE_{\alpha+1}) \cap \ker \bB_{\alpha+1}
= \scrR(\fbD_{\alpha}; \bB_{\alpha})
\oplus \left(\scrN(\fbD_{\alpha}^*) \cap \ker \bB_{\alpha+1} \right).
\label{eq:aux_induction_stepRefined}
\end{aligned} 
\eeq 
\item There exists a system $\fbD_{\alpha+1}:\tGamma(\bbE_{\alpha+1})\rightarrow\tGamma(\bbF_{\alpha+2};\bbG_{\alpha+2})$ such that the following containment hold:
\beq
\begin{aligned}
\scrR(\fbD_{\alpha};\bB_{\alpha}) \subseteq \scrN(\fbD_{\alpha+1}),
\end{aligned} 
\label{eq:contain_induction_step}
\eeq
with the further identity:
\beq
\scrR(\fbD_{\alpha};\bB_{\alpha}) \subseteq \ker\bB_{\alpha+1}.
\label{eq:contain_induction_stepM} 
\eeq
\item The following relations hold:
\beq
\begin{aligned}
\fbD_{\alpha+1} = \bD_{\alpha+1} \text{ on } \scrN(\fbD_{\alpha}^*),
\end{aligned} 
\label{eq:relation_induction_step}
\eeq
\item $\fbD_{\alpha+1} = \bD_{\alpha+1} + \bC_{\alpha+1}$ is an adapted Green operator, with boundary operators $B_{\alpha+1},B_{\alpha+1}^*$, $\bC_{\alpha+1} \in \OP(0,0)$ and
\beq
\bC_{\alpha+1} = -\bD_{\alpha+1} \bD_{\alpha} \bP_{\alpha} = -\bD_{\alpha+1} \tbP_{\alpha}.
\label{eq:correction_induction_step}
\eeq
\end{enumerate}
\subsection{Stage 2: Additional elliptic estimates}
Consider the system:
\beq
\begin{split}
\fbD_{\alpha}\oplus\fbD_{\alpha-1}^* \oplus \bB_{\alpha}.
\end{split}
\label{eq:corrected_D_N_adjoints}
\eeq
Due to the induction hypothesis, we have that $\fbD_{\alpha}-\bD_{\alpha}\in\OP(0,0)$. Hence, by comparison with the overdetermined ellipticity in \defref{def:elliptic_pre_complex}, it follows that the system in \eqref{eq:corrected_D_N_adjoints} is also overdetermined elliptic, as it differs from the original overdetermined elliptic system only by lower-order terms (which, from the perspective of \secref{sec:OD}, are compact perturbations).

Now, by the composition rules of the calculus, the same holds for the corresponding lifted systems obtained by composing the overdetermined ellipticities at $\alpha$ and $\alpha-1$, and neglecting compact terms resulting from the order-reduction properties in \thmref{thm:lifted_complex} (see \cite[Prop.~4.3]{KL23} for an explicit demonstration how this is done on the symbol level):
\beq
\begin{aligned}
\fbD^*_{\alpha} \fbD_{\alpha} \oplus \fbD^*_{\alpha-1} \oplus \bB_{\alpha}.
\end{aligned}
\label{eq:corrected_D_N_adjointsII}
\eeq

We record these observations as a corollary:
\begin{corollary}
\label{corr:overdetermined_ellipticity_corrected}
Under the induction hypothesis, the systems in \eqref{eq:corrected_D_N_adjoints} and \eqref{eq:corrected_D_N_adjointsII} are overdetermined elliptic.
\end{corollary}

The kernel of an overdetermined elliptic systems is always finite dimensional and consists of smooth sections (cf.~\cite[Prop.~2.9]{KL23}) which due to the above translates to the finite-dimensionality of:
\beq
\begin{aligned}
\module^{\alpha}_{\D}:=\ker(\fbD_{\alpha},\fbD_{\alpha-1}^*,\bB_{\alpha}) = \ker(\bD_{\alpha},\fbD_{\alpha-1}^*,\bB_{\alpha}).
\end{aligned}
\label{eq:cohomology_in_the_proof}
\eeq
Here, the identity $\fbD_{\alpha} = \bD_{\alpha}$ on the kernel follows from the induction hypothesis \eqref{eq:relations_induction}, i.e., that $\fbD_\alpha = \bD_\alpha$ on $\scrN(\fbD_{\alpha-1}^*)$.
\begin{proposition}
\label{prop:coinciding_kernels}
The kernel of the systems in \eqref{eq:corrected_D_N_adjointsII} coincide with $\module^{\alpha}_{\D}$.
\end{proposition}

\begin{proof}
In the first direction, the inclusion:
\[
\begin{aligned}
\module_{\D}^{\alpha}\subseteq \ker(\fbD^*_{\alpha} \fbD_{\alpha},\fbD^*_{\alpha-1},\bB_{\alpha}),
\end{aligned}
\]
follow directly from comparing with \eqref{eq:cohomology_in_the_proof} above.  To prove the reverse inclusions, let $
\psi \in \ker(\fbD^*_{\alpha} \fbD_{\alpha},\fbD^*_{\alpha-1},\bB_{\alpha})$. In particular,
\[
\begin{aligned}
\psi \in \scrN(\fbD_{\alpha-1}^*) \cap \ker\bB_{\alpha} &&&\textand &&&&&\fbD_{\alpha} \psi \in \scrN(\fbD_{\alpha}^*).
\end{aligned}
\]

Comparing with the $L^2$-orthogonal decompositions in \eqref{eq:L2_decom}, which hold for any adapted Green operator, including $\fbD_{\alpha}$, we find that:
\[
\begin{aligned}
\fbD_{\alpha} \psi \in \scrN(\fbD_{\alpha}^*) \cap \scrR(\fbD_{\alpha}; \bB_{\alpha}) = \BRK{0}.
\end{aligned}
\]
This implies $\fbD_{\alpha} \psi = 0$, and to summarize $
\psi \in \ker(\fbD_{\alpha},\fbD_{\alpha-1}^*,\bB_{\alpha}) = \module^{\alpha}_{\D}$, 
which completes the proof.
\end{proof}

Let $\frakI_{\alpha}\in\OP(-\infty,-\infty)$ denote the $L^2$-orthogonal projection onto the finite-dimensional kernel $\module_{\DD}^{\alpha}$ (every overdetermined elliptic problem admits such a projection onto its kernel; cf.\ \cite[pp.~26--27]{KL23}). Direct summing $\frakI_{\alpha}$ with the systems in either \eqref{eq:corrected_D_N_adjoints} or \eqref{eq:corrected_D_N_adjointsII} yields injective systems by \propref{prop:coinciding_kernels}. Therefore, by the properties of overdetermined systems in the calculus (cf.\ \cite[Prop.~2.9]{KL23}), this system admits a left inverse within the calculus, allowing us to obtain:  

\begin{proposition}
\label{prop:new_estimate1}
The following system is overdetermined elliptic and injective, (of maximal order $m_{\alpha}$ and class $m_{\alpha}$):
\beq
\begin{aligned}
\fbD_{\alpha} \oplus \bB_{\alpha} \oplus \tbP_{\alpha-1} \oplus \frakI_{\alpha}.
\end{aligned}
\label{eq:order_reduced_elliptiices_proof}
\eeq
Here, $\tbP_{\alpha-1}\in\OP(0,0)$ is the operator from the induction hypothesis, associated with the auxiliary decomposition \eqref{eq:aux_induction}. 
\end{proposition}
\begin{proof}
By \cite[Prop.~2.9]{KL23}, if a system of varying orders has a left-inverse within the calculus, then it is overdetermined elliptic and injective.

With this given, we note that the auxiliary decomposition induced by $\fbD_{\alpha-1}$, assumed in the induction step \eqref{eq:aux_induction}, combined with the properties of $\tbP_{\alpha-1}$, implies that $(\id - \tbP_{\alpha-1})$ is the projection onto $\scrN(\fbD_{\alpha-1}^*)$. Thus, for every $\psi \in \tGamma(\bbE_{\alpha})$, we have:
\[
\fbD_{\alpha-1}^* \tbP_{\alpha-1} \psi = \fbD_{\alpha-1}^* \psi.
\]
We therefore obtain the identity:
\[
\fbD_{\alpha} \oplus \bB_{\alpha} \oplus \fbD_{\alpha-1}^* \oplus \frakI_{\alpha} 
= (\id \oplus \id \oplus \fbD_{\alpha-1}^* \oplus \id)
(\fbD_{\alpha} \oplus \bB_{\alpha} \oplus \tbP_{\alpha-1} \oplus \frakI_{\alpha}).
\]
By the overdetermined ellipticity established in \corrref{corr:overdetermined_ellipticity_corrected}, together with the identification of the system’s kernel in \propref{prop:coinciding_kernels}, the system on the left-hand side is injective and admits a left inverse within the calculus. Composing on the left with this left inverse, we obtain that the system $\fbD_{\alpha} \oplus \bB_{\alpha} \oplus \tbP_{\alpha-1} \oplus \frakI_{\alpha}$ also admits a left inverse in the calculus; hence it is overdetermined elliptic, and the claim follows.
\end{proof}
\begin{proposition}
\label{prop:estimatefbD} 
The following estimate holds for every $s> m_{\alpha}-\tfrac{1}{2}$ and every $\psi \in W^{s,2}(\bbE_{\alpha}) \cap \ker \bB_{\alpha}$:
\beq
\|\psi\|_{s} \lesssim \|\fbD_{\alpha}\psi\|_{s-m_{\alpha}} + \|\tbP_{\alpha-1}\psi\|_{s} + \|\frakI_{\alpha}\psi\|_{0},
\label{eq:Ak_overdetermined_elliptic_estimate}
\eeq
\end{proposition}

\begin{proof}
The estimate follows directly from the a priori estimate associated with the overdetermined ellipticity in \eqref{eq:order_reduced_elliptiices_proof} (cf.\ \cite[Prop.~2.9]{KL23}), together with the observation that, if $\psi\in\ker \bB_{\alpha}$, then the $\bB_{\alpha}\psi$ term vanishes.
\end{proof}
The next proposition is proven similarly to \propref{prop:new_estimate1}--\propref{prop:estimatefbD}, albeit using the second set of overdetermined ellipticities in \corrref{corr:overdetermined_ellipticity_corrected} (see the argument \cite[Prop.~4.3]{KL23} for the details):
\begin{proposition}
\label{prop:new_estimate2}
The following systems are overdetermined elliptic and injective (of maximal order $2m_{\alpha}$ and class $m_{\alpha}$):
\beq
\begin{aligned}
\fbD_{\alpha}^* \fbD_{\alpha} \oplus \bB_{\alpha}\oplus \tbP_{\alpha-1} \oplus \frakI_{\alpha}.
\end{aligned}
\label{eq:order_reduced_elliptiices_proof2}
\eeq 
Hence, for for every $s>-\tfrac{1}{2}+m_{\alpha}$ and $\psi \in W_2^{s}(\bbE_{\alpha}) \cap \ker\bB_{\alpha}$: 
\beq
\begin{aligned}
\|\psi\|_{s} &\lesssim \|\fbD_{\alpha}^*\fbD_{\alpha}\psi\|_{s-2m_{\alpha}} + \|\tbP_{\alpha-1}\psi\|_{s} + \|\frakI_{\alpha}\psi\|_{0},
\end{aligned}
\label{eq:D^*D_estimate}
\eeq
\end{proposition}
\subsection{Stage 3: Closed range arguments and a priori estimates}
By construction, for every $s\geq 0$ we have
\[
\begin{aligned}
\module_{\D}^{\alpha} \subseteq \scrN^{s,2}(\fbD_{\alpha-1}^*).
\end{aligned}
\]

Thus, writing $\id = (\id - \frakI_{\alpha}) + \frakI_{\alpha}$ when restricted to the spaces on the right-hand side, we obtain the topologically direct splittings
\[
\begin{aligned}
\scrN^{s,2}(\fbD_{\alpha-1}^*) = \scrN_{2,\bot}^{s}(\fbD_{\alpha-1}^*) \oplus \module_{\D}^{\alpha}.
\end{aligned}
\]

These decompositions are, by construction, $L^2$-orthogonal for every $s\geq 0$. By the graded Fréchet structure, they therefore yield $L^2$-orthogonal, topologically direct decompositions of Fréchet spaces:
\[
\begin{aligned}
\scrN(\fbD_{\alpha-1}^*) = \scrN_{\bot}(\fbD_{\alpha-1}^*) \oplus \module_{\D}^{\alpha}.
\end{aligned}
\]

Combined with the auxiliary decomposition \eqref{eq:aux_induction} induced by $\fbD_{\alpha}$, we obtain the following $L^2$-orthogonal, topologically direct decompositions:
\beq
\begin{aligned}
\tGamma(\bbE_{\alpha}) = \scrR(\fbD_{\alpha-1}; \bB_{\alpha-1}) \oplus \scrN_{\bot}(\fbD_{\alpha-1}^*) \oplus \module_{\D}^{\alpha},
\end{aligned}
\label{eq:N0spdecomp}
\eeq
and the projection onto the middle summand is given by
\[
\tbP_{\alpha-1}^\bot = (\id - \frakI_{\alpha})(\id - \tbP_{\alpha-1}).
\]
By the composition rules of Green operators, $\tbP_{\alpha-1}^\bot \in \OP(0,0)$. Since all projections onto the closed subspaces in \eqref{eq:N0spdecomp} lie in $\OP(0,0)$, it follows by a density and continuity argument that, for every $s\geq 0$,
\beq
\begin{aligned}
W^{s,2}(\bbE_{\alpha}) = \scrR^{s,2}(\fbD_{\alpha-1}; \bB_{\alpha-1}) \oplus \scrN_{2,\bot}^{s}(\fbD_{\alpha-1}^*) \oplus \module^{\alpha}_{\D},
\end{aligned}
\label{eq:WspN0}
\eeq
and, for every $s\geq 0$ as above,
\beq 
\begin{aligned}
\scrN(\fbD_{\alpha-1}^*) \hookrightarrow \scrN^{s,2}(\fbD_{\alpha-1}^*).
\end{aligned}
\label{eq:first_inclusions} 
\eeq 
\begin{lemma}
\label{lem:density_Nsp}
For every $s\geq 0$, the continuous inclusions
\[
\begin{aligned}
\scrN_{\bot}(\fbD_{\alpha-1}^*) \hookrightarrow \scrN^{s,2}_{\bot}(\fbD_{\alpha-1}^*)
\end{aligned}
\]
are dense. Moreover, if $\psi \in W^{s,2}(\bbE_{\alpha}) \cap \ker\bB_{\alpha}$, then $\tbP_{\alpha-1}^{\bot} \psi \in W^{s,2}(\bbE_{\alpha}) \cap \ker\bB_{\alpha}$. Hence, there is an additional dense inclusion:
\[
\scrN_{\bot}(\fbD_{\alpha-1}^*) \cap \ker\bB_{\alpha} \hookrightarrow \scrN^{s,2}_{\bot}(\fbD_{\alpha-1}^*) \cap \ker\bB_{\alpha}.
\]
\end{lemma}
\begin{proof}
The required inclusions are obtained by applying the projection $\id-\frakI_{\alpha}$ to both sides of \eqref{eq:first_inclusions}. For density, let $\psi\in \scrN^{s,2}_{\bot}(\fbD_{\alpha-1}^*)$. Since $\tGamma(\bbE_{\alpha})$ is dense in $W^{s,2}(\bbE_{\alpha})$, there exists a sequence $\psi_{n}\in \tGamma(\bbE_{\alpha})$ such that
\[
\psi_{n}\to \psi \qquad \text{in } W^{s,2}.
\]
Since $\tbP_{\alpha-1}^{\bot}\in\OP(0,0)$, it is $W^{s,2}\to W^{s,2}$ continuous for every $s\geq 0$. Moreover, by construction, the map
\[
\tbP_{\alpha-1}^{\bot}:W^{s,2}(\bbE_{\alpha})\to W^{s,2}(\bbE_{\alpha})
\]
is the continuous projection onto the closed subspace $\scrN^{s,2}_{\bot}(\fbD_{\alpha-1}^*)$, hence $\tbP_{\alpha-1}^{\bot}\psi=\psi$. Therefore, by continuity,
\[
\scrN_{\bot}(\fbD_{\alpha-1}^*) \ni \tbP_{\alpha-1}^{\bot}\psi_n \to \tbP_{\alpha-1}^{\bot}\psi=\psi,
\]
which proves the first density.

For the second density, since $\module_{\D}^{\alpha}\subseteq \ker \bB_{\alpha}$, the refined decomposition in \eqref{eq:aux_induction_Refined} implies that, when restricted to $\ker \bB_{\alpha}$, the operator $\tbP^{\bot}_{\alpha-1}$ becomes the projection onto
\[
(\scrN_{\bot}(\fbD_{\alpha-1}^*)\cap\ker\bB_{\alpha}) \oplus \module_{\D}^{\alpha}.
\]
The corresponding Sobolev version follows by continuity, and the density argument proceeds exactly as above.
\end{proof}
\begin{proposition}
\label{prop:Ak_closed_range}
For every $s\geq 0$, the subspace $\scrR^{s,2}(\fbD_{\alpha}; \bB_{\alpha}) \subseteq W^{s,2}(\bbE_{\alpha+1})$ is closed. Moreover, for every $\psi \in \scrN^{s+m_{\alpha},2}_{\bot}(\fbD_{\alpha-1}^*) \cap \ker\bB_{\alpha}$, the following estimate holds:
\beq
\|\psi\|_{s+m_{\alpha}} \lesssim \|\fbD_{\alpha} \psi\|_{s}.
\label{eq:estimateNsp}
\eeq
\end{proposition}

\begin{proof}
Combining the dense inclusion in \eqref{eq:first_inclusions} with the decomposition \eqref{eq:WspN0}, we find that \eqref{eq:aux_induction_Refined} completes into the Sobolev version:
\[
W^{s+m_{\alpha},2}(\bbE_{\alpha}) \cap \ker\bB_{\alpha}
=
\scrR^{s+m_{\alpha},2}(\fbD_{\alpha-1}; \bB_{\alpha-1})
\oplus
(\scrN^{s+m_{\alpha},2}_{\bot}(\fbD_{\alpha-1}^*) \cap \ker\bB_{\alpha})
\oplus
\module_{\D}^{\alpha}.
\]
Using the relations in the induction hypothesis \eqref{eq:containments_induction} and continuity, we obtain that $\fbD_{\alpha}$ vanishes identically on the first and third summands of this decomposition. Hence:
\beq
\begin{aligned}
\fbD_{\alpha}(W^{s+m_{\alpha},2}(\bbE_{\alpha}) \cap \ker\bB_{\alpha})
=
\fbD_{\alpha}(\scrN^{s+m_{\alpha},2}_{\bot}(\fbD_{\alpha-1}^*) \cap \ker\bB_{\alpha}).
\end{aligned}
\label{eq:ranges_in_the_proof}
\eeq

Now, for any $\psi \in \scrN^{s+m_{\alpha},2}_{\bot}(\fbD_{\alpha-1}^*) \cap \ker\bB_{\alpha}$, we have by construction that $\tbP_{\alpha-1}\psi=0$, $\frakI_{\alpha}\psi=0$, and $\bB_{\alpha}\psi=0$, hence the elliptic estimate \eqref{eq:Ak_overdetermined_elliptic_estimate} reduces to \eqref{eq:estimateNsp}. This implies that the ranges on the right-hand side in \eqref{eq:ranges_in_the_proof} are closed subspaces (e.g., by \cite[Prop.~6.7, p.~583]{Tay11a}). Comparing with the subspaces on the left-hand side of \eqref{eq:ranges_in_the_proof}, we conclude that the ranges in the statement of the proposition are indeed closed.
\end{proof}

Since the proposition holds for any $s\geq 0$, by the Sobolev grading we have:
\begin{corollary}
$\scrR(\fbD_{\alpha}; \bB_{\alpha}) \subseteq \tGamma(\bbE_{\alpha+1})$ is closed in the Fréchet topology. 
\end{corollary}
The next proposition is the analytical heart of the induction step. It is an adaptation of \cite[Prop.~4.7--4.8]{KL23} to the Dirichlet setting. We remark that although  \cite[Prop.~4.7--4.8]{KL23} there are valid as stated, the analytical argument in their proof is flawed: distributional limits were incorrectly identified with weak $L^2$ limits. Here we remedy this by a more careful treatment; see \cite[Prop.~3.44]{Led25B} for the argument in its highest generality.
\begin{proposition}
\label{prop:AprioriAstarAk} 
Let $\theta \in \scrR^{0,2}(\fbD_{\alpha}; \bB_{\alpha})$. Let $s\geq 0$, and suppose there exists $\xi \in W^{s,2}(\bbE_{\alpha+1})$ such that, 
\[
\begin{aligned}
\theta - \xi \in \scrN^{0,2}(\fbD_{\alpha}^*).
\end{aligned}
\]
Then there exists a
\beq
\begin{aligned}
\psi \in \scrN^{s+m_{\alpha},2}_{\bot}(\fbD_{\alpha-1}^*) \cap \ker\bB_{\alpha},
\end{aligned}
\label{eq:potential_proof}
\eeq
such that $\theta = \fbD_{\alpha} \psi$ and the following estimates hold:
\beq
\begin{aligned}
\|\psi\|_{s+m_{\alpha}} \lesssim \|\fbD_{\alpha}^* \fbD_{\alpha} \psi\|_{s-m_{\alpha}}.
\end{aligned}
\label{eq:estimateN0p2}
\eeq
\end{proposition}
\begin{proof}
By the assumption on $\theta$ and the fact that the smooth version of \eqref{eq:ranges_in_the_proof} reads as
\[
\begin{aligned}
\scrR(\fbD_{\alpha}; \bB_{\alpha}) = \fbD_{\alpha}(\scrN_{\bot}(\fbD_{\alpha-1}^*) \cap \ker\bB_{\alpha}),
\end{aligned}
\]
together with the density inclusions in \lemref{lem:density_Nsp}, we find that there exists a sequence $(\psi_n)$ of smooth sections in
\beq
\begin{aligned}
(\psi_n) \subset \scrN_{\bot}(\fbD_{\alpha-1}^*) \cap \ker\bB_{\alpha},
\end{aligned}
\label{eq:approximating_sequence_prop_proof}
\eeq
such that $\fbD_{\alpha} \psi_n \to \theta$ in $L^{2}$. Using \eqref{eq:Ak_overdetermined_elliptic_estimate} for $s=m_{\alpha}$, we may assume that $\psi_n \to \psi$ in $L^{2}$ and that $\theta = \fbD_{\alpha} \psi$.

Let now $\frakS$ be the left inverse of the injective system from \propref{prop:new_estimate2}. By the construction of the approximating sequence $(\psi_{n})$ above, we have
\beq 
\begin{aligned} 
\frakS(\fbD_{\alpha}^* \fbD_{\alpha} \psi_n, 0, 0, 0) = \psi_{n}.
\end{aligned} 
\label{eq:inverse_analytical_lemma}
\eeq 

Restricting this inverse to the relevant arguments gives the Green operator
\[
\begin{aligned} 
\frakS(\cdot,0;0,0): \tGamma(\bbE_{\alpha}) \to \tGamma(\bbE_{\alpha}).
\end{aligned} 
\]
By the composition rules of the calculus (cf.~\cite[Thm.~2.3]{KL23}), $\fbD_{\alpha}^* \fbD_{\alpha}$ is of order $2m_{\alpha}$ and class $m_{\alpha}$. Hence, by the properties of the left inverse (\cite[Prop.~2.9]{KL23}), the operator $\frakS(\cdot,0;0,0)$ is of class at most $-m_{\alpha}$. In particular, for sufficiently small $s' < 0$ we have the continuous mapping property
\[
\begin{aligned} 
\frakS(\cdot,0;0,0): W^{-m_{\alpha},2}(\bbE_{\alpha}) \to W^{s',2}(\bbE_{\alpha}).
\end{aligned} 
\]
We remark that this map may well be compact; it guarantees continuity between Sobolev norms but does not imply mapping into a space sufficient to invert the original problem.

Now, since $\fbD_{\alpha}^* \fbD_{\alpha}$ is of order $2m_{\alpha}$ and class $m_{\alpha}$, we have
\[
\fbD_{\alpha}^* \fbD_{\alpha} : W^{m_{\alpha},2}(\bbE_{\alpha}) \rightarrow W^{-m_{\alpha},2}(\bbE_{\alpha})
\]
continuously. Therefore, by continuity 
\[
\lim_{n \to \infty} \fbD_{\alpha}^* \fbD_{\alpha} \psi_n = \fbD_{\alpha}^* \xi, \quad \text{in } W^{-m_{\alpha},2}.
\]
On the other hand, by the continuity of $\frakS$ as above, we obtain in $W^{s',2}$
\[
\begin{aligned} 
\lim_{n \to \infty} \frakS(\fbD_{\alpha}^* \fbD_{\alpha} \psi_n, 0, 0, 0) 
= \frakS(\fbD_{\alpha}^* \xi, 0, 0, 0).
\end{aligned} 
\]

Comparing with \eqref{eq:inverse_analytical_lemma} and using uniqueness of limits (since $L^{2}$ convergence is stronger than $W^{s',2}$ convergence for sufficiently small $s'$) gives
\[
\begin{aligned} 
\frakS(\fbD_{\alpha}^* \xi, 0, 0, 0) = \psi.
\end{aligned} 
\]

Finally, since $\fbD_{\alpha}^* \xi \in W^{s-m_{\alpha},2}$ by assumption and the fact that the left inverse $\frakS$ is of order $-2m_{\alpha}$, we obtain $\psi \in W^{s+m_{\alpha},2}$ and, in particular, that it satisfies \eqref{eq:potential_proof}. The estimates in \eqref{eq:D^*D_estimate} then follow, yielding \eqref{eq:estimateN0p2} and thus completing the proof.
\end{proof}

\subsection{Stage 4: The auxiliary decomposition}
We now apply \eqref{eq:L2_decom} to the adapted Green operator $\fbD_{\alpha}$. Since $\fbD_{\alpha}$ has closed range, the right decomposition there translate directly into
\beq
\begin{aligned}
L^{2}(\bbE_{\alpha+1}) = \scrR^{0,2}(\fbD_{\alpha}; \bB_{\alpha}) \oplus \scrN^{0,2}(\fbD_{\alpha}^*).
\end{aligned}
\label{eq:L2_splitting_proof}
\eeq

\begin{proposition}
\label{prop:aux_first_step}
There exists an $L^{2}$-orthogonal, topologically direct decomposition of Fréchet spaces
\beq
\begin{aligned}
\tGamma(\bbE_{\alpha+1}) = \scrR(\fbD_{\alpha}; \bB_{\alpha}) \oplus \scrN(\fbD^*_{\alpha}).
\end{aligned}
\label{eq:aux_decompositionAksmooth}
\eeq
Moreover, the continuous projection onto $\scrR(\fbD_{\alpha}; \bB_{\alpha})$ associated with this decomposition extends continuously to the $L^{2}$-orthogonal projection onto $\scrR^{0,2}(\fbD_{\alpha};\bB_{\alpha})$ in \eqref{eq:L2_splitting_proof}.
\end{proposition}

\begin{proof}
On the one hand, $\scrR^{s,2}(\fbD_{\alpha}; \bB_{\alpha}) \subseteq W^{s,2}(\bbE_{\alpha+1})$ is closed due to \propref{prop:Ak_closed_range}. On the other hand,  $\scrN^{s,2}(\fbD^*_{\alpha})\subseteq W^{s,2}(\bbE_{\alpha+1})$ is closed as it is the kernel of a continuous linear map ($\fbD^*_{\alpha}$ has class zero so it is continuous on $W^{s,2}$ for every $s\geq 0$). Hence, together with the containments:
\beq
\begin{aligned}
\scrR^{s,2}(\fbD_{\alpha}; \bB_{\alpha}) \subseteq \scrR^{0,2}(\fbD_{\alpha}; \bB_{\alpha}), &&&&\scrN^{s,2}(\fbD_{\alpha}^*) \subseteq \scrN^{0,2}(\fbD_{\alpha}^*),
\end{aligned}
\eeq
one finds that these subspaces are closed, intersect trivially, and are mutually $L^{2}$-orthogonal. Thus, to prove:
\beq
\begin{aligned}
W^{s,2}(\bbE_{\alpha+1}) = \scrR^{s,2}(\fbD_{\alpha}; \bB_{\alpha}) \oplus \scrN^{s,2}(\fbD_{\alpha}^*),
\end{aligned}
\label{eq:this_holds}
\eeq
it remains to show that the sum of the spaces on the right exhausts the whole of $W^{s,2}(\bbE_{\alpha+1})$. By the Sobolev grading of the Fréchet space $\tGamma(\bbE_{\alpha+1})$, if \eqref{eq:this_holds} holds for every $s\geq 0$, then \eqref{eq:aux_decompositionAksmooth} holds as well.

To prove this exhaustion, let $\xi \in W^{s,2}(\bbE_{\alpha+1})$. Decompose it as an element in $L^{2}(\bbE_{\alpha+1})$ according to \eqref{eq:L2_splitting_proof}:
\[
\xi = \theta + \phi,
\]
where $\theta \in \scrR^{0,2}(\fbD_{\alpha}; \bB_{\alpha})$ and $\phi \in \scrN^{0,2}(\fbD_{\alpha}^*)$. 
Since $\xi \in W^{s,2}(\bbE_{\alpha+1})$, and $\phi = \xi - \theta\in \scrN^{0,2}(\fbD_{\alpha}^*)$, \propref{prop:AprioriAstarAk} applies, yielding $\theta = \fbD_{\alpha} \psi$ for some $\psi \in W^{s+m_{\alpha},2}(\bbE_{\alpha})$ with $\bB_{\alpha} \psi = 0$. Therefore:
\[
\begin{aligned}
\phi =\xi-\fbD_{\alpha}\psi\in \scrN^{0,2}(\fbD_{\alpha}^*) \cap W^{s,2}(\bbE_{\alpha+1}) = \scrN^{s,2}(\fbD_{\alpha}^*).
\end{aligned}
\]
This completes the proof. The $L^{2}$-continuity of the projections follows directly from this construction.
\end{proof}

\begin{theorem}
\label{thm:aux_induction_step}
$\fbD_{\alpha}$ induces an auxiliary decomposition as in \eqref{eq:aux_induction_step}. 
\end{theorem}

\begin{proof}
By surveying the requirements for the induction step \eqref{eq:aux_induction_step} and comparing them with the results of \propref{prop:aux_first_step}, it remains to show the existence of a continuous linear map
\[
\bP_{\alpha}: \tGamma(\bbE_{\alpha+1}) \to \tGamma(\bbE_{\alpha})
\]
which is an element of $\OP(-m_{\alpha},0)$ such that
\[
\tbP_{\alpha} := \fbD_{\alpha} \bP_{\alpha} \in \OP(0,0)
\]
is the continuous projection onto $\scrR(\fbD_{\alpha}; \bB_{\alpha}) \subseteq \tGamma(\bbE_{\alpha+1})$ in the decomposition \eqref{eq:aux_decompositionAksmooth}.

We begin by noting that the topologically direct decomposition \eqref{eq:aux_decompositionAksmooth} already provides a projection onto this range,
\[
\tbP_{\alpha}: \tGamma(\bbE_{\alpha+1}) \to \tGamma(\bbE_{\alpha+1}).
\]
Since the decomposition is topologically direct, this map is continuous in the Fréchet topology, though it is not yet known to belong to the calculus. By \propref{prop:aux_first_step}, however, this projection does extend to the $L^{2}$-orthogonal projection into $\scrR^{0,2}(\fbD_{\alpha};\bB_{\alpha})$: 
\beq 
\tbP_{\alpha}: L^{2}(\bbE_{\alpha+1}) \to L^{2}(\bbE_{\alpha+1}).
\label{eq:tbP_aux_proof} 
\eeq 

With this in mind, using the containment relations \eqref{eq:containments_induction} from the induction hypothesis and the decompositions \eqref{eq:N0spdecomp} from the previous level, we find that for every $s\geq 0$
\[
\begin{aligned}
\scrR^{s,2}(\fbD_{\alpha}; \bB_{\alpha}) 
= \fbD_{\alpha}(\scrN^{s+m_{\alpha},2}_{\bot}(\fbD_{\alpha-1}^*) \cap \ker \bB_{\alpha}).
\end{aligned}
\]

Together with the estimate \eqref{eq:estimateNsp} at each Sobolev level, this shows that $\fbD_{\alpha}$ restricts to a bijection onto the above closed subspaces. By the open mapping theorem, we obtain isomorphisms of Banach spaces with continuous inverses:
\[
\begin{aligned}
(\fbD_{\alpha})^{-1}: \scrR^{s,2}(\fbD_{\alpha}; \bB_{\alpha}) 
\to \scrN^{s+m_{\alpha},2}_{\bot}(\fbD_{\alpha-1}^*) \cap \ker\bB_{\alpha}.
\end{aligned}
\]
By the graded Fréchet structure and the validity of this statement for every $s\geq 0$, we therefore obtain isomorphisms of Fréchet spaces
\[
\begin{aligned}
(\fbD_{\alpha})^{-1}: \scrR(\fbD_{\alpha}; \bB_{\alpha}) 
\to \scrN_{\bot}(\fbD_{\alpha-1}^*) \cap \ker\bB_{\alpha}.
\end{aligned}
\]
This allows us to define a continuous linear map $\bP_{\alpha}:\tGamma(\bbE_{\alpha+1})\rightarrow \tGamma(\bbE_{\alpha})$ by
\beq 
\bP_{\alpha} := (\fbD_{\alpha})^{-1} \tbP_{\alpha}.
\label{eq:bP_def_proof}
\eeq 
This is well defined because the range of $\tbP_{\alpha}$ lies in the domain of $(\fbD_{\alpha})^{-1}$ by \lemref{lem:density_Nsp}. As a composition of continuous maps, $\bP_{\alpha}$ is continuous, and the relation $\tbP_{\alpha} = \fbD_{\alpha} \bP_{\alpha}$ holds by construction.

In particular, from
\[
\begin{aligned}
(\fbD_{\alpha})^{-1}: \scrR_{2}^{0}(\fbD_{\alpha}; \bB_{\alpha}) 
\to \scrN^{m_{\alpha},2}_{\bot}(\fbD_{\alpha-1}^*) \cap \ker\bB_{\alpha},
\end{aligned}
\]
and \eqref{eq:tbP_aux_proof} combined with \eqref{eq:bP_def_proof}, it follows by composition that
\[
\bP_{\alpha}: L^{2}(\bbE_{\alpha+1}) \to W^{m_{\alpha},2}(\bbE_{\alpha})
\]
is continuous. Thus, if we prove that $\bP_{\alpha}$ belongs to the calculus, it must lie in $\OP(-m_{\alpha},0)$ by the restrictive nature of continuous mapping properties in the calculus (\cite[Prop.~2.5]{KL23}). We therefore proceed to show that $\bP_{\alpha}$ belongs to the calculus.

By the decomposition \eqref{eq:aux_decompositionAksmooth} and the fact that $\tbP_{\alpha} = \fbD_{\alpha} \bP_{\alpha}$ is the projection onto the corresponding range, we obtain
\[
\fbD_{\alpha}^* \fbD_{\alpha} \bP_{\alpha} = \fbD_{\alpha}^*.
\]

On the other hand, since by construction $\bP_{\alpha}$ takes its values in $\scrN_{\bot}(\fbD_{\alpha-1}^*) \cap \ker \bB_{\alpha}$, 
it follows that
\[
\tbP_{\alpha-1} \bP_{\alpha} = 0, 
\qquad 
\frakI_{\alpha} \bP_{\alpha} = 0, 
\qquad 
\bB_{\alpha} \bP_{\alpha} = 0.
\]
Summarizing,
\beq
\begin{aligned}
(\fbD_{\alpha}^* \fbD_{\alpha} \oplus \bB_{\alpha} \oplus \tbP_{\alpha-1} \oplus \frakI_{\alpha}) \bP_{\alpha}
= \fbD_{\alpha}^* \oplus 0 \oplus 0 \oplus 0.
\end{aligned}
\label{eq:defining_relationGa}
\eeq

By \propref{prop:new_estimate2}, the system
\[
\begin{aligned}
\fbD_{\alpha}^* \fbD_{\alpha} \oplus \bB_{\alpha} \oplus \tbP_{\alpha-1} \oplus \frakI_{\alpha}
\end{aligned}
\]
is overdetermined elliptic and injective. Hence its associated left inverse, denoted by $\frakS$, yields
\beq
\begin{aligned}
\bP_{\alpha} = \frakS(\fbD_{\alpha}^* \oplus 0 \oplus 0 \oplus 0),
\end{aligned}
\label{eq:G_def}
\eeq
which proves that $\bP_{\alpha}$ belongs to the calculus, as it is obtained by composition of operators in the calculus. Consequently, by composition, $\tbP_{\alpha} = \fbD_{\alpha} \bP_{\alpha}$ also belongs to the calculus, and is an element of $\OP(0,0)$ as it is $L^{2}\rightarrow L^2$ continuous due to the continuous mapping property \eqref{eq:tbP_aux_proof}.
\end{proof}
The decompositions \eqref{eq:aux_induction_Refined} required in the induction step will then follow by restriction, once we prove that $\bB_{\alpha+1} \fbD_{\alpha} = 0$ on $\ker \bB_{\alpha}$ in the next section.
\subsection{Stage 5: Completion of the induction step} 
By surveying what has been proven thus far, we see that the induction step is completed by establishing the existence and uniqueness of $\fbD_{\alpha+1}$ satisfying the requirements in \eqref{eq:contain_induction_step}, \eqref{eq:contain_induction_stepM}, \eqref{eq:relation_induction_step}, and \eqref{eq:correction_induction_step}. 

We begin this stage by using the auxiliary decompositions already established to define the continuous linear map of Fréchet spaces by setting
\beq
\fbD_{\alpha+1} : \tGamma(\bbE_{\alpha+1}) \to \tGamma(\bbE_{\alpha+2}), \qquad \fbD_{\alpha+1} := \bD_{\alpha+1} - \bD_{\alpha+1} \tbP_{\alpha}.
\label{eq:D_def_induction_step}
\eeq
We show this construction satisfies the required properties in \thmref{thm:lifted_complex}: 
\begin{proposition}
\label{prop:extra_boundary_condition} 
One has the relation
\[
\bB_{\alpha+1}(\scrR(\fbD_{\alpha}; \bB_{\alpha})) = 0,
\]
and the definition \eqref{eq:D_def_induction_step} yields the relations \eqref{eq:contain_induction_step}, \eqref{eq:contain_induction_stepM}, \eqref{eq:relation_induction_step} required in the induction step. It is the unique continuous map of Fréchet spaces with these properties.
\end{proposition}
\begin{proof} 
By the induction hypothesis, the refined decomposition \eqref{eq:aux_induction_Refined} holds for $\alpha - 1$. Hence, since $\fbD_{\alpha} = 0$ on $\scrR(\fbD_{\alpha-1}; \bB_{\alpha-1})$, we obtain
\[
\scrR(\fbD_{\alpha}; \bB_{\alpha}) 
= \fbD_{\alpha}(\scrN(\fbD_{\alpha-1}^*) \cap \ker \bB_{\alpha}).
\]
Moreover, by the induction hypothesis, $\fbD_{\alpha} = \bD_{\alpha}$ on $\scrN(\fbD_{\alpha-1}^*)$, and in particular on $\scrN(\fbD_{\alpha-1}^*) \cap \ker \bB_{\alpha}$. Thus,
\[
\scrR(\fbD_{\alpha}; \bB_{\alpha}) 
= \bD_{\alpha}(\scrN(\fbD_{\alpha-1}^*) \cap \ker\bB_{\alpha}).
\]
Since $\bB_{\alpha+1} \bD_{\alpha} = 0$ on $\ker \bB_{\alpha}$ by the order reduction property in the definition of an elliptic pre-complex (\defref{def:elliptic_pre_complex}), the first claim follows.

For the second part of the claim, we note that by construction $\tbP_{\alpha} = \id$ on $\scrR(\fbD_{\alpha}; \bB_{\alpha})$. Hence, on this space, $\fbD_{\alpha+1} = 0$ by its definition in \eqref{eq:D_def_induction_step}, so \eqref{eq:contain_induction_step} holds. Since $\tbP_{\alpha}$ vanishes on $\scrN(\fbD_{\alpha}^*)$, the relation in \eqref{eq:relation_induction_step} also holds, as implied by the auxiliary decomposition \eqref{eq:aux_decompositionAksmooth}. By construction, it is the unique continuous map with these properties, since the decomposition on which it is defined is topologically direct. 
\end{proof}

It remains to verify the requirements in \eqref{eq:correction_induction_step} and that $\bC_{\alpha+1} \in \OP(0,0)$. The definition of $\fbD_{\alpha+1}$ in \eqref{eq:D_def_induction_step} and the representation $\tbP_{\alpha} = \fbD_{\alpha} \bP_{\alpha}$ yields
\[
\begin{aligned} 
\bC_{\alpha+1} &= \fbD_{\alpha+1} - \bD_{\alpha+1} = -\bD_{\alpha+1} \fbD_{\alpha} \bP_{\alpha} = -\bD_{\alpha+1} \bD_{\alpha} \bP_{\alpha},
\end{aligned} 
\]
where we used our construction of $\bP_{\alpha}$ in \thmref{thm:aux_induction_step}; it takes values in $\scrN(\fbD_{\alpha-1}^*) \cap \ker \bB_{\alpha}$, where $\fbD_{\alpha} = \bD_{\alpha}$ by the induction hypothesis and the order-reduction properties in \defref{def:elliptic_pre_complex}. The fact that $\bC_{\alpha+1}\in\OP(0,0)$ then follows from the composition rules and the fact that $\bP_{\alpha}\in\OP(-m_{\alpha},0)$, together with the order reduction property stated in \defref{def:elliptic_pre_complex}:
\[
\ord(A_{\alpha+1}A_{\alpha})\leq\ord(A_{\alpha})= m_{\alpha}
\]

%
%
%
\subsection{The Hodge decompositions}
\label{app:Hodge_proof}
At this stage, auxiliary decompositions have been established for all levels $\alpha\leq \alpha_0$ of the elliptic pre-complex, and the relevant systems $\fbD_{\alpha}$ have been defined with the properties stated in the induction hypothesis in \secref{sec:stage1}.

Using this collective structure, we now prove \thmref{thm:Hodge_decomposition}, in the notation presented in \eqref{eq:Hodge_smooth_proof}, with the exception of the refined identities for the cohomologies \eqref{eq:cohomology_expression}, which—due to their significance and the distinct method of proof—are deferred to the final subsection.

By comparing the required decompositions in \eqref{eq:Hodge_smooth_proof} with the refined auxiliary decompositions \eqref{eq:N0spdecomp} already established for all levels through the induction steps, we find that it suffices to establish the following:
\begin{proposition}
\label{prop:Hodge_proof}
For all $\alpha< \alpha_0$, the following holds:
\[
\begin{aligned}
\scrN_{\bot}(\fbD^*_{\alpha})=\scrR(\fbD^*_{\alpha+1}).
\end{aligned}
\]
\end{proposition}
We prove this in several stages, following essentially the same lines of the proof in \cite[Sec.~4.3]{KL23}.
\begin{proposition}
\label{prop:correction_dual_containment}
For all $\alpha< \alpha_0$, the following holds:
\[
\begin{aligned}
\scrR(\fbD^*_{\alpha+1})\subseteq\scrN_{\bot}(\fbD^*_{\alpha}).
\end{aligned}
\]
\end{proposition}
\begin{proof}
This is obtained directly by dualizing the relation \eqref{eq:relation_induction_step} with respect to the $L^2$ inner product together with the relation $\scrR(\bD_{\alpha};\bB_{\alpha})\subset\scrN(\fbD_{\alpha+1},\bB_{\alpha+1})$ and $\scrR(\fbD^*_{\alpha+1})\,\bot\,\module^{\alpha+1}_{\D}$.  
\end{proof}
For the next proposition, recall that by the assumptions on an adapted Green operator, $\fbD^*_{\alpha+1}(\fbD_{\alpha+1}$ has order $2m_{\alpha+1}$ and class $m_{\alpha+1}$ due to the compostion rules of the calculus (cf.~\cite[Thm.~2.3]{KL23}). 
\begin{proposition}
\label{prop:correction_final_dual_new}
For all $\alpha< \alpha_0$, the subspace $\scrR^{s,2}(\fbD^*_{\alpha+1})\subseteq W^{s,2}(\bbE_{\alpha+1})$ is closed for every $s\geq 0$, and the following identity holds:
\beq 
\begin{aligned}
\scrR^{s,2}(\fbD^*_{\alpha+1}) 
= \scrR^{0,2}(\fbD^*_{\alpha+1}) 
\cap W^{s,2}(\bbE_{\alpha+1}).
\end{aligned}
\label{eq:s_identity_proof_adjoint} 
\eeq 
\end{proposition}
\begin{proof}
Consider the subspace of $W^{s,2}(\bbE_{\alpha+1})$
\beq
\begin{aligned}
\{\fbD_{\alpha+1}^*\theta : \theta \in W^{s+m_{\alpha+1},2}(\bbE_{\alpha+2})\}.
\end{aligned}
\label{eq:Dstar_range_proof}
\eeq
Iterating the decompositions in \eqref{eq:aux_induction_Refined} and \eqref{eq:N0spdecomp} for $\alpha+1$ and $\alpha$ (this is why we require $\alpha<\alpha_0$, to ensure that these decompositions are available), and using the defining relations of $\fbD_{\alpha}$ together with the dualized versions for $\fbD_{\alpha}^*$, we find that this space is equal to (see also \cite[Prop.~4.12]{KL23})
\[
\begin{aligned}
\{\fbD_{\alpha+1}^*\fbD_{\alpha+1}\psi : \psi \in \scrN^{s+2m_{\alpha+1},2}_{\bot}(\fbD^*_{\alpha})\cap\ker\bB_{\alpha+1}\}.
\end{aligned}
\]
The estimate in \eqref{eq:D^*D_estimate} then apply to the potential $\psi$ in the defining relation for this space, yielding the estimate
\[
\|\psi\|_{s+2m_{\alpha+1}} \lesssim \|\fbD_{\alpha+1}^*\fbD_{\alpha+1}\psi\|_{s}.
\]
By a standard argument (e.g.\ \cite[p.~583]{Tay11a}), this implies that the given subspace closed. Retracing our steps, we conclude that the subspace in \eqref{eq:Dstar_range_proof} is also closed, and therefore the subspace in the claim is closed as well. \eqref{eq:s_identity_proof_adjoint} then follows directly from the estimate above together with \propref{prop:AprioriAstarAk}.
\end{proof}
\begin{PROOF}{\propref{prop:Hodge_proof}}
Applying \eqref{eq:L2_decom} to the adapted Green operator $\fbD_{\alpha}^*$, one obtains the orthogonal $L^{2}$-decomposition:
\[
\begin{aligned}
L^2(\bbE_{\alpha+1}) = \scrR^{0,2}(\fbD_{\alpha+1}^*) \oplus \scrN^{0,2}(\fbD_{\alpha+1},\bB_{\alpha+1}).
\end{aligned}
\]
On the other hand, the $L^2$ version of the decompositions in \eqref{eq:N0spdecomp} reads
\[
\begin{aligned}
L^2(\bbE_{\alpha+1}) = \scrR^{0,2}(\fbD_{\alpha}; \bB_{\alpha})\oplus \scrN^{0,2}_{\bot}(\fbD_{\alpha}^*) \oplus \module_{\D}^{\alpha+1}.
\end{aligned}
\]
Both decompositions are $L^2$-orthogonal. Using the continuity of the orthogonal projection $\frakI_{\alpha+1}: L^{2}(\bbE_{\alpha+1}) \to \module^{\alpha+1}_{\DD}$, we observe:
\[
\begin{aligned}
\scrN^{0,2}(\fbD_{\alpha+1},\bB_{\alpha+1}) \cap \scrN^{0,2}_{\bot}(\fbD_{\alpha}^*) = \module_{\D}^{\alpha+1} \cap \scrN^{0,2}_{\bot}(\fbD_{\alpha}^*) = \{0\}.
\end{aligned}
\]
Comparing the decompositions, we conclude:
\[
\begin{aligned}
\scrN^{0,2}_{\bot}(\fbD_{\alpha}^*) \subseteq \scrR^{0,2}(\fbD_{\alpha+1}^*).
\end{aligned}
\]

Combining this with \propref{prop:correction_final_dual_new}, we obtain the required equality. Intersecting both sides with $W^{s,2}(\bbE_{\alpha+1})$ and applying the second clause of \propref{prop:correction_final_dual_new}, we then have:
\[
\begin{aligned}
\scrR^{s,2}(\fbD_{\alpha+1}^*) = \scrN^{s,2}_{\bot}(\fbD_{\alpha}^*).
\end{aligned}
\]
Since this holds for every $s\geq 0$, the smooth version also follows. 
\end{PROOF}
\subsection{Disrupted case} 
\label{app:disrupted_proof}
We now turn to the disrupted case, namely the content of \propref{prop:disrupted_valid}. We show that, with an additional argument, the result there follows from the analysis in the standard non-disrupted case.

The first step is to observe that, by definition (\defref{def:disrupted_pre_complex}), a disrupted $\alpha_0$-elliptic pre-complex is an $\tilde{\alpha}_0$-elliptic pre-complex with
\[
\tilde{\alpha}_0=\alpha_0-1.
\]
Thus, all relevant lifted systems up to $\fbD_{\tilde{\alpha}_0}$ are constructed as before and satisfy the same properties as in the non-disrupted case. In particular, $\fbD_{\tilde{\alpha}_0-1}$ induces an auxiliary decomposition:
\beq
\begin{aligned}
\tGamma(\bbE_{\tilde{\alpha}_0}) 
= \scrR(\fbD_{\tilde{\alpha}_0-1}; \bB_{\tilde{\alpha}_0-1}) \oplus \scrN(\fbD^*_{\tilde{\alpha}_0-1}).
\end{aligned} 
\label{eq:aux_disrupted_N-1}
\eeq 

Hence, to prove \propref{prop:disrupted_valid}, it remains to show that $\fbD_{\tilde{\alpha}_0}$ induces an auxiliary decomposition (since $\fbD_{\tilde{\alpha}_0+1} = 0$), and that this decomposition, together with \eqref{eq:aux_disrupted_N-1}, refines into Hodge decompositions, albeit without the finite dimensionality of $\module^{\tilde{\alpha}_0}_{\D}$.

By the defining properties of disrupted $\alpha_0$-elliptic pre-complexes and the fact that $\bD_{\tilde{\alpha}_0+1} = 0$, a comparison with the overdetermined ellipticity conditions in \defref{def:elliptic_pre_complex} for $\alpha = \tilde{\alpha}_0 + 1$, together with the fact that $\fbD_{\tilde{\alpha}_0}^*-\bD_{\tilde{\alpha}_0}$ is a lower-order term, shows that the system
\beq
\begin{aligned}
\fbD_{\tilde{\alpha}_0}^*
\end{aligned}
\label{eq:OD_disrupted1}
\eeq
is overdetermined elliptic, without an additional boundary term. Composing this with the overdetermined ellipticity at $\alpha = \tilde{\alpha}_0$ and using the order-reduction property to as usual disregard compact perturbations (again, under which overdetermined ellipticity is invariant, by its equivalence with semi-Fredholmness discussed in \secref{sec:OD}), we find that the following system is also overdetermined elliptic:
\beq
\begin{aligned}
\fbD_{\tilde{\alpha}_0} \fbD_{\tilde{\alpha}_0}^* \oplus \bB_{\tilde{\alpha}_0} \fbD_{\tilde{\alpha}_0}^*.
\end{aligned}
\label{eq:OD_disrupted2}
\eeq
Hence, the sequence
\[
\tilde{\bD}_{-1} = 0, \qquad \tilde{\bD}_{1} = \fbD_{\tilde{\alpha}_0}^*, \qquad \tilde{\bD}_2 = 0, \qquad \tilde{\bD}_{3} = 0, \cdots
\]
is itself an $\tilde{\tilde{\alpha}}_0$-elliptic pre-complex, with $\tilde{\tilde{\alpha}}_0 = 1$, albeit based on {Neumann conditions}, as in \cite[Sec.~3.2]{KL23}. Thus, $\fbD_{\tilde{\alpha}_0}^*$ induces its own {Neumann} auxiliary decompositions, as defined in \cite[Def.~3.3]{KL23}:  
\beq 
\begin{aligned}
\tGamma(\bbE_{\tilde{\alpha}_0}) 
= \scrR(\fbD^*_{\tilde{\alpha}_0}) \oplus \scrN(\fbD_{\tilde{\alpha}_0}; \bB_{\tilde{\alpha}_0}).
\end{aligned}
\label{eq:aux_disrutped_adjoint_N}
\eeq 
\begin{proposition} 
The decomposition \eqref{eq:aux_disrupted_N-1} refines into a Hodge decomposition (with $\tilde{\alpha}_0=\alpha-1$):
\beq 
\begin{aligned}
\tGamma(\bbE_{\tilde{\alpha}_0}) 
= \scrR(\fbD_{\tilde{\alpha}_0-1}; \bB_{\tilde{\alpha}_0-1}) 
\oplus \scrR(\fbD_{\tilde{\alpha}_0}^*) \oplus  \module^{\tilde{\alpha}_0}_{\DD}.
\end{aligned}
\label{eq:Hodge_like_disrupted_1} 
\eeq
with $\module^{\tilde{\alpha}_0}_{\D}$ closed but not necessarily finite dimensional.
\end{proposition}
\begin{proof} 
On the one hand, we have the auxiliary decompositions of both $\fbD_{\tilde{\alpha}_0}^*$ in \eqref{eq:aux_disrutped_adjoint_N} and $\fbD_{\tilde{\alpha}_0-1}$ in \eqref{eq:aux_disrupted_N-1}, which are $L^{2}$-orthogonal; on the other hand, we have the relations obtained from the construction of the lifted complex from the main proof 
\beq 
\begin{aligned}
\scrR(\fbD^*_{\tilde{\alpha}_0}) \subseteq \scrN(\fbD_{\tilde{\alpha}_0-1}^*), \qquad \scrR(\fbD_{\tilde{\alpha}_0-1}; \bB_{\tilde{\alpha}_0-1}) \subseteq \scrN(\fbD_{\tilde{\alpha}_0}, \bB_{\tilde{\alpha}_0}).
\end{aligned}
\label{eq:disrupted_relations_proof} 
\eeq 
Comparing these, we obtain the $L^{2}$-orthogonal decompositions
\[
\begin{aligned}
\tGamma(\bbE_{\tilde{\alpha}_0}) 
= \scrR(\fbD_{\tilde{\alpha}_0-1}; \bB_{\tilde{\alpha}_0-1}) 
\oplus \scrR(\fbD_{\tilde{\alpha}_0}^*) \oplus (\scrN(\fbD_{\tilde{\alpha}_0}, \bB_{\tilde{\alpha}_0}) \cap \scrN(\fbD^*_{\tilde{\alpha}_0-1})).
\end{aligned}
\]
Moreover, since by the defining relation in \thmref{thm:lifted_complex} we have $\fbD_{\alpha} = \bD_{\alpha}$ on $\scrN(\fbD_{\alpha}^*)$, it follows that in the setting of \thmref{thm:Hodge_decomposition} we have:
\[
\begin{aligned}
\scrN(\fbD_{\tilde{\alpha}_0}, \bB_{\tilde{\alpha}_0}) \cap \scrN(\fbD^*_{\tilde{\alpha}_0-1}) = \module^{\tilde{\alpha}_0}_{\DD}.
\end{aligned}
\]
By comparing, these provide the desired Hodge decompositions. 
\end{proof} 
\begin{proposition}
$\fbD_{\tilde{\alpha}_0}$ induces an auxiliary decomposition, which refines into a Hodge decomposition as specified by \propref{prop:disrupted_valid}.
\end{proposition}
\begin{proof}
By iterating the relations in \eqref{eq:disrupted_relations_proof} upon the decompositions in \eqref{eq:Hodge_like_disrupted_1}: 
\[
\begin{aligned} 
\scrR(\fbD_{\tilde{\alpha}_0};\bB_{\tilde{\alpha}_0}) 
= \BRK{\fbD_{\tilde{\alpha}_0}\fbD_{\tilde{\alpha}_0}^*\psi ~:~ \bB_{\tilde{\alpha}_0}\fbD_{\tilde{\alpha}_0}^*\psi=0}.
\end{aligned} 
\]

By the overdetermined ellipticities listed in \eqref{eq:OD_disrupted1}, the space on the right is closed due to the associated a~priori estimates, and so is its Sobolev versions. In particular, the space $\scrR^{0,2}(\fbD_{\tilde{\alpha}_0};\bB_{\tilde{\alpha}_0})$ is closed, and the $L^{2}$-orthogonal decompositions from \eqref{eq:L2_decom} become: 
\[
\begin{aligned}
L^{2}(\bbE_{\tilde{\alpha}_0+1})
= \scrR^{0,2}(\fbD_{\tilde{\alpha}_0}; \bB_{\tilde{\alpha}_0}) \oplus \scrN^{0,2}(\fbD_{\tilde{\alpha}_0}^*).
\end{aligned}
\]
However, by the overdetermined ellipticity in \eqref{eq:OD_disrupted2}, the kernel on the right-hand side is finite dimensional and consists of smooth sections, and the $L^{2}$-orthogonal projection onto it belongs to the calculus, lying in $\OP(-\infty,-\infty)$. In fact, since $\fbD_{\tilde{\alpha}_0+1} = 0$ (and likewise $\bB_{\tilde{\alpha}_0+1}=0$), we directly find, in the setting of \thmref{thm:Hodge_decomposition} and \propref{prop:disrupted_valid}, that
\[
\begin{aligned}
\scrN^{0,2}(\fbD_{\tilde{\alpha}_0}^*) 
= \scrN(\fbD_{\tilde{\alpha}_0}^*) = \module^{\tilde{\alpha}_0+1}_{\DD}.
\end{aligned}
\]

Hence, on the smooth level we have:
\[
\begin{aligned}
\tGamma(\bbE_{\tilde{\alpha}_0+1})
= \scrR(\fbD_{\tilde{\alpha}_0};\bB_{\tilde{\alpha}_0}) \oplus \module^{\tilde{\alpha}_0+1}_{\DD}.
\end{aligned}
\]
This gives both the auxiliary decompositions and the Hodge decompositions induced by $\fbD_{\tilde{\alpha}_0}$, as required.
\end{proof}

\bibliographystyle{amsalpha}

\providecommand{\bysame}{\leavevmode\hbox to3em{\hrulefill}\thinspace}
\providecommand{\MR}{\relax\ifhmode\unskip\space\fi MR }
\providecommand{\MRhref}[2]{%
\href{http://www.ams.org/mathscinet-getitem?mr=#1}{#2}
}
\providecommand{\href}[2]{#2}

\end{document}